\documentclass[11pt]{article}  %[12pt]{article}
\usepackage[utf8]{inputenc} %[latin1]{inputenc}
\usepackage[english]{babel}
\usepackage[T1]{fontenc}
\usepackage{lmodern}
\usepackage{graphicx, xcolor, mathrsfs}
\usepackage{amsfonts}
\usepackage{stmaryrd}
\usepackage{blkarray, bigstrut}
\usepackage{mathtools}
\usepackage[nice]{nicefrac}
\usepackage{amsmath,amsthm,amssymb}
\usepackage[mathlines]{lineno}
\usepackage{mathabx}
\usepackage[normalem]{ulem} % you can use \sout to strike out text
\usepackage{accents}

\usepackage{enumerate}
\usepackage{hyperref} %[colorlinks=true,citecolor=black,linkcolor=black,urlcolor=blue]{hyperref}

\definecolor{MyDarkblue}{rgb}{0,0.08,0.50}
\definecolor{Brickred}{rgb}{0.65,0.08,0}

\usepackage{hyperref,upref}
\hypersetup{
	colorlinks=true,       % false: boxed links; true: colored links
	linkcolor=MyDarkblue,          % color of internal links
	citecolor=Brickred,        % color of links to bibliography
	filecolor=red,      % color of file links
	urlcolor=cyan           % color of external links
}

\usepackage[top=.9in, bottom=.9in, left=.5in , right=.5in]{geometry}
\usepackage[skip=5pt]{caption}
\usepackage{xifthen} 
\usepackage{wrapfig}
\usepackage[numbers, square]{natbib}
\usepackage{changes}
\usepackage{comment}
\usepackage{float}
\usepackage{ifthen}
\mathtoolsset{showonlyrefs}
\usepackage{tikz}                      
\usetikzlibrary{arrows}
\usetikzlibrary{graphs,graphs.standard}
\usetikzlibrary{positioning,arrows.meta}
\usetikzlibrary{shapes.multipart}
\usetikzlibrary{automata}
\usepackage{adjustbox}

\usepackage{diagbox}

\newtheorem{thm}{Theorem}[section]

\newtheorem{open}[thm]{Open problem}

\newtheorem{lemma}[thm]{Lemma}
\newtheorem{clm}{Claim}[thm]
\newtheorem{prop}[thm]{Proposition}
\newtheorem{cor}[thm]{Corollary}
\theoremstyle{definition}

\newtheorem{definition}[thm]{Definition}
\newtheorem{assumption}[thm]{Assumption}
\newtheorem{remark}[thm]{Remark}

\def\Var{\mathop{\rm Var}\nolimits}

\newcommand{\R}{\mathbb R}
\newcommand{\N}{\mathbb N}
\newcommand{\E}[2][]{\ensuremath{\mathbb{E}_{#1}\left[#2 \right]}}
\newcommand{\Prob}[2][]{\ensuremath{\mathbb{P}_{#1} \left(#2 \right)}}

\newcommand{\eps}{\varepsilon}

\newcommand{\cL}{\mathcal L}

\newcommand{\cM}{\mathcal M}

\newcommand{\cT}{\mathcal{T}}

\newcommand{\ind}[1]{\mathbf 1_{#1}}
\newcommand{\dd}{\mathrm d}
\newcommand{\sss}{\scriptscriptstyle}
\DeclareMathOperator{\exponentialrv}{Exp}
\newcommand{\Exp}[1]{\exponentialrv\left( #1 \right)}
\DeclareMathOperator{\Path}{path}

\DeclareMathOperator{\Deg}{deg}
\newcommand{\outdeg}[1]{\ensuremath{\Deg^{+}(#1)}}

\newcommand*{\be}{\begin{equation}}
	\newcommand*{\ee}{\end{equation}}
\newcommand*{\ba}{\begin{aligned}}
	\newcommand*{\ea}{\end{aligned}}
\renewcommand{\P}[1]{\mathbb P\left(#1\right)}

\newcommand{\e}{\mathrm e}
\newcommand{\cB}{\mathcal B}
\newcommand{\cO}{\mathcal O}
\newcommand{\cP}{\mathcal P}
\newcommand{\cU}{\mathcal U}
\newcommand{\wt}{\widetilde}

\DeclareMathOperator*{\argmin}{arg\,min}

\newcommand{\ensymboldremark}{\hfill$\blacktriangleleft$}

\newcommand{\invisible}[2]{%
	\ifthenelse{\isempty{#1}}
	{}% if #1 is empty
	{#2}% if #1 is not empty
}

\numberwithin{equation}{section}

\usepackage{enumitem}
\usepackage{bbm}

\newif\ifdraft
\drafttrue % or \draftfalse
\ifdraft
\else 
\fi

\title{On the structure of genealogical trees associated with explosive Crump-Mode-Jagers branching processes}

\author{T. Iyer\footnote{Weierstrass Institute for Applied Analysis and Stochastics, Mohrenstrasse 39, 10117 Berlin, Germany.} \, \& B. Lodewijks\footnote{University of Augsburg, Universit\"atsstra\ss e 14, 
		86159 Augsburg, Germany. }}
\date{November 27, 2023}

\begin{document}
	\maketitle
	\abstract{ We study the structure of genealogical trees associated with explosive Crump-Mode-Jagers branching processes (stopped at the explosion time), proving criteria for the associated tree to contain a node of infinite degree (a \emph{star}) or an infinite path. Next, we provide uniqueness criteria under which with probability $1$ there exists exactly one of a unique star or a unique infinite path. Under the latter uniqueness criteria we also provide an example where, with strictly positive probability less than $1$, there exists a unique star in the model. We thus illustrate that this probability is not restricted to being $0$ or $1$. Moreover, we provide structure theorems when there is a star, where we prove that certain trees appear as sub-trees in the tree infinitely often. We apply our results to a general discrete evolving tree model, named \emph{explosive recursive trees with fitness}. As a particular case, we study a family of \emph{super-linear preferential attachment models with fitness}. For these models, we derive phase transitions in the model parameters in three different examples, leading to either exactly one star with probability $1$ or one infinite path with probability $1$, with every node having finite degree. Furthermore, we highlight examples where sub-trees $T$ of \emph{arbitrary} size can appear infinitely often; behaviour that is markedly distinct from super-linear preferential attachment models studied in the literature so far.}
	
	\tableofcontents
	
	\section{Introduction}
	
	Given a population of an entity moving towards explosion, that is, the emergence of infinitely many individuals in finite time, what can be said about the genealogical tree associated with the population at the time of explosion? On the one hand, an infinite path in the tree may be interpreted as an infinite line of evolution, with infinitely many `variants' contributing to the explosion; on the other, a node of infinite degree (which we often call a \emph{star}) may be interpreted, informally, as the emergence of a `dominant variant'. This is the goal of the present investigation, where as a simplified model of an evolving population, we use \emph{Crump--Mode--Jagers} branching processes. 
	
	In a CMJ branching process (named after~\cite{crump-mode, jagers-origin}), an ancestral \emph{root} individual produces offspring according to a collection of points on the non-negative real line. Each individual `born' produces offspring according to an identically distributed collection of points, translated by their birth time (see Section~\ref{sec:notation-prelims} for a more formal description). One is generally interested in properties of the population as a function of time. Classical work from the 1970s and '80s related to this model generally deals with the \emph{Malthusian case}, which, informally, refers to the fact that the population grows exponentially in time. These include strong laws of large numbers for \emph{characteristics} associated with the process~\cite{nerman_81}, properties of birth times in the $k$th generation~\cite{kingman1975}, an $x\log{x}$ theorem~\cite{olofsson-x-log-x}, and numerous other results, for example~\cite{berndtsson-jagers, nerman-jagers-84, jagers-nerman89,jagers-89, jagers-nerman-96}; see also the classical books~\cite{arthreya-ney, jagers-book}. A number of more recent results are concerned with \emph{asymptotic fluctuations} associated with the process in the Malthusian case, see, for example, \cite{iksanov-kolesko-fluctuations, janson-cmj-fluctuations, iksanov-marynych-fluctuations, clt-cmj-kolesko}. Other results are motivated by applications of these processes, including M/G/1 queues~\cite{Grishechkin-92}, vaccination and epidemic modelling~\cite{ball-epidemics, ball-epidemic-2016, Levesque-21}, and numerous applications to random graphs, see Section~\ref{sec:rec-fit} below. 
	
	Far fewer results exist for CMJ branching processes when a Malthusian parameter does not exist. In a particular case of reinforced branching process, a \emph{condensation} phase transition can occur, where non-exponential growth occurs due to individuals having random weights that influence their offspring distribution. In this case, a `small' numbers of individuals with large weight produce larger and larger families, which in turn lead a rate of growth faster than exponential~\cite{dereich-mailler-morters}. In more extreme circumstances, individuals produce larger and larger families so quickly that the process explodes in finite time. General criteria for explosion have been provided in terms of the solution of a functional fixed point equation by Komj\'athy~\cite{Komjathy2016ExplosiveCB}, who also extended the necessary and sufficient criteria for explosion in branching random walks in~\cite{amini-et-al}. We refer the reader to \cite{Komjathy2016ExplosiveCB} for a more comprehensive overview of the literature related to explosion in branching processes. In~\cite{bertoin-local-expl}, the authors provide sufficient criteria for \emph{local explosion} in closely related growth-fragmentation processes. Meanwhile, more, sufficient criteria for explosion in CMJ processes are in preparation in~\cite{limiting-structure}. 
	
	\subsection{Random recursive trees with fitness} \label{sec:rec-fit}
	Aside from the applications outlined above, CMJ branching processes are often involved in the analysis of random graph models, more often, random trees. As far as the authors are aware, direct applications date back to Pittel~\cite{pittel}, in providing a new proof for the limiting behaviour of the heights of \emph{random recursive trees} and \emph{affine preferential attachment trees}, but the technique of using continuous-time embeddings to analyse discrete combinatorial processes is more classical, going back to works of Arthreya and Karlin~\cite{arthreya-karlin-split-times, arthreya-karlin-embedding-68}. 
	
	Later, works by, for example~\cite{rudas, bhamidi, holmgren-janson, Oliveira-spencer, pagerank} showed that CMJ branching processes can be applied to a large number of growing tree models. A natural framework of evolving trees, which corresponds to genealogical trees of CMJ branching processes and encompasses many existing models of recursive trees (which we refer to as random recursive trees with fitness~\cite{rec-trees-fit}), posits that nodes $v$ arrive one at a time, and are assigned a random i.i.d.\ \emph{weight} $W_{v}$ sampled from a measure $\mu$ on an arbitrary measure space $(S, \mathcal{S})$ (see Definition~\ref{def:wrt}). Newly arriving nodes then connect, with edges directed outwards from the target nodes, with probability proportional to a general, measurable fitness function $f: \mathbb{N}_{0} \times S \rightarrow [0,\infty)$ that incorporates information about the current out-degree of the target, and its weight. A natural quantity of interest in this model is the proportion of 
	nodes $\wt{N}_{k}$ (at the 
	$n$th time-step) having out-degree $k$. This model may be roughly classified according to the following conjectured phases~\cite{rec-trees-fit}:
	\begin{enumerate}
		\item The \emph{non-condensation phase:} There exists $\lambda > 0$ such that $\sum_{j=1}^{\infty} \mathbb{E} \left[\prod_{i=0}^{j-1} \frac{f(i,W)}{f(i,W) + \lambda}\right] = 1.$ In this case, if $p_{k}$ denotes the limit of $\wt{N}_{k}$, we have $\sum_{k=0}^{\infty} p_{k} = 1$. In other words, all of the mass of edges is distributed around nodes of \emph{microscopic} degrees. 
		\item\label{item:cond} The \emph{condensation phase} (see~\cite{borgs-chayes, dereich-mailler-morters, dereich-ortgiese, dereich-unfolding}):  We have $\sum_{j=1}^{\infty} \mathbb{E}\left[\prod_{i=0}^{j-1} \frac{f(i,W)}{f(i,W) + \lambda}\right] < 1$, for any $\lambda > 0$ such that the sum converges. In this case, $0 < \sum_{k=0}^{\infty} p_{k} < 1$, so that a positive fraction of `mass' is lost from the empirical measure to a sub-linear number of nodes of `large degrees'.  
		\item\label{item:expl} The \emph{extreme-condensation phase:} For any $\lambda > 0$ we have $\sum_{j=1}^{\infty} \mathbb{E}\left[\prod_{i=0}^{j-1} \frac{f(i,W)}{f(i,W) + \lambda}\right] = \infty.$ In this case, $\sum_{k=0}^{\infty} p_{k} = 0$, so that all of the `mass' of edges is concentrated in nodes of `large' degrees. 
	\end{enumerate}
	Part of the goal of this article is to investigate the behaviour of the third phase above, in the \emph{explosive} case that 
	\be 
	\sum_{i=0}^{\infty} \frac{1}{f(i,w)} < \infty,\qquad \text{for $\mu$-almost all $w \in S$.}
	\ee 
	Note that this implies that $\prod_{i=0}^\infty \frac{f(i,W)}{f(i,W)+\lambda}>0$ almost surely, so that the condition in Phase~\ref{item:expl} is satisfied. 
	
	In~\cite{Oliveira-spencer}, Oliveira and Spencer showed that, in the case $f(i,w) = i^{p}$, with $1 + \frac{1}{k} < p \leq  1 + \frac{1}{k-1}$ for some $k\in\N$ (equality only for $k\geq 2$), the infinite tree associated with this process is somehow extreme: it contains a single node of infinite degree, connected to infinitely many children with an associated sub-tree of size at most $k$, and only finitely many with an associated sub-tree of size $k+1$ or larger. This paper uses the fact that the associated CMJ branching process is \emph{explosive}, a technique also exploited in similar works related to `balls-in-bins' processes~\cite{exp-scaling}. Related work by Arthreya~\cite[Theorem~2.1]{arthreya-gen-weight-function} attempts to prove that, according to a certain summability condition, either every node in the infinite tree has finite degree, or with positive probability there exists a single node such that all but finitely many new nodes connect to this node. However, there is a mistake here, in that~\cite[Theorem~2.1b]{arthreya-gen-weight-function} should really state: the probability that there exists a single node such that all but finitely many new-coming nodes connect to this node is zero (indeed, note that~\cite[Corollary~2.2]{arthreya-gen-weight-function} directly contradicts~\cite{Oliveira-spencer}). Nevertheless, the associated summability condition and result in this paper is interesting, and motivates the question of whether there is a critical condition that guarantees the existence of a node of infinite degree in the infinite tree, or every node having finite degree, cf.\ Theorem~\ref{thm:star-path-rif}, below. 
	
	A related question is whether or not, in the associated \emph{recursive tree with fitness} model, the index associated with the node of maximal degree is fixed after some finite time, or changes infinitely often, that is, whether or not there is a \emph{persistent hub}. A unique node of infinite degree in the infinite tree associated with the model thus implies the existence of such a hub. In a slightly different model of evolving graphs, when $f(i,w) = g(i)$ with $g$ being a \emph{concave} sub-linear function, one of the results of Dereich and M\"orters~\cite{dereich-morters-persistence} shows that a persistent hub emerges if and only if $\sum_{i=1}^{\infty} g(i)^{-2} < \infty$. In the recursive tree model described above, Galashin~\cite{galashin2014existence} proved that, if $f(i,w) = g(i)$, with $g$ \emph{convex} and unbounded, a persistent hub always appears. This has been extended to a much wider range of functions $g$, independent of the weight $w$, by Banerjee and Bhamidi in~\cite{banerjee-bhamidi}. 
	
	When weights are added, however, in the sense that $f(i,w)$ may depend on $w$, a different picture emerges. Suppose that $w$ takes values in $[0, \infty)$.  In the case $f(i,w) = w (i+1)$, under a particular set of conditions leading to the \emph{condensation phase} (Item~\ref{item:cond} above), in~\cite{dereich-mailler-morters} the authors show that there is no persistent hub, and the size of the node of maximal degree grows sub-linearly in the size of the tree. In the case $f(i,w) = i+w$ or $f(i,w) = w$, when the weights $w$ are sampled according to certain classes of distributions, in~\cite{bas,Sen21} and~\cite{LodOrt20,Lod21}, respectively, the authors provide critical criteria depending on the parameters of the weight distribution, for the existence, or non-existence, of a persistent hub. A number of other particular models of so called \emph{preferential attachment with fitness} have been studied, see, for example, \cite{vdh-aging-mult-fitness-2017, jonathan-1, pref-att-neighbours}, and related works regarding local weak limits of preferential attachment type models~\cite{berger-noam-chayes-2014, lo2021weak, garavaglia2023universality}.
	
	\subsection{Overview of our contribution and structure}
	In this paper we provide general sufficient conditions for the genealogical tree associated with an explosive CMJ branching process to contain a node of infinite degree  or an infinite path at the explosion time (in Theorems~\ref{thm:star} and~\ref{thm:path}, respectively). When there is a node of infinite degree, we provide criteria for one to see a fixed tree as a sub-tree of a child of that node either infinitely often, or finitely often, in Theorem~\ref{thm:structure}. We also prove uniqueness criteria in Theorem~\ref{thm:uniqueness}, under which there almost surely exists a unique node of infinite degree or a unique infinite path. Under the conditions of the uniqueness theorem, we provide a counter-example in Theorem~\ref{thm:counter-example}, where the events that there is a node of infinite degree, or an infinite path, both have positive probability, less than $1$. Finally, in Theorems~\ref{thm:star-path-rif}, \ref{thm:sub-treecount}, and Corollary~\ref{cor:sumnufinmean} we apply our results to the recursive tree with fitness model and prove phase transitions in three particular models in Theorems~\ref{thrm:cmjexamples} and~\ref{thrm:sub-treecount}. We encourage the reader more interested in this discrete model to refer to these results first. 
	
	The question of whether the genealogical tree of an explosive CMJ branching process contains an infinite path or a node of infinite degree has not been investigated in this level of generality previously. Our techniques involve significant improvements of those of~\cite{Oliveira-spencer} (see also Sections~\ref{sec:techniques-discussion-1} and~\ref{sec:techniques-discussion-2}), and thus allow us to greatly extend the picture associated with the general recursive tree with fitness model. Intriguingly, our results show that when there is a unique node of infinite degree in the infinite tree associated with the model, in many particular cases there exist children of the node of infinite degree that have arbitrarily large, but finite, degree (or even an arbitrarily large, but finite, number of descendants). Previous comments in the literature seem to indicate that it was believed that when there is an infinite degree node, the degrees of all other nodes are bounded, see~\cite[Section~3]{banerjee-bhamidi}. Finally, we remark that the phase transitions related to the emergence of a node of infinite degree are reminiscent of a different notion of \emph{condensation} in conditioned Bienaym\'{e}-Galton-Watson trees~\cite{svante-condensation}.
	
	\subsubsection{Structure of the paper}
	The paper is structured in the following way. Below, in Section~\ref{sec:notation-prelims}, we introduce a formal description of the model and the notation we use in this paper. Section~\ref{sec:results} states the main results, which are most general and require certain assumptions on the inter-birth time distribution. Section~\ref{sec:examples} then discusses the particular example of exponentially distributed inter-birth times and how this relates to a family of discrete tree models coined \emph{recursive trees with fitness}. Here, we derive sufficient conditions such that the assumptions used for the main results are satisfied. Moreover, when considering certain sub-families of recursive trees with fitness, we prove more precise results in terms of phase diagrams for the existence of either unique infinite-degree nodes or unique infinite paths. As mentioned above, we encourage the reader more interested in results related to the discrete recursive tree model (which is also less abstract), to refer to the results of Section~\ref{sec:recursive} first, before reading the section below. 
	
	Aside for a few exceptions, Section~\ref{sec:proofmain} proves the main results of Section~\ref{sec:results}, Section~\ref{sec:applications} proves the most general results of Section~\ref{sec:examples}, and Section~\ref{sec:examplesproof} proves the particular examples of Section~\ref{sec:examples}. Finally, we consider a number of other models in~\hyperlink{sec:app}{Appendix}~\ref{sec:appadd} and~\ref{sec:app}, showing that the assumptions subject to which the main results hold are valid more broadly.
	
	\subsection{Notation and preliminaries} \label{sec:notation-prelims}
	
	In this paper, we consider properties of the genealogical trees associated with Crump-Mode-Jagers branching processes; a tree-valued stochastic process $(\mathscr{T}_{t})_{t \geq 0}$ which one may regard as the genealogical tree representing a population evolving over time. The goal then, is to define a state space of individuals, in this setting, the infinite \emph{Ulam-Harris} tree of potential individuals associated with a common ancestor, \emph{birth times} $\mathcal{B}(u)$ associated with individuals, which themselves are encoded by a random function $(X,W)$, and then define $\mathscr{T}_{t}$ as the set of individuals born up to time $t$. 
	Note that the notation we use is slightly different to the common notation regarding CMJ branching processes, see Remark~\ref{rem:notation-form} further below. 
	
	First, we generally use $\mathbb{N} := \{1, 2, \ldots\}$, $\mathbb{N}_{0} := \mathbb{N} \cup \{0\}$ and for $n \in \mathbb{N}$ we let $[n]:= \{1, \ldots, n\}$. We consider \emph{individuals} in the process as being labelled by elements of the infinite \emph{Ulam-Harris} tree $\cU_\infty : = \bigcup_{n \geq 0} \N^{n}$; where $\N^{0} := \{\varnothing\}$ contains a single element $\varnothing$ which we call  the \emph{root}. 
	We denote elements $u \in \cU_{\infty}$ as a tuple, so that, if $u = (u_{1}, \ldots, u_{k}) \in \N^{k}, k \geq 1$, we write $u = u_{1} \cdots u_{k}$. An individual $u = u_1u_2\cdots u_k$ is to be interpreted recursively as the $u_k$th child of the individual $u_1 \cdots u_{k-1}$; for example, $1, 2, \ldots$ represent the offspring of $\varnothing$. Suppose that $(\Omega, \Sigma, \mathbb{P})$ is a complete probability space and $(S, \mathcal{S})$ is a measure space. We also equip $\mathcal{U}_{\infty}$ with the sigma algebra generated by singleton sets. Then, we fix a random mappings $X: \Omega \times \cU_{\infty} \rightarrow [0, \infty]$, $W: \Omega \times \cU_{\infty} \rightarrow S$, and define $(X, W): \Omega \times \cU_{\infty} \rightarrow [0, \infty] \times S$, so that $u \mapsto (X(u), W_{u})$. 
	In general, for $u \in \cU_{\infty}$ and $j\in \N$, one interprets $W_{u}$ as a `weight' associated with $u$, and $X(uj)$ the waiting time between the birth of the $(j-1)$th and $j$th \emph{child} of $u$.
	
	We introduce some notation related to elements $u \in \cU_{\infty}$: we use $|\cdot|$ to measure the \emph{length} of a tuple $u$, so that, if $u = \varnothing$ we set $|u| = 0$, whilst if $u = u_{1} \cdots u_{k}$ then $|u| = k$. If, for some $x \in \cU_{\infty}$, we have $x = u v$, we say $u$ is a \emph{ancestor} of $x$. We introduce a notation to refer to ancestors: given $\ell \leq |u|$, we set $u_{|_\ell} := u_{1} \cdots u_{\ell}$.  It will be helpful to equip $\cU_{\infty}$ with the lexicographic total order $\leq_{L}$: given elements $u, v$ we say $u \leq_{L} v$ if either $u$ is a ancestor of $v$, or $u_{\ell} < v_{\ell}$ where $\ell = \min \left\{i \in \mathbb{N}: u_{i} \neq v_{i} \right\}$. We say a subset $T \subset \mathcal{U}_{\infty}$ is a \emph{tree} if, given that $u \in T$, we also have $u_{|_{\ell}} \in T$, for each $\ell \leq |u|$. Note that any such trees can be viewed as graphs in the natural way, connecting nodes to their children. 
	
	Now, we use the values of $X$ to associate \emph{birth times} $\mathcal{B}(u)$ to individuals $u \in \cU_{\infty}$. In particular, we define $\mathcal{B}: \Omega \times \cU_{\infty} \rightarrow [0, \infty]$ recursively as follows:  
	\[\mathcal{B}(\varnothing) : = 0 \quad \text{and for $u \in \cU_{\infty}, i \in \mathbb{N}$,} \quad \mathcal{B}(ui) := \mathcal{B}(u) + \sum_{j=1}^{i} X(uj).\]
	Consequentially, a value of $X(ui) = \infty$ indicates that the individual $u$ has stopped producing offspring, and does not produce $i$ children or more. 
	
	Finally, we set $\mathscr{T}_{t} = \{x \in \cU_\infty: \cB(x) \leq t\}$ and identify for each $t \geq 0$, $\mathscr{T}_{t}$ as the \emph{genealogical tree} of individuals with birth time at most $t$. Again we emphasise that one may think of this, intuitively, as the set of all individuals, originating from a common ancestor, that have been born by time $t$. More formally, we identify the process with $(\widetilde{\mathscr{T}}_{t})_{t \geq 0}$; a measurable mapping 
	$\Omega \times [0, \infty) \times \mathcal{U}_{\infty} \rightarrow [0, \infty]$. Then, if $u \in \mathscr{T}_{t}$, we set $\widetilde{\mathscr{T}}_{t}(u) = \mathcal{B}(u)$, and otherwise, set $\widetilde{\mathscr{T}}_{t}(u) = \infty$. In addition, we set $\mathcal{W}_{t} = \{(x, \mathcal{B}(x), W_{x}): x \in \mathscr{T}_{t}\}$, so that $(\mathcal{W}_{t})_{t \geq 0}$ also incorporates information about the random `weights' of individuals in the tree $\mathscr{T}_{t}$. We let $(\mathcal{F}_{t})_{t \geq 0}$ and $(\mathscr{W}_{t})_{t \geq 0}$ denote the filtrations generated by $(\widetilde{\mathscr{T}}_{t})_{t \geq 0}$ and $(\mathcal{W}_{t})_{t \geq 0}$, respectively; and by taking their completions if necessary, assume that both $\mathcal{F}_{t}$ and $\mathscr{W}_{t}$ are complete. By abuse of notation, we use the symbol $\mathscr{T}_{t}$ to refer to  $\widetilde{\mathscr{T}}_{t}$, the set $\mathscr{T}_{t}$, and the graph associated with $\mathscr{T}_{t}$ where the vertex set is $\mathscr{T}_{t}$ and edges connect elements to their children. For a given choice of $X, W$, we say $(\mathscr{T}_{t})_{t \geq 0}$ is the genealogical tree process associated with an $(X,W)$-\emph{Crump-Mode-Jagers} branching process; often, we refer to $(\mathscr{T}_{t})_{t \geq 0}$ directly as an $(X,W)$-\emph{Crump-Mode-Jagers} branching process, viewed as a stochastic process in $t$, adapted to its natural filtration $(\mathscr{F}_{t})_{t \geq 0}$.\footnote{Note that distinct functions $(X,W)$, $(X',W')$ may lead to the same tree, if, for example $X(ui) = X'(ui) = \infty$, but $X(u(i+1)) \neq X'(u(i+1))$, but this is only a formal technicality, which we can overcome by viewing $(X,W)$ as an appropriate \emph{equivalence class} of functions.} 
	
	For $u \in \cU_{\infty}$, we let $\mathcal{P}_{i}(u)$ denote the time, after the \emph{birth} of $u$, required for $u$ to produce $i$ offspring. That is,
	\be\label{eq:P}
	\mathcal{P}_{i}(u) := \sum_{j=1}^{i} X(uj) \quad \text{and} \quad \mathcal{P}(u) := \mathcal{P}_{\infty}(u) = \sum_{j=1}^{\infty} X(uj).
	\ee
	Note that, as a result of this definition, for any $u, v \in \cU_{\infty}$ with $u=u_1 \cdots u_{k}$ and $v= v_1 \cdots v_{\ell}$, if we set $uv_0 := u$ we have 
	\begin{equation} \label{eq:birth-time-identity}
		\mathcal{B}(uv) - \mathcal{B}(u) = \sum_{j=0}^{\ell-1} \left(\mathcal{P}_{v_{j+1}}(uv_{|_{j}})\right). 
	\end{equation}
	It will also be beneficial to extend the notation $\mathcal{P}$ to arbitrary trees $T$: for $T \subseteq \mathcal{U}_{\infty}$, we define 
	\be \label{eq:P-t}
	\mathcal{P}_{T}(u) := \inf\left\{t > 0: \mathcal{B}(uv) - \mathcal{B}(u) < t\text{ for all } v \in T \right\}. 
	\ee
	Thus, with the above notation $\mathcal{P}_{i}(u) = \mathcal{P}_{[i]}(u)$. We also set $\mathcal{P}_{T} := \mathcal{P}_{T}(\varnothing)$, $\cP_i=\cP_i(\varnothing)$, and $\cP:=\cP(\varnothing)$.
	
	We generally assume a dependence between the values $(\mathcal{P}_{i}(u))_{i \in \mathbb{N} \cup \{\infty\}}$ and $W_{u}$. However, for brevity of notation, we often do not explicitly indicate this dependence. We use the notation $\mathcal{P}_{i}$ and $\mathcal{P}$ to denote generic copies of random variables distributed like $\mathcal{P}_{i}(u)$, and $\mathcal{P}(u)$ respectively. 
	
	For each $u \in \cU_{\infty}$ it will be helpful to a have a map $\xi^{(u)}: \Omega \times [0, \infty] \rightarrow \mathbb{N}$ indicating the number of children $u$ has produced, more precisely, we define $\xi^{(u)}(t)$ such that 
	\[
	\xi^{(u)}(t) = \begin{cases}
		\sum_{i=1}^{\infty}\mathbf{1}_{\{ \mathcal{P}_i(u)\leq s\}} &  \text{if $t = \mathcal{B}(u) + s, s \in [0, \infty]$;}
		\\ 0 & \text{otherwise.}
	\end{cases}
	\]
	With regards to the process $(\mathscr{T}_{t})_{t \geq 0}$, we define the stopping times $(\tau_{k})_{k \in \mathbb{N}_{0}}$ such that \[\tau_{k} := \inf\{t \geq 0: |\mathscr{T}_{t}| \geq k\},\] 
	where we adopt the convention that the infimum of the empty set is $\infty$. 
	One readily verifies that $(|\mathscr{T}_{t}|)_{t \geq 0}$ is right-continuous, and thus $|\mathscr{T}_{\tau_{k}}| \geq k$. For each $k \in \mathbb{N}$ we define the tree $\mathcal{T}_{k}$ as the tree consisting of the first $k$ individuals in $\mathscr{T}_{\tau_{k}}$ ordered by birth time, breaking ties lexicographically. We call $\tau_{\infty} := \lim_{k \to \infty} \tau_{k}$ the \emph{explosion time} of the process. We also define the tree $\mathcal{T}_{\infty} := \bigcup_{k=1}^{\infty} \mathcal{T}_{k}$. Note that it may be the case that $\left| \mathcal{T}_{\infty} \right| < \infty$; in this case $\tau_{\infty} = \infty$ and
	\begin{equation} \label{eq:extinction}
		\mathcal{T}_{\infty} = \left\{x \in \mathcal{U}_{\infty}: \mathcal{B}(x) < \infty \right\},
	\end{equation}
	and the set on the right-hand side is finite. If $|\mathcal{T}_{\infty}| < \infty$, we say \emph{extinction} occurs, otherwise \emph{survival} occurs. For a non-negative real-valued random variable $X$ and $\lambda > 0$ we let $\mathcal{M}_{\lambda}(Z), \mathcal{L}_{\lambda}(Z)$ denote the associated moment generating function and Laplace transform, respectively, i.e.
	\[
	\mathcal{M}_{\lambda}(Z) := \E{\e^{\lambda Z}} \quad \text{and} \quad \mathcal{L}_{\lambda}(Z) := \E{\e^{-\lambda Z}}.
	\]
	Moreover, if the random variable $Z$ has, additionally, some dependence on a random variable $W \in S$, we write 
	\[
	\cM_\lambda(Z;W):=\E{\e^{\lambda Z}\Big| W}\quad\text{and} \quad \mathcal{L}_{\lambda}(Z; W) := \E{\e^{-\lambda Z}\Big| W}.
	\]
	In addition, for real valued random variables $Z_1$ and $Z_2$, we say $Z_1 \leq_{S} Z_{2}$, if, for each $a \in \mathbb{R}$
	\[
	\Prob{Z_1 > a} \leq \Prob{Z_2 > a}.  
	\]
	Finally, for $r \geq 0$ we use $\Exp{r}$ to denote the exponential distribution with parameter $r$.
	
	\begin{remark} \label{rem:notation-form}
		With the more commonly used notation for CMJ branching processes, one assigns a point process (denoted $\xi^{(u)}$) to each $u \in \cU_{\infty}$, and refers to the points $\sigma^{(u)}_{1} \leq  \sigma^{(u)}_{2}, \ldots$ associated with this point process (in the notation used here $\mathcal{B}(u1), \mathcal{B}(u2), \ldots$). We do not use this framework here, because, this requires one to be able to write the measure $\xi^{(u)} = \sum_{i=1}^{\infty} \delta_{\sigma_{i}}$, which requires one to impose $\sigma$-finiteness assumptions on the point process (see, for example, \cite[Corollary~6.5]{last-penrose}). This $\sigma$-finiteness is implied by the classical Malthusian condition, but, in this general setting, we believe it is easier to have a framework where one can directly refer to the points $\mathcal{B}(u1), \mathcal{B}(u2), \ldots$ 
	\end{remark}
	
	\section{Statements of main results}\label{sec:results}
	In this paper, we are interested in properties of the infinite tree $\mathcal{T}_{\infty}$, in particular the question of whether or not this tree contains an infinite path or a node of infinite degree. This section deals with results in their most general form: Section~\ref{sec:globass} states some global assumptions imposed throughout the paper, Section~\ref{sec:star} deals with criteria for a node of infinite degree (or star), Section~\ref{sec:structure-} deals with criteria for an infinite path, and structural properties of the tree $\mathcal{T}_{\infty}$ when there is a star, and Section~\ref{sec:unique} deals with uniqueness properties, providing criteria for a the appearance of a unique star or unique infinite path (but not both) to appear almost surely. In Theorem~\ref{thm:counter-example} we also show that in the regime where there is, almost surely, exactly one of a unique star or infinite path, either may appear with positive probability. Finally, we provide an overview of the proof techniques used, and the relation to existing literature in Section~\ref{sec:techniques-discussion-1}. 
	
	\subsection{Global assumptions}\label{sec:globass}
	In general in this paper, we assume that the values of $(X(u j))_{j \in \mathbb{N}}$ depend on $W_{u}$. 
	We also assume that
	the sequences of random variables
	\begin{equation} \label{eq:cmj-assumption}
		((X(uj))_{j \in \mathbb{N}}, W_{u}) \quad \text{are i.i.d.\ for different $u \in \cU_{\infty}$};
	\end{equation}
	although we expect that some of our techniques may carry over to a more general setting. 
	For a given $w \in \mathcal{S}$, we let $(X_{w}(uj))_{j \in \mathbb{N}}$ denote a sequence $(X(uj))_{j \in \mathbb{N}}$, conditionally on the weight $W_{u} = w$. Another common assumption we use throughout is the following: for any given $w \in S$ and any $u\in\cU_\infty$, the sequence of random variables 
	\be\label{eq:starindep}
	(X_w(ui))_{i\in\N} \quad \text{is mutually independent.}
	\ee
	The existence of an explosive $(X,W)$-CMJ process satisfying~\eqref{eq:cmj-assumption} is a straightforward application of (for example) the Kolmogorov extension theorem (indeed, this is the classical definition of a CMJ process). 
	
	We also generally assume that the event $\left\{|\mathcal{T}_{\infty}| = \infty \right\}$ has positive probability, and 
	\begin{equation} \label{eq:expl-almost-surely} 
		\Prob{\tau_{\infty} < \infty \, \big | \, |\mathcal{T}_{\infty}| = \infty} = 1.
	\end{equation}
	That is, the process is almost surely explosive, when conditioned on survival. In general, we say an event $A$ occurs \emph{almost surely on survival} if $\Prob{A \, \big | \, |\mathcal{T}_{\infty}| = \infty} = 1$. In all statements in this paper referring to an ``explosive $(X,W)$-CMJ process'', we assume it satisfies~\eqref{eq:cmj-assumption} and~\eqref{eq:expl-almost-surely}. 
	
	In this paper, we also rely on the following well-known fact in graph theory.
	
	\begin{lemma}[K{\H o}nig's Lemma]\label{lemma:konig}
		Any infinite tree contains a node of infinite degree or an infinite path.
	\end{lemma}

	\subsection{Sufficient criteria for a star}\label{sec:star}
	In this subsection, we provide sufficient criteria for the infinite tree $\mathcal{T}_{\infty}$ to contain an infinite star. Our main assumptions are as follows.
	\begin{assumption} \label{ass:star}
		We have the following conditions.
		\begin{enumerate}
			\item\label{item:stardom} There exist non-negative real-valued random variables $(Y_{n})_{n \in \mathbb{N}_{0}}$ with finite mean such that, for any $w \in S$, 
			\begin{equation} \label{eq:stochastic-bound}
				\sum_{i=n+1}^{\infty} X_{w}(i) \leq_{S} Y_{n}. 
			\end{equation}
			\item\label{item:starlimsup} If we let $\mu_{n} := \E{Y_{n}}$, then we also have, for some $c \in (0, \infty)$,
			\begin{equation} \label{eq:limsup-mgf}
				\limsup_{n \to \infty} \mathcal{M}_{c\mu^{-1}_{n}}(Y_n) < \infty.
			\end{equation}
			Moreover, we assume that $(\mu_{n})_{n \in \mathbb{N}_{0}}$ is non-increasing in $n$, with $\lim_{n \to \infty} \mu_{n} = 0$.
			\item\label{item:starindep} For each $n \in \mathbb{N}$ and any given $w \in S$, the random variables $(X_{w}(i))_{i \in \mathbb{N}}$ are mutually independent.
			\item\label{item:starnonzero} For each $n \in \mathbb{N}$, 
			\begin{align}
				\sum_{i=n+1}^{\infty} X(i) &> 0 \quad \text{almost surely},  \label{eq:non-instantaneous-sideways}
				\shortintertext{and additionally, we have}
				\E{\xi(0)} &= \E{\sup\{k: X(k) = 0\}} < 1. \label{eq:non-instantaneous-explosion}
			\end{align} 
			\item\label{item:starlaplace} With $c$ as appearing in Equation~\eqref{eq:limsup-mgf},
			\be\label{eq:laplacesum}
			\sum_{i=1}^{\infty}\E{\mathcal{L}_{c\mu^{-1}_{i}}(\mathcal{P}_{i}(\varnothing); W)} < \infty.  
			\ee
		\end{enumerate}
	\end{assumption}
	
	\begin{remark}
		Note that under Condition~\ref{item:stardom} of Assumption~\ref{ass:star}, we have $\Prob{\left|\mathcal{T}_{\infty} \right| = \infty} =1$. 
		{\Large\ensymboldremark}
	\end{remark}
	
	\begin{remark} \label{rem:assumpt-motiv}
		In Assumption~\ref{ass:star} we can consider Conditions~\ref{item:stardom} and~\ref{item:starindep} as a \emph{uniform explosivity} condition: it implies that, for any $u \in \mathcal{U}_{\infty}$ and any $\eps > 0$, 
		\[
		\lim_{L \to \infty} \Prob{\sum_{j=L+1}^{\infty} X_{j}(u) \geq \eps \, \bigg | \, X_{1}(u), \ldots, X_{L}(u)} = 0,
		\] 
		with the convergence uniform in $X_{1}(u), \ldots, X_{L}(u)$; a fact that is crucial for Lemma~\ref{lem:expl-min-exp-time}, and hence the proof of Theorem~\ref{thm:star}, to hold. Condition~\ref{item:starnonzero} is there as a technical assumption, used, for example, in Proposition~\ref{prop:finite-l-moderate} and Lemma~\ref{lemma:tree-ident}. It ensures that  $\tau_{k} < \tau_{\infty}$ for each $k \in \mathbb{N}$. Indeed, if, for example, $\E{\xi(0)} > 1$, the tree consisting of all the individuals born \emph{instantaneously} at time $0$ is the genealogical tree of a supercritical Bienaym\'{e}-Galton-Watson branching process. Hence, with positive probability  this tree is infinitely large, and thus there may be no node of infinite degree in this infinite tree. Condition~\ref{item:starlimsup} is used to prove the Chernoff type concentration bound in Lemma~\ref{lem:moment-prob-bounds} which, when combined with the summability condition in Condition~\ref{item:starlaplace}, leads to the proof of the crucial Proposition~\ref{prob:local-explosions}. {\Large\ensymboldremark}
	\end{remark}
	
	\noindent The conditions of Assumption~\ref{ass:star} allow us to formulate the following theorem.
	
	\begin{thm}[Infinite star] \label{thm:star}
		Under Assumption~\ref{ass:star}, almost surely the infinite tree $\cT_{\infty}$ contains a node of infinite degree (i.e.\ an infinite \emph{star}). 
	\end{thm}
	
	\noindent The proof of Theorem~\ref{thm:star} appears in Section~\ref{sec:starproof}.
	
	\subsection{Sufficient criteria for an infinite path and structural results in the star regime} \label{sec:structure-}
	In this subsection, we provide sufficient criteria for $\mathcal{T}_{\infty}$ to contain an infinite path and whether or not $\mathcal{T}_{\infty}$ contains infinitely many copies of a fixed tree $T$. We first state the following assumption. 
	
	\begin{assumption} \label{ass:path}
		There exists a collection of numbers $\left\{\nu^{w}_{n} \in [0, \infty): w \in S, n \in \mathbb{N}\right\}$, such that for any $w \in S$,
		\begin{align} 
			\sum_{i=1}^{\infty} \Prob{\mathcal{P} < \nu^{w}_{i}} &= \infty,\label{eq:div-condition}
			\shortintertext{and}
			\hspace{-1cm}\liminf_{i \to \infty} \Prob{\sum_{j=i+1}^{\infty} X_{w}(j) \geq \nu^{w}_{i}} &> 0. \label{eq:smallest-expl-prob}
		\end{align}
	\end{assumption}
	
	\begin{remark} 
		Assumption~\ref{ass:path} intuitively states that, conditionally on the weight $w$ of the root $\varnothing$, infinitely many children $i$ of $\varnothing$ produce an infinite offspring within time $\nu_i^w$. On the other hand, the root takes at least $\nu_i^w$ amount of time after the birth of its $i$th child to produce an infinite offspring, with a probability that is bounded from below, uniformly in $i$. Hence, infinitely many children $i$ explode before their parent $\varnothing$. {\Large\ensymboldremark}
	\end{remark}
	
	\noindent We can then formulate the following theorem. 
	
	\begin{thm}[Infinite path] \label{thm:path}
		Under Assumption~\ref{ass:path}, the tree $\cT_{\infty}$ contains an infinite path almost surely on survival. 
	\end{thm}
	
	\noindent The proof of Theorem~\ref{thm:path} appears in Section~\ref{sec:inf-path-structure-proof}. 
	
	Similar criteria to those related to the criteria for an infinite path allow us to also determine results related to the \emph{structure} of $\mathcal{T}_{\infty}$ in the sense that, when we know that $\cT_\infty$ contains an infinite star, certain sub-structures appear infinitely often; others only finitely often. We define $\mathcal{T}_{\infty}(\downarrow u) := \left\{v \in \mathcal{T}_{\infty}: v = uw, \; w \in \mathcal{U}_{\infty} \right\}$ as the sub-tree in $\cT_\infty$ rooted at $u$. For a fixed tree $T$ containing $\varnothing$ and $u \in \mathcal{U}_{\infty}$, we define $uT:= \left\{uv: v \in T\right\}$. We say such a tree $T$ is a sub-tree rooted at $u \in \mathcal{T}_{\infty}$, if, $uT \subseteq \mathcal{T}_{\infty}(\downarrow u)$. Recalling Equation~\eqref{eq:P-t} we then have the following set of assumptions.
	
	\begin{assumption} \label{ass:structure}
		Let $(\mathscr{T}_{t})_{t \geq 0}$ be an explosive $(X,W)$-CMJ process. For a given finite tree $T \subseteq \mathcal{U}_{\infty}$ containing $\varnothing$ we have the following conditions.
		\begin{enumerate}
			\item\label{item:treediv-cond} There exists a collection of numbers $\left\{\nu^{w}_{n} \in [0, \infty): w \in S, n \in \mathbb{N}_{0}\right\}$, such that for any $w \in S$,
			\begin{align} 
				\sum_{i=1}^{\infty} \Prob{\mathcal{P}_{T} < \nu^{w}_{i}} &= \infty,\label{eq:div-condition-sub-tree}
				\shortintertext{and}
				\hspace{-1cm}	\liminf_{i \to \infty} \Prob{\sum_{i=j+1}^{\infty} X_{w}(j) \geq \nu^{w}_{i}} &> 0.\label{eq:smallest-expl-prob-sub}
			\end{align}
			\item\label{item:treedomsum} For any $w \in S$ and with $(\widetilde{X}_{w}(i), i \in \mathbb{N}) \sim (X_{w}(i), i \in \mathbb{N})$, independent of the process $(\mathscr{T}_{t})_{t \geq 0}$, 
			\begin{equation} \label{eq:sum-cond-sub-tree}
				\sum_{i=1}^{\infty} \Prob{\mathcal{P}_{T} < \sum_{j = i+1}^{\infty} \wt X_{w}(j)} < \infty.
			\end{equation}
		\end{enumerate}
	\end{assumption}
	
	\noindent We can then formulate the following theorem. 
	
	\begin{thm}[Sub-tree count] \label{thm:structure}
		Let $(\mathscr{T}_{t})_{t \geq 0}$ be an explosive $(X,W)$-CMJ process. Then, for any finite tree $T \subseteq \mathcal{U}_{\infty}$ containing $\varnothing$:
		\begin{enumerate}
			\item \label{item:structure-1} If Condition~\ref{item:treediv-cond} of Assumption~\ref{ass:structure} is satisfied, almost surely,  if $u \in \mathcal{T}_{\infty}$ has infinite degree, $T$ appears infinitely often as a sub-tree rooted at a child of $u$. 
			\item \label{item:structure-2} If Condition~\ref{item:treedomsum} of Assumption~\ref{ass:structure} is satisfied, almost surely,  if $u \in \mathcal{T}_{\infty}$ has infinite degree, $T$ appears only finitely often as a sub-tree of a child of $u$. In particular, if Assumption~\ref{ass:star} is satisfied, almost surely, the tree $T$ appears only finitely often as a sub-tree of $\mathcal{T}_{\infty}$.
		\end{enumerate}
	\end{thm}
	
	\noindent The proof of Theorem~\ref{thm:structure} appears in Section~\ref{sec:inf-path-structure-proof}. 
	
	\subsection{Uniqueness conditions related to the existence of a star or an infinite path}\label{sec:unique}
	In many cases, we expect $\mathcal{T}_{\infty}$ to contain \emph{exactly one} node of infinite degree or \emph{exactly one} infinite path and also, often expect \emph{co-existence} of an infinite path and node of infinite degree to be impossible. This leads us to the following assumption.
	
	\begin{assumption} \label{ass:uniqueness}
		We have the following conditions.
		\begin{enumerate}
			\item\label{item:uniquenonzero} Condition~\ref{item:starnonzero} of Assumption~\ref{ass:star} is satisfied. 
			\item\label{item:uniquenoexplatom} The distribution of $\mathcal{P}(\varnothing)$ contains no atom on $[0, \infty)$.
			\item\label{item:uniquenobirthatom} For some $\eps > 0$ and for each $i \in \mathbb{N}$, the distribution of $\mathcal{B}(i)$ contains no atom on $[0, \eps)$.
		\end{enumerate}
	\end{assumption}
	
	\noindent We then have the following result. 
	
	\begin{thm}[Unique infinite star or path] \label{thm:uniqueness}
		Let $(\mathscr{T}_{t})_{t \geq 0}$ be an explosive $(X,W)$-CMJ process. Then:
		\begin{enumerate}
			\item If Conditions~\ref{item:uniquenonzero} and~\ref{item:uniquenoexplatom} of Assumption~\ref{ass:uniqueness} are satisfied, $\mathcal{T}_{\infty}$ contains at most $1$ node of infinite degree, almost surely. 
			\item If Conditions~\ref{item:uniquenonzero} and~\ref{item:uniquenobirthatom} of Assumption~\ref{ass:uniqueness} are satisfied, $\mathcal{T}_{\infty}$ contains at most $1$ infinite path, almost surely.
			\item\label{item:uniquepors} If all the conditions of Assumption~\ref{ass:uniqueness} are satisfied, almost surely, on survival, $\mathcal{T}_{\infty}$ contains exactly one of the following: a node of infinite degree, an infinite path. 
		\end{enumerate}
	\end{thm}   
	
	\noindent The proof of Theorem~\ref{thm:uniqueness} appears in Section~\ref{sec:uniqueproof}. 
	
	\begin{remark}
		Though assumed to hold throughout, Theorem~\ref{thm:uniqueness} can be proved without assuming~\eqref{eq:starindep}.{\Large\ensymboldremark}
	\end{remark}
	
	\begin{remark}
		A case when the tree $\mathcal{T}_{\infty}$ has more than one infinite path with positive probability is when \emph{time $0$ explosion} can occur, i.e.\ when $\E{\xi(0)} > 1$, so that $\tau_{\infty} = 0$ with positive probability, and $\mathcal{T}_{\infty}$ is the genealogical tree of a supercritical Bienaym\'{e}-Galton-Watson branching process. This case is ruled out by Item~\ref{item:uniquenoexplatom} of Assumption~\ref{ass:uniqueness}, which, in particular, implies that $\xi(0) = 0$ almost surely. {\Large\ensymboldremark}
	\end{remark}
	
	\noindent Given Item~\ref{item:uniquepors} of Theorem~\ref{thm:uniqueness}, one might expect the event that $\mathcal{T}_{\infty}$ contains a node of infinite degree to occur with probability $0$ or $1$: for example, perhaps one might expect this to event to `look like' a tail event, measurable with respect to the tail sigma algebra of an appropriate filtration. The following theorem shows that this is not actually the case in full generality.
	
	\begin{thm} \label{thm:counter-example}
		There exists an explosive $(X,W)$-CMJ process satisfying the conditions of Assumption~\ref{ass:uniqueness} such that
		\begin{equation} \label{eq:star-non-tail}
			\Prob{\mathcal{T}_{\infty} \text{ contains a node of infinite degree } }\in (0,1). 
		\end{equation}
	\end{thm}
	
	\noindent Note that, applying Item~\ref{item:uniquepors} of Theorem~\ref{thm:uniqueness}, Equation~\eqref{eq:star-non-tail} also implies that for such an explosive $(X,W)$-CMJ process,
	\[
	\Prob{\mathcal{T}_{\infty} \text{ contains an infinite path} }\in (0,1).
	\]
	We note that, unlike the other theorems stated in this section, the proof of Theorem~\ref{thm:counter-example} appears in Section~\ref{sec:applications}, in particular 
	in Section~\ref{sec:counter}.  
	
	\subsection{Proof techniques and relation to existing literature} \label{sec:techniques-discussion-1}
	As mentioned in the introduction, Theorem~\ref{thm:star} was proved in the case that the $(X(i), i \in \mathbb{N})$ are independent, with $X(i) \sim \Exp{i^{p}}, p > 1$ in~\cite{Oliveira-spencer}. The technique used in that paper was to show that the number of nodes $u \in \mathcal{U}_{\infty}$ that are \emph{$k$-fertile} in the tree $\mathcal{T}_{\infty}$ (that is, contain a sub-tree of size at least $k+1$), is almost surely finite for all $k>1/(p-1)$~\cite[Lemma~5.1]{Oliveira-spencer}, and then deduce an infinite path cannot exist, almost surely. Applying Lemma~\ref{lemma:konig} then yields the desired conclusion. An immediate generalisation of these techniques to the more general setting considered here does not allow one to to prove Theorem~\ref{thm:star}. Indeed, as we will see in Theorem~\ref{thrm:sub-treecount}, Theorem~\ref{thm:star} applies to cases where  the number of $k$-fertile nodes is \emph{almost surely infinite} for \emph{any} $k\in\N$. Instead, we use a different approach to prove Theorem~\ref{thm:star}. By a first moment method and appropriate concentration bounds (Lemma~\ref{lem:moment-prob-bounds}), we show that the expected number of nodes $a \in \mathcal{U}_{\infty}$, with a high enough initial index $a_{1}$, that explodes before all of its ancestors (that is, produces infinitely many offspring before any of its direct ancestors does), is finite (see~Proposition~\ref{prob:local-explosions}). Combining this with a coupling argument in Proposition~\ref{prop:finite-l-moderate} (a significant generalisation of~\cite[Lemma~5.3]{Oliveira-spencer}), we show that the expected number of nodes that explodes before all of their ancestors is finite. Finally, the \emph{uniform explosivity} assumption (see Remark~\ref{rem:assumpt-motiv}) allows one to deduce that the explosion time of the process $\tau_{\infty}$ is the infimum of the explosion times $\mathcal{B}(u) + \mathcal{P}(u)$ of individuals $u \in \mathcal{U}_{\infty}$. Using the aforementioned first moment arguments, we can show that this infimum coincides with an infimum over a finite set. Hence, at $\tau_{\infty}$ there exists at least one node of infinite degree. 
	
	The proofs of Theorems~\ref{thm:path} and~\ref{thm:structure} use a different approach: by Borel-Cantelli arguments, we can show that \emph{before} the explosion time associated with an individual, either infinitely many children explode themselves (leading to an infinite path, cf.\ Theorem~\ref{thm:path}), or otherwise certain finite sub-trees appear infinitely often when there is a star (cf.\ Theorem~\ref{thm:structure}). The uniqueness conditions appearing in Theorem~\ref{thm:uniqueness} are reminiscent of similar uniqueness conditions appearing, for example, previously in~\cite{Oliveira-spencer} (using the fact that the associated distribution of $\mathcal{P}(\varnothing)$ is smooth). However, this requires novelty when proving the existence of a unique infinite path in the level of generality we consider (see Lemma~\ref{lem:no-atom}).

	\section{Examples of applications and an open problem} \label{sec:examples}
	In this section, we provide applications of our main results, Theorem~\ref{thm:star}, Theorem~\ref{thm:path}, and Theorem~\ref{thm:uniqueness} in the case that the inter-birth times $(X_{w}(i))_{i \in \mathbb{N}}$ are exponentially distributed. If $X_w(i)$ has an exponential distribution with parameter $f(i,w)$, say, the memory-less property and the property of minima of exponential distributions show that $\mathcal{T}_{\infty}$ may be interpreted as the limiting infinite tree in a model of \emph{$(W,f)$-recursive trees with fitness}. In this model, evolving in discrete time, nodes arrive one at a time, are assigned i.i.d.\ weights, and connect to an existing node sampled with probability proportional to its `fitness function'. In this model, we are not only able to provide phase transitions related to the emergence of an infinite path, but also apply Theorem~\ref{thm:structure} to provide criteria for the emergence of a particular sub-tree infinitely often. This is the content of Section~\ref{sec:recursive}. 
	
	In Section~\ref{sec:gen-cmj} we consider more concrete cases, when the weights $W$ are real-valued, closely connected to super-linear preferential attachment models. In Section~\ref{sec:motiv-2} we provide some background for the analysis of such models, in the context of existing literature. Then, Theorem~\ref{thrm:cmjexamples} in Section~\ref{sec:gen-cmj} provides a classification of the phases where one sees a unique node of infinite degree or a unique infinite path, proving phase transitions for three different examples. These results apply not only to the case that the values $(X_{w}(i))_{i \in \mathbb{N}}$ have exponential distributions, but other distribution types (see Remarks~\ref{rem:more-general} and~\ref{rem:otherdistr}); which may be of interest in applications. In Theorem~\ref{thrm:sub-treecount} we are able to characterise the sub-trees of children of the star that can emerge in this model. We discuss implications of these results in the particular `super-linear degree' example in Section~\ref{sec:discussion-phase-diagrams}, which, in particular, allows us to produce phase diagrams in Figures~\ref{fig:mixsub-tree} and~\ref{fig:addsub-tree}.
	
	Finally, we discuss the proof techniques involved in Section~\ref{sec:techniques-discussion-2} and state an open problem in Section~\ref{sec:open}.
	
	\subsection{The structure of explosive recursive trees with fitness} \label{sec:recursive}
	
	Suppose the values of $(X_{w}(i), i \in \mathbb{N})$ are exponentially distributed and independent. The properties related to the exponential distribution yield that the sequence of trees $(\mathcal{T}_{i})_{i \in \mathbb{N}}$ associated with an explosive $(X,W)$-CMJ branching process are identical in law to a sequence of \emph{recursive trees with fitness} which we define below. First, we define the \emph{fitness function} $f: \mathbb{N}_{0} \times S \rightarrow [0, \infty)$ to be a measurable function such that $f(i, w)$ is the rate of the exponential random variable $X_{w}(i+1)$. 
	
	In this section, we generally consider trees as being rooted with edges directed away from the root, and hence the number of `children' of a node corresponds to its \emph{out-degree}. More precisely, given a vertex labelled $v$ in a directed tree $T$ we let $\outdeg{v, T}$ denote its out-degree in $T$.  We now define the recursive tree with fitness model.
	
	\begin{definition}[Recursive tree with fitness]\label{def:wrt}
		Suppose that $(W_{i})_{i \in \mathbb{N}}$ are i.i.d.\ copies of a random variable $W$ that takes values in $S$, and let $f:\mathbb{N}_{0}\times S \rightarrow [0, \infty)$ denote a fitness function. A $(W, f)$-recursive tree with fitness is the sequence of random trees $(\mathcal{T}_{i})_{i \in \mathbb{N}}$ such that: $\mathcal{T}_{0}$ consists of a single node $0$ with weight $W_{0}$ and for $n \geq 1$, $\mathcal{T}_{n}$ is updated recursively from $\mathcal{T}_{n-1}$ as follows:
		\begin{enumerate}
			\item Sample a vertex $j \in \mathcal{T}_{n-1}$ with probability proportional to its fitness, i.e., with probability 
			\begin{equation}
				\frac{f(\outdeg{j, \mathcal{T}_{n-1}}, W_j)}{\sum_{j=0}^{n-1} f(\outdeg{j, \mathcal{T}_{n-1}}, W_j)}.
			\end{equation}
			\item Connect $j$ with an edge directed outwards to a new vertex $n$ with weight $W_{n}$.
		\end{enumerate}
	\end{definition}
	\begin{remark}
		Due to the equivalence in law with trees associated with an $(X,W)$-CMJ branching process, by abuse of notation we refer to a sequence of recursive tree with fitness by $(\mathcal{T}_{i})_{i \in \mathbb{N}}$, despite the fact that the vertex set of these trees is $\mathbb{N}_0$ rather than $\mathcal{U}_{\infty}$. {\Large\ensymboldremark}
	\end{remark}
	\begin{remark}
		The correspondence between recursive trees with fitness and the trees $(\mathcal{T}_{i})_{i \in \mathbb{N}}$ associated with an $(X, W)$-CMJ process, when $X_{w}(i) \sim \Exp{f(i+1,w)}$ holds for all $i\in\N, w\in S$, is a consequence of the memory-less property and the fact that the minimum of exponential random variable is also exponentially distributed, with a rate given by the sum of the rates of the corresponding variables; see for example~\cite[Section~2.1]{rec-trees-fit}. The use of such continuous-time embeddings to analyse combinatorial processes was pioneered by Arthreya and Karlin~\cite{arthreya-karlin-embedding-68}. This correspondence allows us to translate our main results to the infinite recursive tree, which we also denote by $\mathcal{T}_{\infty}$.
		{\Large\ensymboldremark}
	\end{remark}
	
	\subsubsection{The star/path transition in explosive recursive trees with fitness}
	
	Our first result pertains to the existence of a node of infinite degree or an infinite path in recursive trees with fitness. To this end, we make the following assumption:
	\be\label{eq:minfass}
	\exists w^*\in S: \forall w\in S, j\in \N: f(j,w)\geq f(j,w^*) \quad \text{and } \quad \sum_{j=0}^{\infty} \frac{1}{f(j, w^{*})} < \infty.\tag{$w^*$}
	\ee 
	That is, there exists a minimiser $w^*\in S$ that, uniformly in $j\in \N$, minimises $f(j,\cdot)$, and the reciprocals of $f(j,w^*)$ are summable. Moreover, we define
	\be \label{eq:mu-def}
	\mu_n^w:= \sum_{i=n}^\infty \frac{1}{f(i,w)}, \qquad w\in S, n\in \N, \quad \text{and set } \mu_n:=\mu_n^{w^*}.
	\ee 
	Note that $\mu_n^w\leq \mu_n<\infty$ for all $w\in S$ by~\eqref{eq:minfass}.
	
	\begin{thm}[Star/path in explosive recursive trees] \label{thm:star-path-rif}
		Let $(\mathcal{T}_{i})_{i \in \mathbb{N}}$ be a $(W, f)$-recursive tree with fitness and assume $f$ satisfies~\eqref{eq:minfass}. Then, 
		\begin{enumerate}
			\item \label{item:star-explosive-rif}If, for some $c < 1$, we have  \be \label{eq:star-explosive-rif}
			\sum_{n=1}^{\infty} \E{\prod_{i=0}^{\infty}\frac{f(i,W)}{f(i,W) + c\mu_{n}^{-1}}} < \infty, \ee the tree $\mathcal{T}_{\infty}$ contains a unique node of infinite degree, and no infinite path. 
			\item \label{item:path-explosive-rif} If either for some $c > 1$ and all $w \in S$ , we have  
			\be \label{eq:path-explosive-rif}
			\sum_{n=1}^{\infty} \E{\prod_{i=0}^{\infty}\frac{f(i,W)}{f(i,W) + c (\mu^{w}_{n})^{-1} \log{n}}} = \infty,\ee 
			or, as a weaker condition, Equation~\eqref{eq:div-condition} is satisfied with $\nu^{w}_{n} := d \mu^{w}_{n}$ for $d <1$, the tree $\mathcal{T}_{\infty}$ contains a unique infinite path, and no node of infinite degree. 
		\end{enumerate}
	\end{thm}
	
	\noindent The proof of Theorem~\ref{thm:star-path-rif} appears in Section~\ref{sec:applications}, in particular Section~\ref{proof:star-path-rif}. 
	\begin{remark}
		It turns out that Equation~\eqref{eq:path-explosive-rif} is a sufficient condition for Equation~\eqref{eq:div-condition} to be satisfied with $\nu^{w}_{n} := d \mu^{w}_{n}$ for $d <1$. However, we include it as a general comparison to Equation~\eqref{eq:star-explosive-rif}. {\Large\ensymboldremark}
	\end{remark}
	
	\begin{remark} \label{rem:more-general}
		Analogues of  Theorem~\ref{thm:star-path-rif} extend to more general distributions of $X_{w}(i)$, other than exponential distributions. In particular, we can apply the same techniques used to prove Item~\ref{item:path-explosive-rif} of Theorem~\ref{thm:star-path-rif} whenever $(X_{w}(i), i \in \mathbb{N})$ are independent and $\Var{(X_{w}(i))} \leq K \E{X_{w}(i)}^2$ for some $K > 0$, possibly depending on $w$. In this case the expected value in~\eqref{eq:path-explosive-rif} is replaced by the Laplace transform $\E{\mathcal{L}_{c (\mu^{w}_{n})^{-1} \log{n}}(\mathcal{P}(\varnothing); W)}$. See the proof in Section~\ref{proof:star-path-rif} for more details.{\Large\ensymboldremark}
	\end{remark}
	
	\subsubsection{Sub-trees in explosive recursive trees with fitness when there is a star}\label{sec:expsub-tree}
	
	In Section~\ref{sec:structure-}, we described a tree $T$ as a finite subset of the Ulam-Harris tree containing the root, with a natural directed edge structure induced by parents being connected to children. We apply the same notion here, upon identifying labels of elements of $\mathcal{T}_{\infty}$ with the Ulam-Harris labelling. For a tree $T \subseteq \mathcal{U}_{\infty}$ and $u \in \mathcal{T}_{\infty}$ (when we label the elements of $\mathcal{T}_{\infty}$ with the Ulam-Harris labelling), we say that $T$ appears as a sub-tree, rooted at $u$, in $\mathcal{T}_{\infty}$ if $uT \subseteq \mathcal{T}_{\infty}$. Because the presence of `earlier siblings' in a copy of a tree $T$ can influence the probability of a tree emerging \footnote{For example, if the tree $T$ corresponds to a path, it is intuitively less likely to have a path emerge where every node of the path is the first child of its parent rather than a path where some nodes are born later but, by random chance, produce children faster.} it is convenient to also assume that $T$ is \emph{sibling closed}, where we define as follows. If $u = u_1 \cdots u_{m} \in T$ then $u = u_1 \cdots u_{m-1} \ell \in T$ for each $\ell \in [u_{m}]$. 
	
	The occurrence of a sibling-closed tree $T$ in $\mathcal{T}_{\infty}$ may also depend on the order in which the vertices in $T$ appear, which can vary in such a way that they preserve the lexicographic ordering. An \emph{ordering} of a tree $T$ with $|T|=k+1$ vertices, for some $k\in\N$, is a permutation $O:T \rightarrow \left\{0, 1, \ldots, k\right\}$, such that $O(u) \leq O(v)$ if and only if $u \leq_{L} v$. Given an ordering $O$, we generally refer to the vertices of a tree $T$ with $k+1$ vertices as $\left\{v_{0}, \ldots, v_{k} \right\}$, where  $v_{i} := O^{-1}(i)$ for each $i$. Given a sibling-closed tree $T$, we let $\mathcal{O}(T)$ denote the set of all orderings of $T$. For a given ordering $O$ and $j \leq k$, we let $O_{|_j}$ denote the (also sibling-closed) tree on the vertex set $\left\{v_0, \ldots, v_j \right\}$; note that this is well defined because $O$ preserves the order $\leq_{L}$. Also note that each $O_{|_j}$ inherits the natural directed edge structure from $T$. For a given vertex $v_{i}$, with $i \leq j$, $\outdeg{v_{i}, O_{|_j}}$ denotes its out-degree in $O_{|_j}$. 
	
	We then have the following theorem.
	
	\begin{thm}[Sub-tree counts]\label{thm:sub-treecount}
		Fix $k\in\N$, and let $(\mathcal{T}_{i})_{i \in \mathbb{N}}$ be a $(W, f)$-recursive tree with fitness such that $f$ satisfies~\eqref{eq:minfass} and so that $\cT_\infty$ contains a unique star. Moreover, assume that for each $w \in S$ we have $\mu_{n}^{w} \geq c_1(w) \mu_{n}$, where $0 < c_{1}(w) \leq 1$. Let $T$ be a sibling-closed tree, with $|T| = k+1$. The tree $\mathcal{T}_{\infty}$ contains $T$ as a sub-tree infinitely often if and only if 
		\begin{equation} \label{eq:tree-summ-cond}
			\sum_{n=1}^{\infty} \sum_{O \in \mathcal{O}(T)} \E{ \prod_{j=0}^{k} \frac{\prod_{\ell=0}^{\outdeg{v_j, T} - 1} f(\ell, W_{v_j})}{\sum_{i=0}^{j}f(\outdeg{v_i, O_{|_j}},W_{v_i}) \ind{\outdeg{v_i, O_{|_j}} < \outdeg{v_i, T}} + \mu^{-1}_{n}}} = \infty. 
		\end{equation}
	\end{thm}
	
	\begin{cor} \label{cor:sumnufinmean}
		Fix $k\in\N$ and assume that, for each $i \in [k]$ we have $\E{f(i,W)^{k}} < \infty$. Let $T$ be a sibling-closed tree with $k+1$ vertices. Then, under the assumptions of Theorem~\ref{thm:sub-treecount} the tree $\cT_\infty$ contains $T$ as a sub-tree infinitely often if and only if $\sum_{n=1}^\infty \mu_n^k=\infty$.
	\end{cor}
	
	\noindent The proof of Theorem~\ref{thm:sub-treecount} appears in Section~\ref{sec:sub-tree-count-general-rif}. On the other hand, the proof of Corollary~\ref{cor:sumnufinmean} appears separately, in Section~\ref{sec:cor-sumnifinmean-proof}. 
	
	\subsection{Phase transitions in specific models of explosive recursive trees with fitness} \label{sec:gen-cmj}
	In this section we investigate three particular cases of the results presented in Section~\ref{sec:recursive}, where we are able to prove \emph{phase transitions} for the structure of the infinite limiting tree $\cT_\infty$ in terms of the \emph{fitness function} and the \emph{vertex-weight distribution}. We assume that the vertex-weights are non-negative and real-valued, i.e.\ they take values in $S=[0, \infty)$. We also assume that the fitness function $f(i,W)$ grows \emph{faster than linear} in the degree (i.e.\ its first argument). These cases are thus examples of \emph{super-linear preferential attachment with fitness}.
	
	\subsubsection{Connection to existing literature: super-linear preferential attachment} \label{sec:motiv-2}
	
	As alluded to in the introduction, Section~\ref{sec:rec-fit}, a model of recursive trees (with fitness) that has received substantial attention the last two decades are preferential attachment models. Such models are thought to serve as a good explanation of the formation \emph{real-world} networks due the \emph{preferential attachment paradigm}, which suggests that networks are constructed by adding vertices and edges successively, in such a way that new vertices prefer to be connected to existing vertices with large degree. In particular, many of such models intrinsically give rise to properties also found in many real-world networks (i.e.\ the scale-free property and (ultra)small-world property), rather than such properties being imposed on the model. We refer to~\cite{Hof16} and the references therein for an extensive overview of the literature on such models and their applications.
	
	Super-linear preferential attachment is a particular type of preferential attachment where new vertices connect to existing vertices with out-degree $i$ with a probability proportional to $f(i)$, for some fitness function $f:\N_0\to(0,\infty)$ such that $\sum_{i=0}^\infty f(i)^{-1}<\infty$. Most often, as in e.g.~\cite{choi-sethuraman-2013,Oliveira-spencer,sethuraman-venkataramani-2019}, the case $f(i)=(i+1)^p$ for some $p>1$ is studied, though there are also choices for $f$ that satisfy the summability condition such that $f(i)i^{-p}\to 0$ as $i\to\infty$ for any $p>1$. We coin these functions \emph{barely super-linear}. Though P\'olya urn models with barely super-linear fitness functions have been studied previously~\cite{GottGross23}, as far as the authors are aware this is the first case such fitness functions are treated for preferential attachment models.
	
	Super-linear preferential attachment models are suggested to possibly explain the formation of real-world networks such as the Internet, where these networks are in a `preasymptotic regime' (are of relatively small size) where the explosive nature of the model cannot be observed yet, based on statistical parameter estimation, simulations, and non-rigorous analysis~\cite{KrapKri08,KunBlatMos13,PhamSherShi16}.
	
	The inclusion of vertex-weights allows for a more heterogeneous and hence more realistic model, where different vertices may behave differently (in distribution), even when their out-degree is the same, as also discussed in the introduction.  The presence of vertex-weights often leads to rich behaviour where phase transitions based on the vertex-weight distribution can be observed (see the introduction for examples), which we show in this section to be case for the examples we consider here, too. 
	
	We study a number of examples for which we can apply the results in Section~\ref{sec:gen-cmj}. We state the  assumptions for the fitness function $f$ and the vertex-weight distribution, after which we present the results related to Theorems~\ref{thm:star-path-rif} and~\ref{thm:sub-treecount}. We conclude the section with a discussion of these results in Section~\ref{sec:discussion-phase-diagrams}, where we also provide some interesting  phase diagrams in Figures~\ref{fig:mixsub-tree} and~\ref{fig:addsub-tree}, and with some open problems in Section~\ref{sec:open}.
	
	\subsubsection{Assumptions for the fitness function and vertex-weight distribution}
	
	When stating the particular assumptions for the fitness function $f$ and the vertex-weight distribution, it is helpful to use the notions of \emph{slowly-varying} and \emph{regularly-varying} functions, which we recall in the following definition.
	
	\begin{definition}
		A measurable function $L: [0, \infty) \rightarrow [0, \infty)$ is said to be \emph{slowly varying} if for any $a > 0$ we have 
		\[
		\lim_{x \to \infty} \frac{L(ax)}{L(x)} = 1. 
		\]
		We say a measurable function $g:[0, \infty) \rightarrow [0, \infty)$ is \emph{regularly varying} with exponent $\beta \in \mathbb{R}$ if $g(x) = x^{\beta}L(x)$, where $L:[0,\infty)\to[0,\infty)$ is slowly varying. Finally, we say that a \emph{random variable} $W$ is regularly varying with exponent $z<0$ if the tail distribution $\P{W\geq x}$ is a regularly-varying function (in $x$) with exponent $z<0$. 
	\end{definition}
	
	We then assume that the fitness function satisfies the following assumption.
	
	\begin{assumption}[Fitness function]\label{ass:f} The fitness function $f$ is such that Equation~\eqref{eq:minfass} is satisfied with $w^*=0$.
		Furthermore, there exists $s:\N_0\to (0,\infty)$, which we call the \emph{degree function}, and continuous functions $g:[0,\infty)\to (0,\infty)$ and $h:[0,\infty)\to[0,\infty)$, which we call the \emph{weight functions}, such that $f$ satisfies
		\be \label{eq:f}
		f(i,w):=g(w)s(i)+h(w), \qquad i\in \N_0, w\geq 0.
		\ee 
		We then distinguish the following two cases, based on the weight functions.
		\begin{itemize}
			\item \hypertarget{additive}{\textbf{Additive weights.}}  $g\equiv 1$ and $h$ is regularly varying with exponent $1$.
			\item \hypertarget{mixed}{\textbf{Mixed weights.}} $g$ and $h$ are regularly varying with exponents $1$ and $\gamma\geq 0$, respectively.
		\end{itemize}
	\end{assumption}

	\begin{remark}
		The assumption that $w^*=0$ is not necessary, but used to simplify notation and computations. The results presented here follow equivalently for $w^*>0$ as well. {\Large \ensymboldremark}
	\end{remark}
	
	\begin{remark}
		The function $g$ and $h$ are regularly varying with exponents $1$ and $\gamma\geq 0$ in the mixed case; $g\equiv 1$ and $h$ is regularly varying with exponent $1$ in the additive case. The choice of the exponents is due to the fact that, when the vertex-weights are regularly varying with exponent $-(\alpha-1)<0$, then $g(W)$ and $h(W)$ are random variables that are regularly varying with exponents $-(\alpha-1)$ and $-(\alpha-1)/\gamma$, respectively (see Lemma~\ref{lemma:regexp} for details). Hence, changing the exponent of, for example, the regularly-varying function $g$ to $\zeta\neq 1$ in the mixed case, is equivalent to changing the exponent of the regularly-varying random variable $W$ from $-(\alpha-1)$ to $-(\alpha-1)/\zeta$ and changing the exponent of the regularly-varying function $h$ from $\gamma$ to $\gamma/\zeta$. As such, we take $\zeta=1$ without loss of generality. The function $h$ is regularly varying with exponent $1$  in the additive case without loss of generality for the same reason. 
		{\Large\ensymboldremark}
	\end{remark}
	
	\noindent Depending on the precise form of the degree function, the model behaviour markedly differs. We assume the \emph{degree function} $s$ satisfies the following assumption.
	\begin{assumption}[Degree function]\label{ass:deg}
		The degree function $s:\N_0\to(0,\infty)$ 
		satisfies one of the following cases.
		\begin{itemize} 
			\item \hypertarget{superlinear}{\textbf{Super-linear preferential attachment.}} $s$ is  regularly varying with exponent $p>1$.
			\item \hypertarget{barsuperlinear}{\textbf{Barely super-linear preferential attachment.}} $s$ is regularly varying with exponent $1$, such that $\sum_{i=0}^\infty s(i)^{-1}<\infty$. 
		\end{itemize}
		As a particular example of the barely super-linear case, we consider 
		\begin{itemize}
			\item \hypertarget{log-stretched}{\textbf{Barely super-linear log-stretched preferential attachment.}}  For some $\beta\in(0,1)$,
			\be \label{eq:slogstretched}
			s(i)=(i+1)\exp((\log(i+1))^\beta), \qquad i\in\N_0.
			\ee
		\end{itemize}
	\end{assumption}
	
	\begin{remark} 
		We note that these choices for the fitness function $f$ are not exhaustive, but do cover a wide range of examples. In particular, the weight types considered, i.e.\ additive or mixed weights, are common in the literature of linear preferential attachment with fitness (see e.g.~\cite{dereich-ortgiese,vdh-aging-mult-fitness-2017,bas,dereich-mailler-morters,rec-trees-fit,pref-att-neighbours}). When the vertex-weights are constant almost surely, the additive and mixed cases all fall into the same classes of super-linear preferential attachment. 
		The barely super-linear class has not been studied previously, as far as the authors are aware. {\Large \ensymboldremark}
	\end{remark}

	\noindent Finally, we require several assumptions on the distribution of the vertex-weights. For different choices of the fitness function $f$, different assumptions are required, which are summarised in the following overview.
	
	\begin{assumption}[Vertex-weight distribution]\label{ass:weights}
		The vertex-weights $(W_i)_{i\in\N}$ are i.i.d.\ and their tail distribution satisfies one (or more) of the following conditions.
		\begin{itemize}
			\item \textbf{Power law.} Let $\alpha>1$. We have the following two conditions. 
			\begin{enumerate} 
				\item There exist $\overline x>0$ and a slowly-varying function $\overline \ell:[0, \infty)\to [0, \infty)$ such that 
				\be \label{eq:weightasspowerlawub}
				\P{W\geq x}\leq \overline \ell(x)x^{-(\alpha-1)}, \qquad x\geq \overline x.
				\ee 
				\item There exist $ \underline x>0$ and a slowly-varying function $\underline \ell:[0, \infty)\to [0, \infty)$ such that 
				\be \label{eq:weightasspowerlawlb}
				\P{W\geq x}\geq \underline\ell(x)x^{-(\alpha-1)}, \qquad x\geq \underline x.
				\ee 
			\end{enumerate}
			\item \textbf{Log-stretched exponential.} Let $\nu>1$. We have the following two conditions.
			\begin{enumerate}
				\item There exist $\overline c,\overline x>0$ such that
				\be \label{eq:weightasslogstrechtedub}
				\P{W\geq x}\leq \e^{-\overline c(\log x)^\nu}, \qquad x>\overline x. 
				\ee 
				\item There exist $\underline c,\underline x> 0$ such that 
				\be \label{eq:weightasslogstrechtedlb}
				\P{W\geq x}\geq \e^{-\underline c(\log x)^\nu}, \qquad x\geq \underline x. 
				\ee 
			\end{enumerate}
		\end{itemize}
	\end{assumption}
	
	\subsubsection{Phase transitions for super-linear preferential attachment models with fitness}
	
	Based on these assumptions for the fitness function, degree function, and vertex-weight distribution, we can formulate the following theorem that treats the appearance of a unique infinite-degree vertex or a unique infinite path in $\cT_\infty$. 
	
	\begin{thm}
		\label{thrm:cmjexamples}
		Let $(\mathcal{T}_{i})_{i \in \mathbb{N}}$ be a $(W, f)$-recursive tree with fitness, where the fitness function $f$ and degree function $s$ satisfy one of the cases in Assumptions~\ref{ass:f} and~\ref{ass:deg}, respectively, and the vertex-weight distribution satisfies one of the cases in Assumption~\ref{ass:weights}. The tree $\mathcal T_\infty$ either contains a unique vertex with infinite degree and no infinite path almost surely, or contains a unique infinite path and no vertex with infinite degree almost surely, when the following conditions are met, based on the \emph{fitness function}, \emph{degree function}, and \emph{vertex-weight} assumptions: 
		\vspace{-7pt}
		\begin{table}[H]
			\centering
			\footnotesize
			\captionsetup{width=0.89\textwidth}
			\begin{tabular}{|c|c|l|l|}
				\hline
				\textbf{Fitness} & \textbf{Degree} &  \textbf{Star} & \textbf{Path} \\  
				\hline
				\hyperlink{mixed}{Mixed} & \hyperlink{superlinear}{Super-linear} & \eqref{eq:weightasspowerlawub} \& $(p-1)(\alpha-1)>\big(\gamma-\tfrac{\gamma-1}{p}\big)\vee 1$ & \eqref{eq:weightasspowerlawlb} \& $(p-1)(\alpha-1)<\big(\gamma-\tfrac{\gamma-1}{p}\big)\vee 1$ \\ 
				\hline 
				\hyperlink{additive}{Additive} & \hyperlink{superlinear}{Super-linear} & \eqref{eq:weightasspowerlawub} \& $p(\alpha-1)>1$ & \eqref{eq:weightasspowerlawlb} \& $p(\alpha-1)<1$ \\
				\hline
				\hyperlink{mixed}{Mixed} & \hyperlink{log-stretched}{Log-stretched} & \eqref{eq:weightasslogstrechtedub} \& $\beta\nu>1$ & \eqref{eq:weightasslogstrechtedlb} \& $\beta \nu < 1$ \\
				\hline  
			\end{tabular}
			\caption{\footnotesize The first column represents the form of the fitness function, as in Assumption~\ref{ass:f}; the second represents the form of the degree function, as in Assumption~\ref{ass:deg}. The third and fourth column, respectively, list the required assumptions on the vertex-weight distribution, as in Assumption~\ref{ass:weights}, together with the choices of the parameters that lead to either a unique node of infinite degree or a unique infinite path.}
		\end{table}
	\end{thm}
	\noindent The proof of Theorem~\ref{thrm:cmjexamples} appears in Section~\ref{sec:thm-star-phase}.
	
	\begin{remark}\label{rem:otherdistr}
		Though Theorem~\ref{thrm:cmjexamples} is presented for exponentially distributed inter-birth times, we expect the same results to hold for a large family of distributions, such that $X_w(uj)$ has mean $1/f(j-1,w)$ for each $u\in\cU_\infty,j\in\N$, and $w\in [0,\infty)$. In particular, we discuss the extension of Theorem~\ref{thrm:cmjexamples} to the following examples Appendix~\ref{sec:app}:
		\begin{enumerate}
			\item \textbf{Gamma distribution}: For each $u\in \cU_\infty,j\in\N$, and $w\in[0,\infty)$, the inter-birth time $X_w(uj)$ follows a $\mathrm{Gamma}(k,kf(j-1,w))$ distribution, for some $k>0$.
			\item \textbf{Beta distribution}: For each $u\in \cU_\infty,j\in\N$, and $w\in[0,\infty)$, the inter-birth time $X_w(uj)$ equals $\frac{\alpha+\beta}{\alpha}\frac{1}{f(j-1,w)}B(uj)$, where $(B(uj))_{u\in\cU_\infty,j\in\N}$ is a sequence of i.i.d.\ copies of a $\mathrm{Beta}(\alpha,\beta)$ random variable, for some $\alpha\geq 1$ and $\beta\leq (0,1]$.
			\item \textbf{Rayleigh distribution}: For each $u\in \cU_\infty,j\in\N$, and $w\in[0,\infty)$, the inter-birth time $X_w(uj)$ follows a $\mathrm{Rayleigh}(\sqrt{2/\pi}/f(j-1,w))$ distribution. {\Large\ensymboldremark }
		\end{enumerate}
	\end{remark}
	
	\begin{remark}
		In Section~\ref{sec:appadd} we discuss three more examples of barely super-linear preferential attachment with fitness: the~\hyperlink{log-stretched}{log-stretched} case with~\hyperlink{additive}{additive} fitness, as well as a~\hyperlink{poly-log}{poly-log} case with either additive or~\hyperlink{mixed}{mixed} fitness. Again, the extension of the results to the other inter-birth time distributions as in Remark~\ref{rem:otherdistr} apply here, too. 
		
		We do not include these results here, as we cannot prove a complete phase diagram, i.e.\ for certain parameter choices we cannot prove the appearance of a unique infinite-degree vertex nor a unique infinite path.  {\Large\ensymboldremark }
	\end{remark}
	
	\begin{remark}
		When the vertex-weights are almost surely bounded (or, in particular, a deterministic constant), it follows from the above theorem that in all cases, a unique vertex with infinite degree emerges in $\cT_\infty$ almost surely. Indeed, in such a case the vertex-weight distribution would satisfy, the upper bounds in~\eqref{eq:weightasspowerlawub} or~\eqref{eq:weightasslogstrechtedub} for any $\alpha>1,\nu>1$. As such, we can take $\alpha\to \infty,\nu\to\infty$ to conclude the claim. In the case $s(i)=(i+1)^p$ with $p>1$, this recovers the results of Oliveira and Spencer~\cite[Theorem $1.1$]{Oliveira-spencer}.{\Large\ensymboldremark }
	\end{remark}
	
	\begin{remark}
		It is interesting to note that, though two different techniques with distinct assumptions are used to prove Theorems~\ref{thm:star} and~\ref{thm:path} (for the existence of an infinite star or path in $\cT_\infty$), the application of these two general results in Theorem~\ref{thrm:cmjexamples} allows us to obtain a complete phase diagram for the three examples discussed here.{\Large\ensymboldremark }
	\end{remark}
	
	\noindent When the infinite tree $\cT_\infty$ contains a unique vertex with infinite degree almost surely, we can also quantify Theorem~\ref{thrm:sub-treecount}, in the sense that, depending on assumptions on the fitness function $f$ and vertex-weight distribution, we can determine almost surely whether or not $\cT_\infty$ contains an infinite number of copies of which kinds of sub-trees. We remark that we can do this in a relatively general manner, subject to the assumption that $\cT_\infty$ contains a unique vertex with infinite degree. That is, for \emph{any} degree function $s$ that satisfies the (barely) super-linear cases in Assumption~\ref{ass:deg} the results below apply. 
	
	To this end, we define for a finite tree $T$ and constant $z>0$,
	\be\label{eq:nus}
	G_1=G_1(T,z):=\sum_{v\in T} \outdeg{v,T} \ind{\outdeg{v,T}>z},\qquad G_2=G_2(T,z):=\sum_{v\in T}\ind{\outdeg{v,T}>z}.
	\ee 
	We can then have the following result.
	
	\begin{thm}\label{thrm:sub-treecount}
		Let $(\cT_i)_{i\in\N}$ be a $(W,f)$-recursive tree with fitness, where we assume that $f$ satisfies Assumption~\ref{ass:f}, that the degree function $s$ satisfies Assumption~\ref{ass:deg}, and that the vertex-weight distribution satisfies~\eqref{eq:weightasspowerlawub} and~\eqref{eq:weightasspowerlawlb} for some slowly-varying functions $\underline \ell,\overline\ell$, respectively, and some $\alpha>1$. Furthermore, we assume that $\cT_\infty$ contains a unique vertex with infinite degree, almost surely. Fix $k\in\N$ and let $T$ be a sibling-closed tree of size $k+1$. Then, we have the following.
		\begin{itemize}
			\item \textbf{\hyperlink{superlinear}{Super-linear case}}: We take $\gamma=1$ for~\hyperlink{additive}{additive} weights, $\gamma\geq 0$ for~\hyperlink{mixed}{mixed} weights, and set $z=(\alpha-1)/(\gamma\vee 1)$. The tree $T$ almost surely appears infinitely often as a sub-tree of $\cT_\infty$ when
			\be \label{eq:Tsumcond}
			p<1+\frac{1}{k-(G_1(T,z)-zG_2(T,z))}.
			\ee 
			The tree $T$ almost surely appears finitely often as a sub-tree of $\cT_\infty$ when
			\be \label{eq:Tnonsumcond}
			p>1+\frac{1}{k-(G_1(T,z)-zG_2(T,z))}.
			\ee 
			\item \textbf{\hyperlink{barsuperlinear}{Barely super-linear case}}: The tree $T$ appears as a sub-tree of $\cT_\infty$ infinitely often, almost surely.  
		\end{itemize}
	\end{thm}
	
	\noindent The proof of Theorem~\ref{thrm:sub-treecount} appears in Section~\ref{sec:sub-treeproof}, in particular, Section~\ref{sec:proof-of-sub-treecount-superlinear}.
	
	\begin{remark}
		In fact, we have a strengthening of Theorem~\ref{thrm:sub-treecount} which applies to \emph{any} $(W,f)$-recursive tree with fitness (in particular without the assumption that $W$ is real-valued), as long as, for any $j \in \mathbb{N}_{0}$, $f(j,W)$ satisfies~\eqref{eq:weightasspowerlawub} and~\eqref{eq:weightasspowerlawlb} for some slowly-varying functions $\overline \ell_j,\underline \ell_j$, respectively, and an exponent $z>0$, and for each $k\in\N_0$ there exists $i_k\in\{0,\ldots ,k\}$ such that $\inf_{i\leq k}f(i,w)=f(i_k,w)$ for all $w\in S$. See Proposition~\ref{prop:sumnuinfmean} for more details. {\Large\ensymboldremark }
	\end{remark}
	
	\begin{remark}
		When the vertex-weight distribution satisfies $\E{W^a}<\infty$ for all $a>0$, it directly follows from the assumptions on the fitness function $f$ in Assumption~\ref{ass:f} and Corollary~\ref{cor:sumnufinmean} that the results in Theorem~\ref{thrm:sub-treecount} extend to this case, where we set $G_1(T,z)=G_2(T,z)=0$. 
		
		Furthermore, we stress that for \emph{all}~\hyperlink{barsuperlinear}{barely super-linear} cases and for certain~\hyperlink{superlinear}{super-linear} cases (see the upcoming discussion), sub-trees $T$ of \emph{arbitrary} size appear infinitely often as a sub-tree of  $\cT_\infty$, almost surely. This is markedly different when compared to the super-linear preferential attachment model $f(i,w)=(i+1)^p$ ($p>1$) studied by Oliveira and Spencer in~\cite{Oliveira-spencer}, where only sub-trees $T$ with size at most $\lceil (p-1)^{-1}\rceil$ can appear infinitely often almost surely. {\Large\ensymboldremark }
	\end{remark}
	
	\subsubsection{Discussion related to Theorem~\ref{thrm:sub-treecount}: the super-linear case}
	\label{sec:discussion-phase-diagrams}
	Let us provide some intuition for Theorem~\ref{thrm:sub-treecount} by discussing two particular examples.\\[0.02cm]
	
	\noindent \emph{Super-linear degree, mixed weights.} We let $f(i,w)=(w+1)(i+1)^p$ for some $p>1$. That is, we consider the mixed weights case for the fitness function $f$, with $g(x):=x+1,h\equiv 0$, and $\gamma=0$, as in Assumption~\ref{ass:f}, and the super-linear case for the degree function $s$, i.e.\ $s(i):=(i+1)^p$, $p>1$, as in Assumption~\ref{ass:deg}. We require that $\cT_\infty$ contains a unique vertex with infinite degree almost surely, so that, by Theorem~\ref{thrm:cmjexamples}, we assume that the vertex-weight distribution satisfies~\eqref{eq:weightasspowerlawub} and that $(p-1)(\alpha-1)>1$. Then, additionally assume that the vertex-weight distribution satisfies~\eqref{eq:weightasspowerlawlb} with the same $\alpha>1$ but potentially with a different slowly-varying function. Now, Theorem~\ref{thrm:sub-treecount} states that a tree $T$ of size $k+1$, for some $k\in\N$, appears infinitely often as a sub-tree of $\cT_\infty$ almost surely when 
	\be 
	p<1+\frac{1}{k-(G_1(T,\alpha-1)-(\alpha-1)G_2(T,\alpha-1))}. 
	\ee 
	First, noting that the sum of all degrees equals $|T|-1=k$, we observe that
	\be \label{eq:Gspos}
	G_1(T,\alpha-1)-(\alpha-1)G_2(T,\alpha-1)\in[0,k),
	\ee
	for any choice of $T$ and $\alpha$, so that the upper bound yields a restriction on $p$. We omit the arguments of $G_1$ and $G_2$ from here on out for ease of writing. Combining our two assumptions we then require that
	\be 
	\frac{1}{\alpha-1}<p-1<\frac{1}{k-(G_1-(\alpha-1)G_2)},
	\ee 
	and we can only find $p>1$ that satisfy both inequality when 
	\be \label{eq:Gineq}
	k-(G_1-(\alpha-1)G_2)<\alpha-1.
	\ee 
	Now, if $k<\alpha-1$, there is no vertex in $T$ with an out-degree larger than $\alpha-1$, so that $G_1=G_2=0$ and the inequality is satisfied.  For $k\geq \alpha-1$ we distinguish two cases. $(i)$ There is no vertex in $T$ with a degree larger than $\alpha-1$. It again follows that $G_1=G_2=0$, so that the inequality in~\eqref{eq:Gineq} is \emph{not} satisfied; $(ii)$ There exists at least one  vertex with degree larger than $\alpha-1$. Then, $k-G_1>0$ since $k$ equals the sum of all degrees, whilst $(\alpha-1)(1-G_2)\leq 0$ (since $G_2\geq 1$), so that~\eqref{eq:Gineq} is not satisfied. 
	
	We thus conclude that, when $(p-1)(\alpha-1)>1$, only trees $T$ with size $k+1$, where $k<\alpha-1$, appear infinitely often. In particular, we do not require any assumptions on the \emph{structure} of such trees $T$; only their size is relevant. The reversed inequality in Theorem~\ref{thrm:sub-treecount} can be analysed in a similar manner to derive the phase diagram in Figure~\ref{fig:mixsub-tree}.\\[0.02cm]
	
	\noindent \emph{Super-linear degree, additive weights.} We let $f(i,w)=(i+1)^p+w$ for some $p>1$. That is, we consider the additive weight case for the fitness function $f$, with $h(x)=x$, as in Assumption~\ref{ass:f}, and the super-linear case for the degree function $s$, i.e.\ $s(i)=(i+1)^p,p>1$, as in Assumption~\ref{ass:deg}. We make the same assumptions as in the first example, except that we now require $p(\alpha-1)>1$. First, when $p$ is so large that $(p-1)(\alpha-1)>1$, we can derive the same conclusions as in the first example. When $p$ is such that 
	\be \label{eq:prange}
	\frac{1}{\alpha-1}\vee 1<p<1+\frac{1}{\alpha-1},
	\ee 
	different behaviour can be observed. Here, as we illustrate with the following particular family of trees, we observe the peculiar behaviour that the \emph{structure} of a tree $T$ plays a role in terms of whether it appears infinitely or finitely often as a sub-tree in $\cT_\infty$. 
	
	Let $T$ be an $m$-ary tree of size $k+1$ for some $m,k\in\N$ such that $k=\ell m$ for some $\ell\in\N$, i.e.\ a tree where the $\ell$ internal vertices (non-leaf vertices) have out-degree $m$ (note that stars are the particular case $\ell=1$). We observe that, by distinguishing the two cases $m > \alpha -1$ and $m \leq \alpha -1$ (where in the latter case $G_1 = G_2 = 0$),
	\be 
	G_1(T,\alpha-1)-(\alpha-1)G_2(T,\alpha-1)=\max\{k-(\alpha-1)\ell,0\}.
	\ee 
	As a result, recalling $k = \ell m$, an $m$-ary tree $T$ appears infinitely often as a sub-tree of $\cT_\infty$, almost surely, when 
	\be 
	\frac{1}{\alpha-1}\vee 1<p<1+\frac{1}{\ell \min\{\alpha-1,m\}}.
	\ee 
	Observe that for $\ell=1$, i.e.\ stars of size $k+1=m+1$, these inequalities are satisfied by any $p$ that satisfies~\eqref{eq:prange}, so that a star of \emph{any} size appears infinitely often as a sub-tree of $\cT_\infty$, almost surely, when $p$ satisfies~\eqref{eq:prange}. For $\ell\geq 2$, these inequalities can be satisfied only when $\alpha\leq 2$ and $\ell<1/(2-\alpha)$ or when $\alpha>2$. 
	
	Again, the reversed inequality in Theorem~\ref{thrm:sub-treecount} can be analysed in a similar manner to derive the phase diagram in Figure~\ref{fig:addsub-tree}, where some of the phases for ternary trees ($m=3$) are shown.
	
	We thus observe that it is possible for two trees $T_1, T_2$ of sizes $k_1+1<k_2+1$, respectively, to appear finitely and infinitely often as a sub-tree of $\cT_\infty$, almost surely. This behaviour is completely opposite to the  behaviour of the first example (or when $(p-1)(\alpha-1)>1$ in this example).
	
	As a final remark, any tree of any size appears infinitely often as a sub-tree of $\cT_\infty$ in the~\hyperlink{barsuperlinear}{barely sub-linear} case discussed in Theorem~\ref{thrm:sub-treecount}. As discussed in Section~\ref{sec:techniques-discussion-1}, this behaviour  has not been observed in explosive tree models studied so far.
	
	\begin{figure}[H]
		\centering
		\captionsetup{width=0.88\textwidth}
		\includegraphics[width=0.75\textwidth]{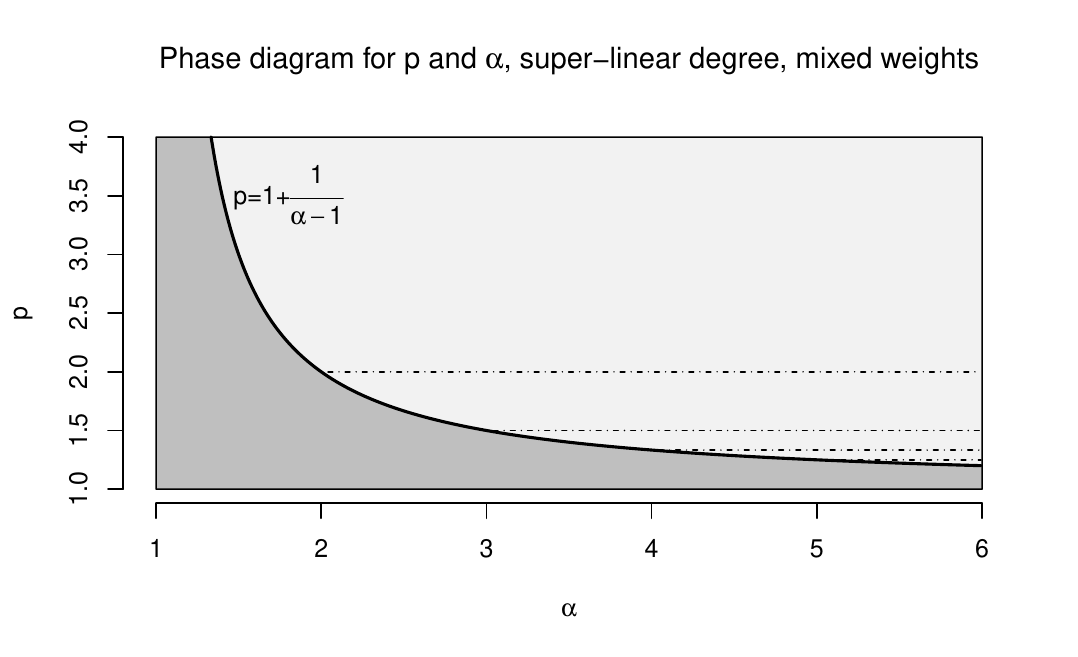}		\caption{\footnotesize The phase diagram for the case $f(i,w)=(w+1)(i+1)^p$. Below the curve $(p-1)(\alpha-1) = 1$ (the dark shaded region), the tree $\cT_\infty$ contains a unique infinite path, whereas above the curve (the light shaded region) the tree contains a unique vertex of infinite degree. Moreover, in the light shaded region, below the $k$th horizontal dotted line, corresponding to $1 + 1/k$, (with $k=1,2,3,4$ visible here), any tree of size $k+1$ appears as a sub-tree of a child of the node of infinite degree infinitely often, whereas above the $k$th horizontal line, in the light shaded region, trees of size $k+1$ appear only finitely often as a sub-tree of a child of the node of infinite degree. }
		\label{fig:mixsub-tree}
	\end{figure}
	
	\begin{figure}[H]
		\centering
		\captionsetup{width=0.88\textwidth}
		\includegraphics[width=0.75\textwidth]{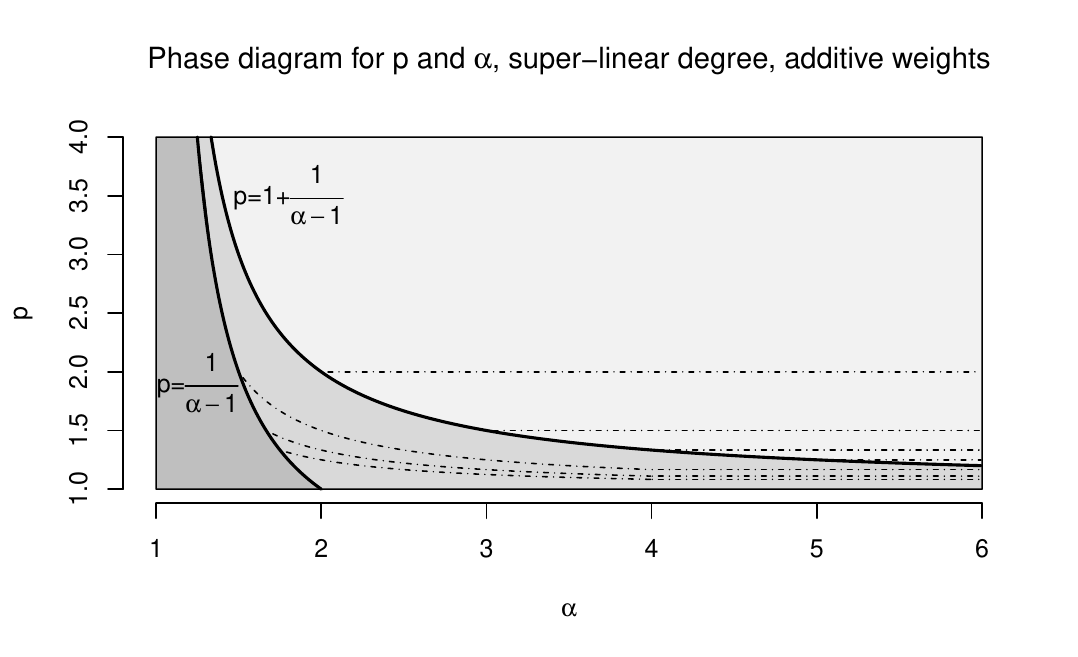}
		\caption{\footnotesize The phase diagram for the case $f(i,w)=(i+1)^p+w$. Below the curve $p(\alpha-1) = 1$ (the dark shaded region), the tree $\cT_\infty$ contains a unique infinite path, whereas above the curve (the lighter shaded regions) the tree contains a unique vertex of infinite degree.  In the area above the line $(p-1)(\alpha-1) = 1$ (lightest shaded region), below the $k$th horizontal dotted line, corresponding to $1 + 1/k$, (with $k=1,2,3,4$ visible here), any tree of size $k+1$ appears as a sub-tree of a child of the node of infinite degree infinitely often, whereas above the $k$th horizontal line, in the lightest shaded region, trees of size $k+1$ appear only finitely often as a sub-tree of a child of the node of infinite degree. In the area in between the curves $p(\alpha-1) =1$ and $(p-1)(\alpha-1)=1$ (the semi-dark region) a star of \emph{any} size appears as a sub-tree of a child of the node of infinite degree infinitely often. Additionally, below the $\ell$th curve in this region, corresponding to $(p-1)\ell\min\{\alpha-1,3\}=1$, (with $\ell=2,3,4$ visible here) ternary trees of size $3\ell+1$ appear as a sub-tree of the node of infinite degree infinitely often, whereas above the $\ell$th curve, such trees only appear finitely often. }\label{fig:addsub-tree}
	\end{figure}

	\subsection{Proof techniques} \label{sec:techniques-discussion-2}
	
	The proofs of the results in Section~\ref{sec:examples} generally apply the results of Section~\ref{sec:results}. However, we are unaware of previous proofs of~Lemma~\ref{lem:lower-bound-infinite-events}, used to prove Theorem~\ref{thm:star-path-rif}. Moreover, exploiting the memory-less property of the exponential distribution allows the derivation of a necessary and sufficient condition for the emergence of a structure infinitely often when $\cT_\infty$ contains an infinite star, as in Theorem~\ref{thm:sub-treecount}. The proof of this theorem, we believe, is more elegant than the approach used to prove~\cite[Theorem~1.2]{Oliveira-spencer} for the particular case $f(i,w)=(i+1)^p, p>1$. The assumptions made in Section~\ref{sec:gen-cmj} allow us to deduce a variety of phase-transitions in applied models (cf.\ Theorem~\ref{thrm:cmjexamples}), which we believe extend to a fairly general family of distributions (see Remark~\ref{rem:otherdistr}). 
	
	We prove the most general results, as presented in Section~\ref{sec:recursive}, in Section~\ref{sec:applications}, and prove the results for the examples presented in Section~\ref{sec:gen-cmj} in Section~\ref{sec:examplesproof}.

	\subsection{Open problem} \label{sec:open}
	
	It is unclear whether or not \emph{any} explosive $(X,W)$-CMJ process for which the vertex-weights are almost surely constant \emph{always} yields an infinite star almost surely. In other words, when assuming that the $(X(i))_{i \in \mathbb{N}}$ are mutually independent and positive, we have the following open problem. 
	\begin{open}\label{con:infstar}
		Consider an explosive $(X,W)$-CMJ process $(\mathscr T_t)_{t\geq 0}$, such that $(X(i))_{i \in \mathbb{N}}$ are mutually independent and positive. Is it the case that almost surely  $\mathcal  T_{\infty}$ contains a unique vertex with infinite degree?
	\end{open}
	
	\begin{remark}
		In the case of the recursive tree with fitness, if $X(i) \sim \Exp{g(i)}$, with $g$ unbounded, and convex, we believe that the affirmative of Open problem~\ref{con:infstar} holds, by an argument using a combination of the result of Galashin~\cite[Theorem~1]{galashin2014existence} and Proposition~\ref{prop:finite-l-moderate}. {\Large\ensymboldremark }
	\end{remark}
	
	\begin{remark}
		Note that the counter-example in Theorem~\ref{thm:counter-example} relies on the dependence of the $(X(i))_{i \in \mathbb{N}}$ on the weights. Thus, if $\mathcal{T}_{\infty}$ \emph{does} almost surely contain a unique node of infinite degree in Open problem~\ref{con:infstar}, one may interpret this, informally, as saying that `new nodes are unable to out-compete older nodes, without the influence of a random weight'. {\Large\ensymboldremark }
	\end{remark}
	
	\section{Proofs of main results}\label{sec:proofmain}
	
	This section is dedicated to proving the most general results, as presented in Section~\ref{sec:results}. We prove the existence of an infinite star (cf. Theorem~\ref{thm:star}) in Section~\ref{sec:starproof}, prove the existence of of an infinite path (cf.\ Theorem~\ref{thm:path}) and the structural result of sub-trees in the star regime (cf.\ Theorem~\ref{thm:structure}) in Section~\ref{sec:inf-path-structure-proof}, and finally prove the uniqueness properties (cf.\ Theorem~\ref{thm:uniqueness}) in Section~\ref{sec:uniqueproof}.
	
	\subsection{Sufficient criteria for a star}\label{sec:starproof}
	
	This section is dedicated to the proof of Theorem~\ref{thm:star}. We first have the following lemma:
	\begin{lemma} \label{lem:moment-prob-bounds}
		Fix $k,n\in\N$, let $(a_{1}, \ldots, a_{k}) \in \mathbb{N}^{k}$. Let $(Z_i)_{i \in [k]}$ be independent random variables, and let $(X(\ell))_{\ell\geq n+1}$ be as in~\eqref{eq:cmj-assumption}, satisfying Conditions~\ref{item:stardom} and~\ref{item:starlimsup} of Assumption~\ref{ass:star}. Then, there exists $n_{0} > 0$ 
		and $C=C(n_0)>0$ such that, for all $n \geq n_{0}$,
		\begin{equation} \label{eq:moment-prob-bound1}
			\Prob{\sum_{i=1}^{k} Z_{i} \leq \sum_{\ell = n + 1}^{\infty} X(\ell)} \leq \Prob{\sum_{i=1}^{k} Z_{i} \leq Y_{n}} \leq  
			C \prod_{j=1}^{k} \mathcal{L}_{c\mu_{n}^{-1}}(Z_i),
		\end{equation}
		where $Y_{n}$ is as in~\eqref{eq:stochastic-bound} and independent of the $(Z_i)_{i \in [k]}$.
	\end{lemma}
	\begin{proof}
		First note that the first inequality in~\eqref{eq:moment-prob-bound1} is an immediate consequence of~\eqref{eq:stochastic-bound}. Then, for any $\lambda > 0$, 
		\begin{linenomath}
			\begin{align*}
				\Prob{\sum_{i=1}^{k} Z_{i} \leq Y_{n}} 
				= \Prob{\exp\left(\lambda Y_{n} - \lambda\sum_{i=1}^{k} Z_i\right) \geq 1}
				\leq \mathcal{M}_{Y_{n}}(\lambda) \prod_{i=1}^{k} \mathcal{L}_{\lambda}(Z_{i}),
			\end{align*}
		\end{linenomath}
		where the last inequality uses Markov's inequality and the independence of the random variables. Next, setting $\lambda:= c\mu_{n}^{-1}$ and recalling that, by~\eqref{eq:limsup-mgf}, we can set $C:=\limsup_{n \to \infty} \mathcal{M}_{Y_{n}}(c\mu_{n}^{-1}) < \infty$, we deduce the result.   
	\end{proof}
	
	\noindent We now introduce the following terminology, used in the remainder of the section, which, although not strictly needed, we believe makes the proofs conceptually easier to understand. For $a, b \in \cU_{\infty}$ we say that 
	\[
	\text{``$a$ has at least $k$ children before $b$ explodes''}
	\]
	if $\mathcal{B}(a) + \mathcal{P}_{k}(a) < \mathcal{B}(b) + \mathcal{P}(b)$. We say that 
	\[
	\text{``$a$ explodes before all of its ancestors''}
	\]
	if, for each $\ell < |a|$, we have $\mathcal{B}(a) + \mathcal{P}(a) < \mathcal{B}(a_{|_\ell}) + \mathcal{P}(a_{|_\ell})$. Finally, for $a \in \cU_{\infty}$ with $|a| \geq 1$, we say that $a= a_{1} \cdots a_{m}$ is $a_{1}$-\emph{conservative} if, for each $j \in \left\{2, \ldots, m\right\}$, we have $a_{j} \leq a_{1}$. (Note that this implies that any $a$ such that $|a| = 1$ is $a_1$-conservative.)
	
	\begin{lemma} \label{lem:cons-bound-gen}
		Under Assumption~\ref{ass:star}, there exist $\eta < 1$ and $K = K(\eta) > 0$ such that for all $a_1 > K(\eta)$, all integers $m\in \mathbb{N}$, and some constant $C>0$, 
		\be
		\sum_{\substack{a:|a| = m\\ a\text{ is $a_1$-conservative}}}\!\!\!\!\!\!\! \Prob{a \text{ has at least $a_1$ children before } \varnothing \text{ explodes }} \leq C\eta^{m-1} \E{\cL_{c\mu_{a_1}^{-1}}(\mathcal{P}_{a_{1}};W)}.
		\ee 
	\end{lemma}
	
	\begin{proof}
		Suppose that $a = a_1 \cdots a_m \in \cU_\infty$. For $a$ to have at least $a_1$ children before the explosion of $\varnothing$, in particular, each of the births corresponding to the ancestors of $a$ need to occur (leading to a term as in Equation~\eqref{eq:birth-time-identity} with $u=\varnothing$ and $v=a$). Thus, for $a_{1}$ sufficiently large, by~\eqref{eq:birth-time-identity} and Lemma~\ref{lem:moment-prob-bounds}, we have
		\be\ba  \label{eq:root-prob}
		\Prob{a \text{ has at least $a_1$ children before } \varnothing \text{ explodes }}
		&= \Prob{\mathcal{B}(a) + \mathcal{P}_{a_1}(a) <  \mathcal{P}(\varnothing)}
		\\ & = \Prob{\left(\sum_{j=1}^{m-1}\mathcal{P}_{a_{j+1}}(a_{|_j})\right) + \mathcal{P}_{a_1}(a) \leq\sum_{k= a_{1} + 1}^{\infty} X(k)} 
		\\ &
		\stackrel{\eqref{eq:moment-prob-bound1}}{\leq} C \prod_{j=1}^{m-1} \E{\mathcal{L}_{c\mu_{a_{1}}^{-1}}(\mathcal{P}_{a_{j+1}}; W)} \times \E{\mathcal{L}_{c\mu_{a_{1}}^{-1}}(\mathcal{P}_{a_{1}}; W)},
		\ea\ee
		where the last line follows from the fact that, by~\eqref{eq:cmj-assumption}, for each $u \in \cU_{\infty}$ the sequence $(\mathcal{P}_{j}(u))_{j\in \mathbb{N}}$ is independent and distributed like $(\mathcal{P}_{j}(\varnothing))_{j\in \mathbb{N}}$. When we sum over the possible conservative sequences $a$ that are $a_1$-conservative, each $a_{j}$ takes values between $1$ and $a_1$, for $j=2,\ldots, m$. Thus, 
		\be\ba \label{eq:sumbound}
		\sum_{\substack{a:|a| = m\\ a\text{ $a_1$-conservative}}}\!\!\!\!\!\!\!\!\!\!\!{}&\Prob{a \text{ has at least $a_1$ children before } \varnothing \text{ explodes }}
		\\ \leq{}& \sum_{a_{2} = 1}^{a_{1}} \sum_{a_{3} = 1}^{a_{1}} \cdots \sum_{a_{m}=1}^{a_1}  C \prod_{j=1}^{m-1} \E{\mathcal{L}_{c\mu_{a_{1}}^{-1}}(\mathcal{P}_{a_{j+1}}; W)} \times \E{\mathcal{L}_{c\mu_{a_{1}}^{-1}}(\mathcal{P}_{a_{1}}; W)}
		\\ = {}& C \left(\sum_{\ell = 1}^{a_{1}} \E{\mathcal{L}_{c\mu_{a_{1}}^{-1}}(\mathcal{P}_{\ell}; W)} \right)^{m-1}  \E{\cL_{c\mu_{a_{1}}^{-1}}(\mathcal{P}_{a_{1}}; W)}. 
		\ea\ee 
		We now need only show that for $a_1$ sufficiently large,
		\[
		\sum_{\ell=1}^{a_1}\E{\mathcal{L}_{c\mu^{-1}_{a_{1}}}(\mathcal{P}_{\ell}; W)} < \eta. 
		\]
		Indeed, since $\sum_{i=\ell}^{\infty}\E{\mathcal{L}_{c\mu^{-1}_{\ell}}(\mathcal{P}_{\ell}; W)} < \infty$ by~\eqref{eq:laplacesum} in Assumption~\ref{ass:star},  there exists $L = L(\eta) > 0$ such that, for all $a_1 > L$,
		\begin{equation} \label{eq:bound-part-1}
			\sum_{\ell = L}^{a_1} \E{\mathcal{L}_{c\mu^{-1}_{a_{1}}}(\mathcal{P}_{\ell}; W)} < \sum_{\ell = L}^{\infty} \E{\mathcal{L}_{c\mu^{-1}_{\ell}}(\mathcal{P}_{\ell}; W)} < \frac{\eta}{2},
		\end{equation}
		where the inequality uses the fact that $c\mu_{n}^{-1}$ is non-decreasing in $n$. On the other hand, since $\lim_{n \to \infty} c\mu_{n}^{-1} = \infty$, by bounded convergence (bounding the integrand by $1$) we have 
		\[
		\lim_{n \to \infty} \sum_{\ell = 1}^{L-1}\E{\mathcal{L}_{c\mu^{-1}_{n}}(\mathcal{P}_{\ell}; W)} = 0.
		\]
		As a result, for some $K \geq L$ sufficiently large and for all $a_{1} > K$, we arrive at 
		\begin{equation} \label{eq:bound-part-2}
			\sum_{\ell = 1}^{L-1} \E{\mathcal{L}_{c\mu^{-1}_{\ell}}(\mathcal{P}_{\ell}; W)} < \frac{\eta}{2}.
		\end{equation}
		Combining Equations~\eqref{eq:bound-part-1} and~\eqref{eq:bound-part-2} in~\eqref{eq:sumbound}, we conclude the proof.  
	\end{proof}
	
	\noindent The above lemma provides an upper bound for the probability of the event that a vertex $a$ explodes before the root of the tree, in the case that $a$ is $a_1$-\emph{conservative}. However, when $a$ does not satisfy this condition, we can view $a$ as a concatenation of a number of conservative sequences. That is, we write $a=\overline b_1\cdots \overline b_\ell$, where $\overline b_i=b_{i1}\ldots b_{im_i}$ for each $i\in[\ell]$ and for some $\ell\in\N, (m_i)_{i\in[\ell]}\in \N^\ell$, and $(b_{i,j})_{i\in[\ell],j\in[m_i]}$, such that $\overline b_i$ is $b_{i1}$-conservative for each $i\in[\ell]$. By the independence of birth processes of distinct individuals (or in fact, the independence of disjoint sub-trees) by Equation~\eqref{eq:starindep}, we are able to apply Lemma~\ref{lem:cons-bound-gen} to each conservative sequence in the concatenation to arrive at a bound for the expected number of individuals that explode before all its ancestors.
	
	\begin{prop} \label{prob:local-explosions}
		Under Assumption~\ref{ass:star}, there exists $K'>0$ sufficiently large, such that 
		\be 
		\E{\left| \left\{a\in \cU_\infty: a_1>K', a\text{ explodes before all its ancestors} \right\} \right|}<\infty.
		\ee 
	\end{prop}
	
	\begin{proof}
		As explained before the proposition statement, we think of sequences $a\in\cU_\infty$ as a concatenation of conservative sequences. Let $a=a_1\ldots a_m$ be a sequence of length $m\in\N$, and assume that there exist $k\in[m]$ and indices $I_1<I_2<\ldots<I_k$ such that $a_1=:a_{I_1}<a_{I_2}<\ldots <a_{I_k}$. That is, the $I_j$ are the indices of the running maxima of the sequence $a$. For brevity of notation, we also set $I_{k+1} := m + 1$, and set $a_{0} = \varnothing$.
		
		To show that $a$ explodes before all its ancestors, we think of $a$ as a concatenation of the conservative sequences $a_{I_j}\cdots a_{I_{j+1}-1}$, with $j\in[k]$. By the fact that, for each $u \in \cU_{\infty}$, the sequence $(\mathcal{P}_{i}(u))_{i \in \mathbb{N}}$ is independent and distributed like $(\mathcal{P}_{i}(\varnothing))_{i \in \mathbb{N}}$, each of these conservative sequences can be seen as corresponding to an $a_{I_j}$-conservative individual rooted at $a_{I_j-1}$, $j\in[k]$. We can thus apply Lemma~\ref{lem:cons-bound-gen} to all these concatenated sequences. 
		
		Since, by definition we have $a_{I_{\ell +1}} >  a_{I_{\ell}}$, applying a similar logic to~\eqref{eq:root-prob} we have the following inclusion: 
		\begin{linenomath}
			\begin{align} \label{eq:ea}
				E_a:={}&\{a\text{ explodes}\text{ before any of its ancestors explodes}\}\\
				& \bigcap_{\ell=1}^{k}\{a_1\cdots a_{I_{\ell+1}-1} \text{ gives birth to at least } a_{I_{\ell}}\text{ children before }a_1\cdots a_{I_{\ell}-1}\text{ explodes}\}
				\\  ={}& \bigcap_{\ell=1}^{k}\left\{\left(\sum_{j=I_{\ell}}^{I_{\ell+1}-2}\mathcal{P}_{a_{j+1}}(a_{|_j})\right) + \mathcal{P}_{a_{I_{\ell}}}(a_{|_{I_{\ell+1}-1}}) \leq \sum_{i= a_{I_{\ell}} + 1}^{\infty} X(a_1 \cdots a_{I_{\ell}-1} i)\right\}
				=:\bigcap_{\ell=1}^k E_{a,\ell}.
			\end{align}
		\end{linenomath}
		Now, note that the events $(E_{a, \ell}, \ell \in [k])$ are not independent, since, for a given $\ell$, the term $\mathcal{P}_{a_{I_{\ell}}}(a_{|_{I_{\ell+1}-1}})$ appearing in $E_{a, \ell}$ may be correlated with the term $\sum_{i= a_{I_{\ell+1}} + 1}^{\infty} X(a_1 \cdots a_{I_{\ell+1}-1} i)$ appearing in $E_{a, \ell+1}$. However, by the third condition of Assumption~\ref{ass:star}, these events are conditionally independent, given the weights of $a$ and all its ancestors, $W_\varnothing, W_{a_1},W_{a_1a_2},\ldots, W_a$. Thus, 
		\begin{linenomath}
			\begin{align*}
				\mathbb P{}&(E_a \, | \, W_\varnothing, W_{a_1},W_{a_1a_2},\ldots, W_a) 
				\\ & = \prod_{\ell=1}^{k} \Prob{\left(\sum_{j=I_{\ell}}^{I_{\ell+1}-2}\mathcal{P}_{a_{j+1}}(a_{|_j})\right) + \mathcal{P}_{a_{I_{\ell}}}(a_{|_{I_{\ell+1}-1}}) \leq \sum_{i= a_{I_{\ell}} + 1}^{\infty} X(a_1 \cdots a_{I_{\ell}-1} i) \, \bigg | \, W_\varnothing, W_{a_1},W_{a_1a_2},\ldots, W_a}
				\\ &  \stackrel{\eqref{eq:stochastic-bound}}{\leq} \prod_{\ell=1}^{k} \Prob{\left(\sum_{j=I_{\ell}}^{I_{\ell+1}-2}\mathcal{P}_{a_{j+1}}(a_{|_j})\right) + \mathcal{P}_{a_{I_{\ell}}}(a_{|_{I_{\ell+1}-1}}) \leq Y^{(a_1 \cdots a_{I_{\ell}-1})}_{a_{I_{\ell}}} \, \bigg | \, W_\varnothing, W_{a_1},W_{a_1a_2},\ldots, W_a} 
				\\ &  := \prod_{\ell=1}^{k} \Prob{\wt E_{a, \ell} \, \Big| \, W_\varnothing, W_{a_1},W_{a_1a_2},\ldots, W_a}, 
			\end{align*}
		\end{linenomath}
		where each $Y^{(a_1 \cdots a_{I_{\ell}-1})}_{a_{I_{\ell}}}$ is independent and distributed like $Y_{a_{I_{\ell}}}$. Now, each of the terms $\wt E_{a, \ell}$ are independent, as the depend on different weights. Hence, so are each of the terms appearing in the above product, so that
		\begin{equation} \label{eq:probsplitineq}
			\P{E_a}\leq \prod_{\ell=1}^k \P{\wt E_{a,\ell}}.    
		\end{equation}
		We now let $d_j:=I_{j+1}-I_j-1$ for $j\in[k-1]$ and $d_k:=m-I_k$ denote the number of entries between the running maxima in the sequence $a$. We can then define, for $(d_j)_{j\in[k]}\in\N_0^k$ (and with the convention that $[0]$ is the empty set),
		\be 
		\mathscr P_k(a_{I_1},a_{I_2},\ldots, a_{I_k}, d_1, \ldots, d_k):=\{a\in \cU_\infty: \text{ For all }j\in\{1,\ldots, k\}\text{ and all }i\in[d_j],\ a_{I_j+i}\in[a_{I_j}]\}
		\ee 
		as the set of all sequences $a$ with running maxima $a_1=a_{I_1},\ldots, a_{I_k}$ and $d_j$ many entries between the $j^{\text{th}}$ and $(j+1)^{\text{th}}$ maximum. For ease of writing, we omit the arguments of $\mathscr P_k$. We then write the expected value of the number of individuals $a\in\cU_\infty$ that explode before all their ancestors such that $a_1>K'$ as
		\be\ba \label{eq:splitsum1}
		\sum_{\substack{ a\in \cU_\infty \\ a_1>K'}}\P{E_a}=\sum_{m=1}^\infty \sum_{\substack{a: |a|=m\\ a_1>K'}}\P{E_a}\leq \sum_{m=1}^\infty \sum_{k=1}^m \sum_{a_{I_k}>\ldots >a_{I_1}>K'}\sum_{\substack{(d_\ell)_{\ell\in[k]}\in \N_0^k\\ \sum_{\ell=1}^k d_\ell=m-k}}\sum_{a\in \mathscr P_k}\prod_{\ell=1}^k\P{\wt E_{a,\ell}}.
		\ea\ee 
		In the first step, we introduce a sum over all sequence lengths $m$. In the second step, we furthermore sum over the number of running maxima $k$, the values of the running maxima $a_{I_1}, \ldots a_{I_k}$, the number of entries $d_\ell$ between each maxima $I_\ell$ and $I_{\ell+1}$ (or between $I_k$ and $m$ if $\ell=m$), and all sequences $a\in\mathscr P_k$ that admit such running maxima and inter-maxima lengths. Moreover, we use~\eqref{eq:probsplitineq} to bound $\P{E_a}$ from above, now that we know the number of running maxima in $a$.
		
		We can now take the sum over $a\in\mathscr P_k$ into the product, due to the fact that we can decompose each sequence  $a\in\mathscr P_k(a_{I_1},\ldots, a_{I_k},d_1,\ldots d_k)$ into a concatenation of sequences $a^{(1)}\ldots a^{(k)}$, with $a^{(\ell)}:=a_{I_\ell}\cdots a_{I_{\ell+1}-1}\in\mathscr P_1(a_{I_\ell},d_\ell)$ for each $\ell\in[k]$. This yields, for $a_{I_1},\ldots, a_{I_k}$ and $d_1, \ldots, d_k$ fixed, 
		\be 
		\sum_{a\in\mathscr P_k}\prod_{\ell=1}^k \P{\wt E_{a,\ell}}=\prod_{\ell=1}^k\!\!\!\! \sum_{\substack{a^{(\ell)}:|a^{(\ell)}|=d_\ell+1\\ \text{$a^{(\ell)}$ $a_{I_\ell}$-conservative}}}\!\!\!\!\!\!\!\!\!\!\!\P{\wt E_{a,\ell}}.
		\ee 
		We can then directly apply Lemma~\ref{lem:cons-bound-gen} to each of the sums in the product to obtain, for some $\eta<1$ and with $C$ sufficiently large, the upper bound 
		\be 
		\prod_{\ell=1}^k \!\!\!\!\sum_{\substack{a^{(\ell)}:|a^{(\ell)}|=d_\ell\\ \text{$a^{(\ell)}$ $a_{I_\ell}$-conservative}}}\!\!\!\!\!\!\!\!\!\!\!\P{\wt E_{a,\ell}}\leq \prod_{\ell=1}^k C\eta^{d_\ell} \E{\cL_{c\mu_{a_{I_\ell}}^{-1}}(\mathcal{P}_{a_{I_\ell}};W)}=C^k\eta^{\sum_{\ell=1}^k d_\ell}\prod_{\ell=1}^k \E{\cL_{c\mu_{a_{I_\ell}}^{-1}}(\mathcal{P}_{a_{I_\ell}};W)}.
		\ee 
		We substitute this in~\eqref{eq:probsplitineq} to arrive at
		\be\ba 
		\sum_{m=1}^\infty{}& \sum_{k=1}^m \sum_{a_{I_k}>\ldots >a_{I_1}>K'}\sum_{\substack{(d_j)_{j\in[k]}\in \N_0^k\\ \sum_{\ell=1}^k d_\ell=m-k}}C^k\eta^{\sum_{\ell=1}^k d_\ell}\prod_{\ell=1}^k \E{\cL_{c\mu_{a_{I_\ell}}^{-1}}(\mathcal{P}_{a_{I_\ell}};W)}\\ 
		&\leq \sum_{m=1}^\infty \sum_{k=1}^m \binom{m-1}{k-1}C^k\eta^{m-k}\sum_{a_{I_1}>K'}\cdots \sum_{a_{I_k}>K'}\prod_{\ell=1}^k \E{\cL_{c\mu_{a_{I_\ell}}^{-1}}(\mathcal{P}_{a_{I_\ell}};W)}\\
		&=\sum_{m=1}^\infty \sum_{k=1}^m \binom{m-1}{k-1}C^k\eta^{m-k}\left(\sum_{a> K'} \E{\cL_{c\mu_a^{-1}}(\mathcal{P}_a;W)}\right)^k.
		\ea\ee 
		By~\eqref{eq:bound-part-1} we can bound the innermost sum from above by $\eta/2$ when $K'$ is sufficiently large, so that we obtain the upper bound 
		\be 
		\sum_{m=1}^\infty \sum_{k=1}^m \binom{m-1}{k-1}C^k\eta^m 2^{-k}=\frac{C}{2}\sum_{m=1}^\infty  (1+C/2)^{m-1} \eta^m<\infty, 
		\ee 
		where the last step follows when $\eta<(1+C/2)^{-1}$, which holds by choosing $K'$ sufficiently large, as follows from~\eqref{eq:bound-part-1} and~\eqref{eq:bound-part-2}.
	\end{proof}
	
	\noindent The following proposition requires the following notation:
	\[
	\mathcal{U}_{L} := \left\{u \in \mathcal{U}_{\infty}: u_{i} \leq L, i \in [|u|] \right\} \cup \{\varnothing\}.
	\]
	Following~\cite{Oliveira-spencer}, we say elements of $\mathcal{U}_{L}$ are $L$-\emph{moderate}. We also let $\mathcal{U}^{c}_{L}$ denote the set complement of $\mathcal{U}_{L}$.
	
	\begin{prop} \label{prop:finite-l-moderate}
		Let $(\mathscr{T}_{t})_{t \geq 0}$ be a $(X,W)$-CMJ branching process satisfying Condition~\ref{item:starnonzero} of Assumption~\ref{ass:star}.  Then, almost surely, for all $t \in (0,\infty)$, we have
		\[
		\left|\{u \in \mathscr{T}_{t}: u \, \text{ is $L$-moderate} \}\right| < \infty.
		\]
	\end{prop}
	
	\begin{proof}
		First, let $(X^{(L)}, W^{(L)})$ denote the distribution of an auxiliary CMJ branching process, where, if the symbol `$\sim$' denotes equality in distribution, we have $W^{(L)} \sim W$ and 
		\begin{equation}
			(X^{(L)}(j), W^{(L)}) \sim \begin{cases}
				(X(j), W) & \text{ if $j \leq L$,} \\
				(\infty, W) & \text{otherwise.}
			\end{cases}
		\end{equation}
		In other words, a $\left(X^{(L)}, W^{(L)} \right)$-CMJ process is truncated to ensure that no node produces more than $L$ children. Moreover, if $\mathcal{B}_{L}(u)$ denotes the distribution of the random variable $\mathcal{B}(u)$ under the distribution of the $(X^{(L)}, W^{(L)})$-CMJ process, this definition ensures that if $u \in \mathcal{U}_{L}$ then $\mathcal{B}_{L}(u) \sim \mathcal{B}(u)$. Now, note that, if $(\mathscr{T}^{(L)}_{t})_{t \geq 0}$ denotes an $(X^{(L)}, W^{(L)})$-CMJ process, for each $t \in (0, \infty)$ we have $\E[(X^{(L)}, W^{(L)})]{\xi(t)} \leq L$, and, by~\eqref{eq:non-instantaneous-explosion}, also $\E[(X^{(L)}, W^{(L)})]{\xi(0)} < 1$. Therefore, by~\cite[Theorem~3.1(b)]{Komjathy2016ExplosiveCB}, $(\mathscr{T}^{(L)}_{t})_{t \geq 0}$ is \emph{conservative}, i.e.\ almost surely, for each $t \in (0, \infty)$,
		\begin{equation} \label{eq:non-exposive-L}
			\left|\mathscr{T}^{(L)}_{t}\right| < \infty.
		\end{equation}
		We now construct a coupling of a $(X,W)$-CMJ branching process $(\overline{\mathscr{T}}_{t})_{t \geq 0}$ with a $(X^{(L)}, W^{(L)})$-CMJ process $(\overline{\mathscr{T}}^{(L)}_{t})_{t \geq 0}$. Note that, by definition, for each $L$-moderate $u \in \mathcal{U}_{L}$, we have $X^{(L)}(u) \sim X(u)$. We then construct $(\overline{\mathscr{T}}_{t})_{t \geq 0}$ in the natural way from the random variables defining $(\overline{\mathscr{T}}^{(L)}_{t})_{t \geq 0}$: for each $u \in \mathcal{U}_{L}$, we set $(X(u), W_{u}) = (X^{(L)}(u), W^{(L)}_{u})$, and for each $u \in \mathcal{U}^{c}_{L}$, with $|u| = m \geq 1$, sample $X(u)$ independently, conditionally on the weight $W_{u_{|_{m-1}}}$. One readily verifies that this coupling has the correct marginal distributions. Moreover, on this coupling, almost surely, for each $t \geq 0$, we have 
		\[
		\left|\{u \in \overline{\mathscr{T}}_{t}: u \, \text{ is $L$-moderate}\} \, \right| = \left|\overline{\mathscr{T}}^{(L)}_{t}\right| \stackrel{\eqref{eq:non-exposive-L}}{<} \infty,
		\]
		as desired
	\end{proof}
	
	\noindent Recall that, for a $(X,W)$-Crump-Mode-Jagers branching process $(\mathscr{T}_{t})_{t \geq 0}$, we have  
	\[\tau_{\infty} = \lim_{k \to \infty} \tau_{k} = \inf\left\{t > 0: |\mathscr{T}_t| = \infty\right\}.\]
	Recall also that we have $\mathcal{T}_{\infty} = \bigcup_{k=1}^{\infty} \mathcal{T}_{k} = \bigcup_{k=1}^{\infty} \mathscr{T}_{\tau_{k}}$. 
	\begin{lemma} \label{lemma:tree-ident}
		Let $(\mathscr{T}_{t})_{t \geq 0}$ be  a $(X,W)$-CMJ branching process satisfying Condition~\ref{item:starnonzero} of Assumption~\ref{ass:star}. Then,
		\begin{equation} \label{eq:inclusion}
			\mathcal{T}_{\infty} = \left\{u \in \mathcal{U}_{\infty}: \mathcal{B}(u) < \tau_{\infty}\right\} \subseteq \mathscr{T}_{\tau_{\infty}}.
		\end{equation}    
	\end{lemma}
	\begin{proof}
		The second inclusion is clear, hence we just prove the first equality. If $|\mathcal{T}_{\infty}| < \infty$, in other words, the process becomes extinct, then $\tau_{\infty} = \infty$ and the equality is clear. Otherwise, on the event $\left\{|\mathcal{T}_{\infty}| = \infty\right\}$, since by~\eqref{eq:expl-almost-surely} we have $\tau_{\infty} < \infty$, it suffices to show that for any $k \in \mathbb{N}$, $\tau_{k} < \tau_{\infty}$. First, we order elements of $\mathcal{U}_{\infty}$, $u^{(1)}, u^{(2)} \ldots,$ according to birth time, breaking ties with the lexicographic ordering. Suppose that $\tau_{k} = \tau_{\infty} < \infty$ for some $k \in \mathbb{N}$. Then, since $|\mathcal{T}_{\infty}| = \infty$, the set
		\begin{equation} \label{eq:time-zero-indiv}
			\Lambda = \{u \in \cU_{\infty}: \mathcal{B}(u) = \tau_{\infty} = \tau_{k}\}
		\end{equation}
		is infinite and each element of $\Lambda$ is a descendant (though not necessarily a child) of one of $u^{(1)}, \ldots, u^{(k-1)}$. Now, there are two (not mutually exclusive) cases: either there exists a minimal element $u' \in \Lambda$ such that $\Lambda$ contains infinitely many elements of the form $u'v \in \Lambda$, with $v \in \cU_{\infty}$, or, one of $u^{(1)}, \ldots, u^{(k-1)}$ has infinitely many children in $\Lambda$. The latter case occurs with probability $0$, because, by Equation~\eqref{eq:non-instantaneous-sideways}, for each $j \in [k-1]$ and $n \in \mathbb{N}$, we have $\sum_{i=n+1}^{\infty} X(u^{(j)}i) > 0$ almost surely. Hence, a single individual cannot produce infinitely many children instantaneously. But now, for the prior case, the size of the collection
		\begin{equation} \label{eq:time-zero-galton-watson}
			\left\{u'v: v \in \cU_{\infty}\right\},
		\end{equation}
		is the total progeny of a Bienaym\'{e}-Galton-Watson branching process with offspring distribution $\xi^{(u')}(0) \sim \xi(0)$, and by~\eqref{eq:non-instantaneous-explosion}, this is finite almost surely. We deduce that, on $\left\{|\mathcal{T}_{\infty}| = \infty\right\}$, we have $\tau_{k} < \tau_{\infty}$ almost surely. 
	\end{proof}
	
	\begin{lemma}\label{lem:expl-min-exp-time}
		Let $(\mathscr{T}_{t})_{t \geq 0}$ be  a $(X,W)$-CMJ branching process that satisfies Assumption~\ref{ass:star}. Almost surely,
		\begin{equation} \label{eq:expl-equals-inf}
			\tau_{\infty} = \inf_{u \in \mathcal{T}_{\infty}}\left\{\mathcal{B}(u) + \mathcal{P}(u)\right\}.
		\end{equation} 
	\end{lemma}
	
	\begin{proof}
		First, as a shorthand in this proof, we define $\tau^{+}_{\infty} := \inf_{u \in \mathcal{T}_{\infty}}\left\{\mathcal{B}(u) + \mathcal{P}(u)\right\}$, so that we need only show that $\tau_{\infty} = \tau^{+}_{\infty}$ almost surely. Note that, for each $u \in \mathcal{U}_{\infty}$, we have $\tau_{\infty} \leq \mathcal{B}(u) + \mathcal{P}(u)$ and hence $\tau_{\infty} \leq \tau^{+}_{\infty}$. Moreover, Assumption~\ref{ass:star} guarantees that $\mathcal{P}(\varnothing) < \infty$ almost surely, hence, in particular, $\tau_{\infty} \leq \mathcal{P}(\varnothing) < \infty$. Thus, by Proposition~\ref{prop:finite-l-moderate} and Lemma~\ref{lemma:tree-ident}, for any $L-1 \in \mathbb{N}$, $\mathcal{T}_{\infty} \subseteq \mathscr{T}_{\tau_{\infty}}$   contains only finitely many $(L-1)$-moderate elements. Since $\mathcal{T}_{\infty}$ is almost surely infinite (because of Condition~\ref{item:stardom} of Assumption~\ref{ass:star}), by the pigeonhole principle, there must be infinitely many elements of the form $u^{*}L \in \mathcal{T}_{\infty}$, with $u^{*} \in \mathcal{U}_{\infty}$. Now, for such $u^{*}$ and for any $\eps > 0$, 
		\[
		\lim_{L \to \infty} \Prob{\mathcal{B}(u^{*}) + \mathcal{P}(u^{*}) - \mathcal{B}(u^{*}L) > \eps} = \lim_{L \to \infty}\Prob{\sum_{j=L+1}^{\infty} X(u^{*}j) > \eps} \leq \lim_{L \to \infty} \Prob{Y_{L} > \eps} = 0.
		\]
		Formally, to find these elements $u^{*}L$ we must condition on the sigma-algebra generated by the `information' in $\mathcal{T}_{\infty}$. Thus, suppose that $\mathscr{W}(\mathcal{T}_{\infty}) := \bigcup_{k \in \mathbb{N}} \mathscr{W}_{\tau_{k}}$ denotes the sigma algebra generated by the random variables $\{(\mathcal{B}(u), W_{u}): u \in \mathcal{T}_{\infty}\}$. Clearly, $\tau_{\infty}$ and $\mathcal{T}_{\infty}$ are $\mathscr{W}(\mathcal{T}_{\infty})$-measurable. Now, if (conditioning on this sigma algebra) there exists $u^{*} \in \mathcal{T}_{\infty}$ such that for each $j \in \mathbb{N}$ we have $u^{*}j \in \mathcal{T}_{\infty}$, then $\mathcal{B}(u^{*}j) \leq \tau_{\infty}$ for each $j\in\N$. Hence, $\mathcal{B}(u^{*}) + \mathcal{P}(u^{*}) = \tau_{\infty}$, and we are done. Otherwise, for any $L \in \mathbb{N}$, we can guarantee the existence of $u^{*} \in \mathcal{T}_{\infty}$ and a final $J_{L} \geq L$ such that $u^{*}J_{L} \in \mathcal{T}_{\infty}$ but $u^{*}(J_{L}+1) \notin \mathcal{T}_{\infty}$. Then, for any $\eps > 0$,
		\begin{linenomath}
			\begin{align*}
				\Prob{\tau^{+}_{\infty} - \tau_{\infty} > \eps \, | \, \mathscr{W}(\mathcal{T}_{\infty})} 
				& \leq \Prob{\mathcal{B}(u^{*}) + \mathcal{P}(u^{*}) - \tau_{\infty}>\eps \, | \, \mathscr{W}(\mathcal{T}_{\infty})} \\ 
				&\leq \Prob{\sum_{j = J_{L}+1}^{\infty} X(u^{*}j) > \eps \, \bigg | \, \mathscr{W}(\mathcal{T}_{\infty})} \leq \Prob{Y_L \geq \eps}.
			\end{align*}
		\end{linenomath}
		Taking limits $L \to \infty$ and then $\eps \to 0$, we deduce that, $\tau^{+}_{\infty} = \tau_{\infty}$ almost surely, as required. 
	\end{proof} 
	
	\subsubsection{Proof of Theorem~\ref{thm:star}}
	\begin{proof}[Proof of Theorem~\ref{thm:star}]
		Suppose that $K'$ is taken as in Proposition~\ref{prob:local-explosions}. 
		Note that we may view any $w \in \mathcal{U}_{\infty}$ as a concatenation $w = uv$, where $u \in \mathcal{U}_{K'}$ is $K'$-moderate, and $v = v_1 \cdots v_{k}$, where $v_{1} > K'$ (here we also allow $v$ to be empty, so that $K'$-moderate nodes $w$ may also be interpreted as a concatenation). Now, note that (on $\mathcal{B}(u) < \infty$) the birth times $\mathcal{B}(uv) - \mathcal{B}(u) \sim \mathcal{B}(v)$, and thus, by arguments analogous to those appearing in Proposition~\ref{prob:local-explosions}, for any $u \in \mathcal{U}_{\infty}$ (in particular for $u \in \mathcal{U}_{K'}$),
		\be \label{eq:moderate-prefix}
		\E{\left| \left\{a = a_1 \cdots a_{m} \in \cU_\infty: a_1>K', u a\text{ explodes before } ua_{|_{m-1}}, ua_{|_{m-2}}, \ldots, u\right\}\right| }<\infty.
		\ee 
		Now, since $\tau_{\infty} < \infty$ almost surely, we infer from Proposition~\ref{prop:finite-l-moderate} with $L = K'$, that $|\{u \in \mathcal{U}_{K'}: \mathcal{B}(u) \leq \tau_{\infty}\}| < \infty$ almost surely. Therefore, by~\eqref{eq:moderate-prefix}, the set 
		\be \label{eq:Sfin} 
		S:=\left\{u \in \cU_\infty: \mathcal{B}(u) \leq \tau_{\infty}, u \text{ explodes before all of its ancestors} \right\}.
		\ee
		is finite almost surely.
		By the definition of $S$ and the fact that the infimum of a finite set is attained by (at least) one of the elements, by Lemma~\ref{lem:expl-min-exp-time}, almost surely there exists $u^{*}\in S$ such that 
		\[
		\mathcal{B}(u^{*}) + \mathcal{P}(u^{*}) = \inf_{v \in S} \{\mathcal{B}(v) + \mathcal{P}(v)\} = \inf_{u \in \cU_\infty} \{\mathcal{B}(u) + \mathcal{P}(u)\} = \tau_{\infty}.
		\]
		This implies that $u^{*}$ has infinite degree in $\mathscr{T}_{\tau_{\infty}}$. Moreover, since by Condition~\ref{item:stardom} of Assumption~\ref{ass:star} we have $\sum_{i=n+1}^{\infty} X(u^{*}i) > 0$ almost surely, it follows that for each $i \in \mathbb{N}$ we have $\mathcal{B}(u^{*}i) < \tau_{\infty}$ almost surely. Therefore, by the equality in~\eqref{eq:inclusion}, $u^{*}$ has infinite degree in $\mathcal{T}_{\infty}$ as well. 
	\end{proof}
	
	\subsection{Sufficient criteria for an infinite path and structural results in the star regime} \label{sec:inf-path-structure-proof}
	
	To prove Theorem~\ref{thm:path}, we first state and prove the following lemma.
	
	\begin{lemma}\label{lemma:explodingchild}
		Let $(\mathscr{T}_{t})_{t \geq 0}$ be  a $(X,W)$-CMJ branching process. Under Assumption~\ref{ass:path},
		\begin{equation} \label{eq:child-explodes-first-forall}
			\Prob{\bigcap_{u \in \cU_{\infty}} \bigcap_{j=1}^{\infty}\bigcup_{i = j}^{\infty}\left\{\mathcal{B}(ui) + \mathcal{P}(ui) < \mathcal{B}(u) + \mathcal{P}(u)\right\}} = 1. 
		\end{equation}
	\end{lemma}
	
	\begin{proof}
		We first fix $u\in\cU_\infty$ and condition on the random variables $\mathcal{B}(u)$ and $W_{u}$. We then sample each of the values of $\mathcal{P}(ui)$, then the values of $X_{W_{u}}(ui)$, for $i\in\N$.  Note that for each $u \in \cU_{\infty}$, the random variables $(\mathcal{P}(ui))_{ i \in \mathbb{N}}$ are i.i.d.\ and distributed like $\mathcal{P}(i)$. Thus, by~\eqref{eq:div-condition} and the converse of the Borel-Cantelli lemma, on sampling each $\mathcal{P}(ui)$, for any $w \in S$ with probability $1$, we have $\mathcal{P}(ui) < \nu^{w}_{i}$ infinitely often. As a result, conditionally on the value of the weight $W_{u}$, with probability $1$ we have $\mathcal{P}(ui) < \nu^{W_{u}}_{i}$ for infinitely many $i\in\N$.
		Conditioning on this event, let $i_1, i_2, \ldots,$ denote indices such that, for each $\ell \in \mathbb{N}$, we have $\mathcal{P}(ui_\ell) < \nu^{W_{u}}_{i_\ell}$ almost surely. Then, by~\eqref{eq:smallest-expl-prob}, there exists $I_{0} \in \mathbb{N}$, such that, for some $\delta = \delta(I_0) > 0$ and for all $i_{\ell} \geq I_0$,
		\[
		\Prob{\sum_{k=i_\ell+1}^{\infty} X_{W_{u}}(uk) \geq \nu^{W_{u}}_{i_\ell} \, \bigg | \, W_{u}} \geq \delta.
		\]
		As a result, for some $\phi(i_\ell) > i_\ell$ sufficiently large, 
		\[
		\Prob{\sum_{k=i_\ell+1}^{\phi(i_\ell)} X_{W_{u}}(uk) \geq \nu^{W_{u}}_{i_{j}} \, \bigg | \, W_{u}} \geq \frac{\delta}{2}.  
		\]
		We now pass to a sub-sequence $(i_{\ell_n})_{n\in\N}$ of $(i_\ell)_{\ell \in \mathbb{N}}$, such that $ i_{\ell_1} > I_0$, and $i_{\ell_{n+1}} = \inf\{i_\ell: \ell\in \N, i_\ell >\phi( i_{\ell_n})\}$.
		Then, by Equation~\eqref{eq:starindep}, conditionally on $W_{u}$, the random variables 
		\[
		\left(\sum_{k=i_{\ell_n}+1}^{i_{\ell_{n+1}}} X_{W_{u}}(uk)\right)_{n \in \mathbb{N}}
		\] 
		are independent of each other. Again applying the converse of the Borel-Cantelli lemma, we obtain  
		\[
		\Prob{\bigcap_{j=1}^{\infty}\bigcup_{n = j}^{\infty}\left\{\sum_{k=i_{\ell_n}+1}^{i_{\ell_{n+1}}} X_{W_{u}}(uk) \geq  \nu^{W_{u}}_{i_{\ell_n}} \right\} \, \Bigg| \, \mathcal{B}(u), W_{u}} = 1.
		\]
		But for every index $i^{*} \in (i_{\ell_n})_{n \in \mathbb{N}}$, we have by the definition of the sequence $(i_\ell)_{\ell\in\N}$,
		\[
		\mathcal{P}(ui^{*}) <\nu^{W_{u}}_{i^{*}} 
		\leq \sum_{i=i^{*}+1}^{\infty} X_{W_{u}}(ui).
		\]
		Adding $\mathcal{B}(u) + \sum_{i=1}^{i^{*}} X_{W_{u}}(ui)$ (which equals $\cB(ui^*)$) to both sides, we deduce that, almost surely, conditionally on $\mathcal{B}(u)$ and $W_{u}$,  
		\[
		\mathcal{B}(ui^{*}) + \mathcal{P}(ui^{*}) < \mathcal{B}(u) + \mathcal{P}(u).
		\] 
		Then, by taking expectations over $\mathcal{B}(u)$ and  $W_{u}$, we have
		\begin{equation} \label{eq:child-explodes-first}
			\Prob{\bigcap_{j=1}^{\infty}\bigcup_{i = j}^{\infty}\left\{\mathcal{B}(ui) + \mathcal{P}(ui) <\mathcal{B}(u) + \mathcal{P}(u)\right\}} = 1. 
		\end{equation}
		But now, since $\mathcal{U}_{\infty}$ is countable, we deduce~\eqref{eq:child-explodes-first-forall}.
	\end{proof}
	
	\subsubsection{Proof of Theorem~\ref{thm:path}} \label{sec:inf-path-proof}
	\begin{proof}[Proof of Theorem~\ref{thm:path}]
		On $\left\{\left|\mathcal{T}_{\infty} \right| = \infty \right\}$, let us first assume that $\mathcal{T}_{\infty}$ does not contain a node of infinite degree. It then immediately follows from K{\H o}nigs Lemma (Lemma~\ref{lemma:konig}) that $\cT_\infty$ contains an infinite path. We then assume, on $\left\{\left|\mathcal{T}_{\infty} \right| = \infty \right\}$, that there exists $u^* \in \mathcal{T}_{\infty}$ such that $u^*$ has infinite degree, i.e.\ $\mathcal{B}(u^*) + \mathcal{P}(u^*) = \tau_{\infty}$. But then, by Equation~\eqref{eq:child-explodes-first-forall} from Lemma~\ref{lemma:explodingchild}, there must be a child of $u^*$, $u^*j$ say, such that $\mathcal{B}(u^*j) + \mathcal{P}(u^*j) < \mathcal{B}(u^*) + \mathcal{P}(u^*) = \tau_{\infty}$, a contradiction.  
	\end{proof}
	
	\subsubsection{Proof of Theorem~\ref{thm:structure}} \label{sec:structure-proof}
	\begin{proof}[Proof of Theorem~\ref{thm:structure}]
		For each $u \in \mathcal{U}_{\infty}$, we set \[(\mathscr{T}^{(u)}_{t})_{t \geq 0} := \left\{uv \in \cU_{\infty}: \sum_{j=0}^{\ell-1} \left(\mathcal{P}_{v_{j+1}}(uv_{|_{j}}) \leq t\right)\right\},\]
		where we recall that if $v = v_{1} \cdots v_{m}$ then $uv_{|_{j}} = uv_1\cdots v_j$. Note that, if $\mathcal{B}(u) < \infty$, by~\eqref{eq:birth-time-identity} we have 
		$\mathscr{T}^{(u)}_{t} = \left\{uv: uv \in \mathscr{T}_{t +\mathcal{B}(u)}\right\}.$
		However, the definition we use allows us to define $(\mathscr{T}^{(u)}_{t})_{t \geq 0}$ even when $\mathcal{B}(u) = \infty$.  Note that, by~\eqref{eq:cmj-assumption}, $(\mathscr{T}^{(u)}_{t})_{t \geq 0} \sim (\mathscr{T}_{t})_{t \geq 0}$, and for each $u \in \cU_{\infty}$, the $((\mathscr{T}^{(ui)}_{t})_{t \geq 0})_{i \in \mathbb{N}}$ are i.i.d. Now, upon sampling each of the random variables $(X(ui))_{i\in\N}$, (regardless of whether $\mathcal{B}(u) < \infty$ or not), recalling the notation from~\eqref{eq:P-t}, note that if we have
		\[
		\mathcal{P}_{T}(ui) < \sum_{k=i+1}^{\infty} X(uk),
		\]
		then, if $\mathcal{B}(u) + \mathcal{P}(u) < \infty$, we have $\cB(ui)+\mathcal{P}_{T}(ui) < \mathcal{B}(u) + \mathcal{P}(u)$. Now, by exploiting an almost identical argument to the proof of Lemma~\ref{lemma:explodingchild} (as in Equations~\eqref{eq:div-condition} and~\eqref{eq:smallest-expl-prob}), combining Equations~\eqref{eq:div-condition-sub-tree} and~\eqref{eq:smallest-expl-prob-sub} from Condition~\ref{item:treediv-cond} of Assumption~\ref{ass:structure} allows us to deduce that
		\begin{equation}
			\label{eq:sub-tree-first-forall}
			\Prob{\bigcap_{u \in \cU_{\infty}} \bigcap_{j=1}^{\infty}\bigcup_{i = j}^{\infty}\left\{ \mathcal{P}_{T}(ui) < \sum_{k=i+1}^{\infty} X(uk) 
				\right\}} = 1. 
		\end{equation}
		But then, this implies that if a node $u^{*} \in \mathcal{U}_{\infty}$ has infinite degree in $\mathcal{T}_{\infty}$, there exist infinitely many indices $i$ such that $\mathcal{B}(u^*i) + \mathcal{P}_{T}(u^*i) < \mathcal{B}(u^*) + \mathcal{P}(u^*)$, and hence, by Lemma~\ref{lemma:tree-ident}, that $(u^{*}i)T \subseteq \mathcal{T}_{\infty}$. This proves the first statement. 
		
		For the second statement, by Equation~\eqref{eq:sum-cond-sub-tree} we have 
		\begin{linenomath}
			\begin{align} \label{eq:union bound-exp-event}
				\sum_{i=1}^{\infty} \Prob{\mathcal{P}_{T}(ui) < \sum_{k=i+1}^{\infty} X(uk) \, \bigg | \, W_{u}} < \infty,
			\end{align}
		\end{linenomath}
		almost surely. Thus, by a conditional analogue of the Borel-Cantelli lemma, 
		\[
		\Prob{\bigcap_{j=1}^{\infty} \bigcup_{i=j}^{\infty} \left\{\mathcal{P}_{T}(ui) < \sum_{k=i+1}^{\infty} X(uk) \right\} \, \bigg | \, W_{u}} = 0,
		\]
		almost surely. Taking expectations over $W_{u}$ and a union bound over $u \in \mathcal{U}_{\infty}$, we deduce that
		\[
		\Prob{\bigcup_{u \in \cU_{\infty}} \bigcap_{j=1}^{\infty}\bigcup_{i = j}^{\infty}\left\{\mathcal{P}_{T}(ui) < \sum_{k=i+1}^{\infty} X(uk)\right\}} = 0. 
		\]
		This implies that, if $u^{*}$ is such that $\mathcal{B}(u^{*}) + \mathcal{P}(u^{*}) = \tau_{\infty}$, almost surely, there exist only finitely many indices $i$ such that $T$ appears as a sub-tree of $u^*i$ in $\mathcal{T}_{\infty}$. 
		
		The final statement in Item~\ref{item:uniquepors} follows from the proof of Theorem~\ref{thm:star}, Equation~\ref{eq:Sfin} in particular, which states that $\cT_\infty$ contains a finite number of stars $u^*$ in $\cT_\infty$ almost surely under Assumption~\ref{ass:star}. 
	\end{proof}
	
	\subsection{Uniqueness conditions related to the existence of a star or an infinite path}\label{sec:uniqueproof}
	
	In this section, it is useful to introduce some extra notation. First, recall that, given $u \in \mathcal{U}_{\infty}$, the random variable $\mathcal{B}(u) + \mathcal{P}(u)$ determines the time at which $u$ `explodes', in the sense that, if $\sigma = \mathcal{B}(u) + \mathcal{P}(u)$, the individual $u$ has infinite degree in $\mathscr{T}_{\sigma}$. We also define the random variable $\tau_{\Path}(u)$ as the amount of time taken after the `birth' of $u$, for there to be an infinite path containing $u$ in the process $(\mathscr{T}_{t})_{t \geq 0}$. To make this more precise, we extend the notation for elements $v \in \mathcal{U}_{\infty}$ to $v \in \mathbb{N}^{\mathbb{N}}$: for $v = v_1v_2\cdots \in \mathbb{N}^\infty$ and $\ell \in \mathbb{N}$, we set $v_{|_{\ell}} = v_{1} v_{2} \cdots v_{\ell} \in \mathcal{U}_{\infty}$. Then, for $u \in \mathcal{U}_{\infty}$ we set 
	\[
	\tau_{\Path}(u) := \inf_{t \geq 0} \left\{ \exists v \in \mathbb{N}^\infty:   \sum_{j=0}^{i-1} \left(\mathcal{P}_{v_{j+1}}(uv_{|_{j}})\right) \leq t\text{ for all }i\in\N \right\}.
	\footnote{Note that, since the definition of $\tau_{\Path}(u)$ includes the uncountable set $\mathbb{N}^{\infty}$, it is not immediately clear that this is measurable. This is the reason that we assume that $(\Omega, \overline{\Sigma}, \mathbb{P})$ and $\mathscr{F}_{t}$ are complete: complete sigma algebras with respect to probability measures are closed under the \emph{Souslin operation}~\cite[Theorem~1.10.5, page 38]{bogachev}, which makes a number of uncountable unions measurable. In particular, for each $t \geq 0$, we can write 
		$
		\left\{\tau_{\Path}(u) \leq t\right\} = \bigcup_{v \in \mathbb{N}^{\infty}} \bigcap_{i \in \mathbb{N}}\left\{ \sum_{j=0}^{i-1} \left(\mathcal{P}_{v_{j+1}}(uv_{|_{j}})\right) \leq t\right\} \in \overline{\Sigma}, 
		$
		and similarly, we know for each $t \geq 0$ that $\left\{\tau_{\Path}(\varnothing) \leq t\right\} \in \mathcal{F}_{t}$. 
	}
	\]
	Thus, by this definition, if $\sigma = \mathcal{B}(u) + \tau_{\Path}(u)$ then $\mathscr{T}_{\sigma}$ contains an infinite path passing through $u$.
	
	The approach we use in this section is surprisingly simple, and reminiscent of the approach used, for example, in~\cite{Oliveira-spencer} to show a unique node has infinite degree: we show that, for each $u\in \mathcal{U}_{\infty}$, the random variables $\tau_{\Path}(u)$ or $\mathcal{P}(u)$ have distributions that contain no atoms on $[0, \infty)$. We use this to show that for any pair $u, v \in \mathcal{U}_{\infty}$, (which have `independent' sub-trees), the probability that both have infinite degrees, or lie on infinite paths simultaneously is $0$. As $\mathcal{U}_{\infty}$ is countable, we can readily take a union bound over all these pairs, and deduce the result. Condition~\ref{item:uniquenoexplatom} of Assumption~\ref{ass:uniqueness} (and~\eqref{eq:cmj-assumption}) already provides this property to $\mathcal{P}(u)$. For $\tau_{\Path}(u)$ we use the following result.
	
	\begin{lemma} \label{lem:no-atom}
		Let $(\mathscr{T}_{t})_{t \geq 0}$ be an explosive $(X,W)$-CMJ process. Under Condition~\ref{item:uniquenobirthatom} of Assumption~\ref{ass:uniqueness}, the distribution of $\tau_{\Path}(\varnothing)$ contains no atom on $[0, \infty)$. 
	\end{lemma}
	
	\begin{proof}
		We argue by contradiction and suppose that $\tau_{\Path}(\varnothing)$ contains an atom. Set
		\begin{equation}
			\widetilde{a} := \inf\left\{t \in [0, \infty): \Prob{\tau_{\Path}(\varnothing) = t} > 0\right\}.
		\end{equation}
		Now, exploiting Condition~\ref{item:uniquenobirthatom} of Assumption~\ref{ass:uniqueness} and the definition of $\widetilde{a}$, let $a$ be such that $\Prob{\tau_{\Path}(\varnothing) = a} > 0$, and $a - \widetilde{a} < \eps$ (with $\eps$ as in Condition~\ref{item:uniquenobirthatom}). Let $\mathcal{G}_{1}$ denote the sigma algebra generated by $(\mathcal{B}(i))_{i\in\N}$. The ancestral node $\varnothing$ contains an infinite path precisely when one of its children lies on an infinite path. Thus, 
		\begin{linenomath}
			\begin{align} \label{eq:child-path}
				\Prob{\tau_{\Path}(\varnothing) = a \, \big | \,  \mathcal{G}_{1}} & = \Prob{\bigcup_{i=1}^{\infty} \left\{\tau_{\Path}(i) = a - \mathcal{B}(i)\right\} \, \bigg | \, \mathcal{G}_{1}}
				\\ \hspace{1cm} & \leq \sum_{i=1}^{\infty} \Prob{\tau_{\Path}(i) = a - \mathcal{B}(i) \, \bigg| \,  \mathcal{G}_{1}}. 
			\end{align}
		\end{linenomath}
		Now, as the event $\{\tau_{\Path}(i) = a - \mathcal{B}(i)\}$ depends on $\mathcal{G}_{1}$ only via the sigma algebra generated by $\mathcal{B}(i)$, we may re-write the summands on the right hand side as
		\begin{linenomath}
			\begin{align*}
				\Prob{\tau_{\Path}(i) = a - \mathcal{B}(i) \, \bigg| \,  \mathcal{B}(i)} &= \Prob{\tau_{\Path}(i) = a - \mathcal{B}(i) \, \bigg| \,  \mathcal{B}(i)} \mathbf{1}_{\left\{\mathcal{B}(i) < \eps \right\}} \\ & \hspace{1cm} + \Prob{\tau_{\Path}(i) = a - \mathcal{B}(i) \, \bigg| \,  \mathcal{B}(i)}\mathbf{1}_{\left\{\mathcal{B}(i) \geq \eps \right\}}.
			\end{align*}
		\end{linenomath}
		Now, note that (by~\eqref{eq:cmj-assumption}), as $\tau_{\Path}(i)$ is identically distributed to $\tau_{\Path}(\varnothing)$, the distribution of $\tau_{\Path}(i)$ has no atom smaller than $\widetilde{a} > a - \eps$. As a result, on the event $\mathcal{B}(i) \geq \eps$, \[\Prob{\tau_{\Path}(i) = a - \mathcal{B}(i) \, \bigg| \,  \mathcal{B}(i)} = 0, \quad \text{ almost surely. }\] 
		Combining this with~\eqref{eq:child-path}, this leads to the inequality
		\[
		\Prob{\tau_{\Path}(\varnothing) = a \, \big | \,  \mathcal{G}_{1}} \leq \sum_{i=1}^{\infty} \Prob{\tau_{\Path}(i) = a - \mathcal{B}(i) \, \bigg| \,  \mathcal{B}(i)} \mathbf{1}_{\left\{\mathcal{B}(i) < \eps \right\}}, 
		\]
		and since $\Prob{\tau_{\Path}(\varnothing) = a} > 0$, it must be the case that for some $i$, \[
		0 < \Prob{\tau_{\Path}(i) = a - \mathcal{B}(i), \mathcal{B}(i) < \eps} = \E{\Prob{\tau_{\Path}(i) = a - \mathcal{B}(i) \, \bigg| \,  \mathcal{B}(i)} \mathbf{1}_{\left\{\mathcal{B}(i) < \eps \right\}}}.\] 
		But now, when we integrate over the possible values of $\mathcal{B}(i)$ on the right-hand side, this implies that 
		\[
		\left\{ \mathcal{B}(i) \in [0, \eps): \Prob{\tau_{\Path}(i) = a - \mathcal{B}(i) \, \bigg| \,  \mathcal{B}(i)} > 0 \right\}
		\] 
		is a set of positive measure with respect to $\mathbb{P}$. But then this implies that, with respect to the distribution induced by $\mathcal{B}(i)$, the set 
		\begin{equation} \label{eq:null-set}
			\left\{x \in [0, \eps): \Prob{\tau_{\Path}(i) = a - x} > 0 \right\},
		\end{equation}
		has positive measure. But now, as the set of atoms of the distribution of $\tau_{\Path}(i)$ (as with any finite measure) must always be countable,~\eqref{eq:null-set} must be countable. As the distribution of $\mathcal{B}(i)$ contains no atoms on $[0, \eps)$ by Condition~\ref{item:uniquenobirthatom} in Assumption~\ref{ass:uniqueness}, countable subsets of $[0, \eps)$ are null sets  with respect to the distribution induced by $\mathcal{B}(i)$; this implies that~\eqref{eq:null-set} must be a null set, a contradiction. 
	\end{proof}
	
	\subsubsection{Proof of Theorem~\ref{thm:uniqueness}}
	\begin{proof}[Proof of Theorem~\ref{thm:uniqueness}]
		The first and second statements of the theorem are clearly satisfied on $\left\{|\mathcal{T}_{\infty}| < \infty\right\}$. Now note that, for $\mathcal{T}_{\infty}$ to have two nodes of infinite degree, there must be two elements $u, v \in \mathcal{U}_{\infty}$ such that 
		\[
		\mathcal{B}(u) + \mathcal{P}(u) = \mathcal{B}(v) + \mathcal{P}(v) = \tau_{\infty}. 
		\]
		By Equation~\eqref{eq:expl-almost-surely}, $\tau_{\infty} < \infty$ almost surely on $\left\{|\mathcal{T}_{\infty}| = \infty\right\}$, hence the left-hand side must also be finite. 
		But now, note that, as long as $v$ is not an ancestor of $u$, (which is the case if $v >_{L} u$ with respect to the lexicographical ordering), then $\mathcal{P}(v)$ is independent of $\mathcal{B}(u), \mathcal{B}(v)$, and $\mathcal{P}(u)$. Moreover, by Condition~\ref{item:uniquenoexplatom} of Assumption~\ref{ass:uniqueness}, $\cP(v)$ has no atom on $[0, \infty)$. Therefore, for $v >_{L} u$ we have,
		\[
		\Prob{\mathcal{P}(v) = \mathcal{P}(u) + \mathcal{B}(v) - \mathcal{B}(u) \, \bigg | \, \mathcal{B}(u), \mathcal{B}(v), \mathcal{P}(v)} \mathbf{1}_{\{\mathcal{B}(u), \mathcal{B}(v), \mathcal{P}(u) < \infty\}} = 0 \quad \text{almost surely, }
		\]
		and hence, for any $u, v \in \mathcal{U}_{\infty}$ with $u \neq v$,  $\Prob{\mathcal{B}(u) + \mathcal{P}(u) = \mathcal{B}(v) + \mathcal{P}(v), \mathcal{B}(u) + \mathcal{P}(u) < \infty} = 0$. But then, since $\mathcal{U}_{\infty}$ is countable, taking a union bound, we deduce
		\be\ba\label{eq:union-1-star}
		&\Prob{\text{$\mathcal{T}_{\infty}$ has two nodes of infinite degree}} \\ & \hspace{1cm}  
		\leq \Prob{\bigcup_{u, v \in \mathcal{U}_{\infty}, u \neq v} \big\{\mathcal{B}(u) + \mathcal{P}(u) = \mathcal{B}(v) + \mathcal{P}(v), \mathcal{B}(u) + \mathcal{P}(u) < \infty\big\}}
		\\ & \hspace{1cm} \leq \sum_{u, v \in \mathcal{U}_{\infty}, u \neq v} \mathbb{P}\big(\mathcal{B}(u) + \mathcal{P}(u) = \mathcal{B}(v) + \mathcal{P}(v), \mathcal{B}(u) + \mathcal{P}(u) < \infty\big) = 0. 
		\ea\ee
		This proves the first statement of the theorem. In a similar manner, using Condition~\ref{item:uniquenobirthatom} of Assumption~\ref{ass:uniqueness} and Lemma~\ref{lem:no-atom}, we have 
		\be\ba \label{eq:union-1-path}
		&\Prob{\text{$\mathcal{T}_{\infty}$ has two infinite paths}} \\ & \hspace{1cm}  
		\leq \Prob{\bigcup_{u, v \in \mathcal{U}_{\infty}, u \neq v} \big\{\mathcal{B}(u) + \tau_{\Path}(u) = \mathcal{B}(v) + \tau_{\Path}(v), \mathcal{B}(u) + \tau_{\Path}(u) < \infty\big\}}
		\\ & \hspace{1cm} \leq \sum_{u, v \in \mathcal{U}_{\infty}, u \neq v} \mathbb{P}\big(\mathcal{B}(u) + \tau_{\Path}(u) = \mathcal{B}(v) + \tau_{\Path}(v), \mathcal{B}(u) + \tau_{\Path}(u) < \infty\big) = 0, 
		\ea\ee
		proving the second statement. Finally, for the third statement we need only prove that a node of infinite degree and an infinite path cannot co-exist. Note that there may be $u$ such that $\mathcal{P}(u)$ and $\tau_{\Path}(\varnothing)$ are correlated, for example, if $\varnothing$ is a parent of $u$. But, noting $\mathbb{N}^{k} = \left\{u \in \mathcal{U}_{\infty} : |u| = k\right\}$,
		\begin{linenomath}
			\begin{align*}
				& \left\{ \tau_{\Path}(\varnothing) = \mathcal{B}(u) + \mathcal{P}(u), \tau_{\Path}(\varnothing) < \infty\right\} \subseteq \bigcup_{v \in \mathbb{N}^{|u|+1}} \left\{ \mathcal{B}(v) + \tau_{\Path}(v) = \mathcal{B}(u) + \mathcal{P}(u), \mathcal{B}(v) + \tau_{\Path}(v) < \infty\right\}.
			\end{align*}
		\end{linenomath}
		Now, for $v \in \mathbb{N}^{|u|+1}$, $\tau_{\Path}(v)$ \emph{is} independent of $u$, and therefore, for each $v \in \mathbb{N}^{|u|+1}$ we have
		\begin{linenomath}
			\begin{align*}
				& \Prob{\mathcal{B}(v) + \tau_{\Path}(v) = \mathcal{B}(u) + \mathcal{P}(u) \, \big | \, \mathcal{B}(v), \mathcal{B}(u), \mathcal{P}(u)} \mathbf{1}_{\{\mathcal{B}(v), \mathcal{B}(u), \mathcal{P}(u) < \infty\}} \\ & \hspace{2cm} = \Prob{ \tau_{\Path}(v) = \mathcal{B}(u) + \mathcal{P}(u) - \mathcal{B}(v) \, \big | \, \mathcal{B}(v), \mathcal{B}(u), \mathcal{P}(u)} \mathbf{1}_{\{\mathcal{B}(v), \mathcal{B}(u), \mathcal{P}(u) < \infty\}} = 0,
			\end{align*}
		\end{linenomath}
		almost surely. Therefore, again using a union bound, we have 
		\be\ba\label{eq:xor-1-path-1-star}
		&\Prob{\text{$\mathcal{T}_{\infty}$ has an infinite path and a node of infinite degree, } |\mathcal{T}_{\infty}| = \infty} 
		\\ & \hspace{1cm}  
		\leq \Prob{\bigcup_{u \in \mathcal{U}_{\infty}} \big\{\mathcal{B}(u) + \mathcal{P}(u) = \tau_{\Path}(\varnothing), \mathcal{B}(u) + \tau_{\Path}(u) < \infty\big\}}
		\\ & \hspace{1cm} \leq \Prob{\bigcup_{u \in \mathcal{U}_{\infty}} \bigcup_{v \in \mathbb{N}^{|u|+1}} \big\{\mathcal{B}(u) + \mathcal{P}(u) = \mathcal{B}(v) + \tau_{\Path}(v), \mathcal{B}(u) + \tau_{\Path}(u) < \infty\big\}}
		\\ & \hspace{1cm} \leq \sum_{{u\in \mathcal{U}_{\infty}}} \sum_{{v \in \mathbb{N}^{|u|+1}}} \mathbb{P}\big(\mathcal{B}(u) + \mathcal{P}(u) = \mathcal{B}(v) + \tau_{\Path}(v), \mathcal{B}(u) + \tau_{\Path}(u) < \infty\big) = 0. 
		\ea\ee
		We deduce the final statement by using a union bound, and combining Equations~\eqref{eq:union-1-star},~\eqref{eq:union-1-path}, and~\eqref{eq:xor-1-path-1-star} with K{\H o}nig's lemma (Lemma~\ref{lemma:konig}).
	\end{proof}
	
	\section{Proofs for applications of main results} \label{sec:applications}
	
	In this section we use the general results of Section~\ref{sec:results} to prove the main theorems in Section~\ref{sec:recursive}. In particular, we show that the conditions stated in Theorem~\ref{thm:star-path-rif} are sufficient and that the condition in Theorem~\ref{thm:sub-treecount} is necessary and sufficient to apply the results of Section~\ref{sec:results} to exponentially distributed inter-birth times.  
	
	\subsection{Preliminary results and tools}
	We first collect some useful lemmas that we use throughout the proofs in this section. 
	
	\begin{lemma} \label{lem:lower-bound-infinite-events}
		Let $(X(i))_{i\in\N}$ be independent random variables and fix $\eps > 0$ and $(b_{n})_{n \in \mathbb{N}} \in (0, \infty)^{\mathbb{N}}$. Then,
		\[
		\Prob{\sum_{i=1}^{\infty} X_{i} < b_{n}} \geq \prod_{i=1}^{\infty} \cL_{(1+ \eps) \log{(n)} b_{n}^{-1}}(X_i) - \frac{1}{n^{1+\eps}}.
		\]
	\end{lemma}
	\begin{proof}
		Let $Y \sim \Exp{b_{n}^{-1} (1+ \eps) \log{n}}$ be given, independent of each of the $X_{i}$. Then, 
		\be
		\Prob{\sum_{i=1}^{\infty} X_{i} < b_{n}} \geq \Prob{\sum_{i=1}^{\infty} X_{i} < Y, Y \leq b_{n}}
		\geq \Prob{\sum_{i=1}^{\infty} X_{i} < Y} - \Prob{Y \geq b_{n}}.
		\ee 
		Evaluating both probabilities by using that $Y$ is an exponential random variable, we obtain 
		\be 
		\E{\e^{-(1+ \eps) \log (n) \,  b_{n}^{-1}\sum_{i=1}^{\infty} X_i}} - \e^{-(1+ \eps) \log{n}} = \prod_{i=1}^{\infty} \cL_{(1+ \eps) \log{(n)} b_{n}^{-1}}(X_i) - \frac{1}{n^{1+\eps}},
		\ee
		as desired.
	\end{proof}
	
	\noindent We believe that the following lemmas are more well-known, but we provide a proof of the first for completeness. Note, as is clear from the proof, that~\eqref{eq:exp-conc-2} is a special case of the more general Equation~\eqref{eq:moment-prob-bound1} from Lemma~\ref{lem:moment-prob-bounds}.
	
	\begin{lemma} \label{lem:exponential-concentration}
		Let $Z_{0}, \ldots, Z_{k}$ be independent exponential random variables with parameters $r_0, \ldots, r_{k-1}$, respectively, and fix $\lambda > 0$.  Then, 
		\begin{equation} \label{eq:exp-conc}
			\Prob{\sum_{j=0}^{k} Z_{j}\leq \lambda} \geq \prod_{j=0}^{k} \frac{r_{j}}{r_j + (k+1) \lambda^{-1}}.  
		\end{equation} 
		Moreover, let $(X_{w^{*}}(j))_{j \in \mathbb{N}}$ be independent, exponential  random variables with rates $(r_j)_{j\in\N}$, satisfying ~$\mu_i:=\sum_{j=i+1}^{\infty} \E{X_{w^{*}}(j)}  < \infty$ for $i \in \mathbb{N}$. For any $c < 1$ and for any $i \in \mathbb{N}$,
		\begin{equation} \label{eq:exp-conc-2}
			\Prob{\sum_{j=0}^{k} Z_{j}\leq \sum_{j=i+1}^{\infty} X_{w^{*}}(j)} \leq \frac{1}{1-c} \prod_{j=0}^{k} \frac{r_{j}}{r_j + c\mu^{-1}_{i}}.
		\end{equation}
	\end{lemma}
	\begin{proof} 
		For the first inequality, we may, without loss of generality take $\lambda = 1$, since for an exponential random variable $X \sim \Exp{r}$, $X/\lambda$ is again exponentially distributed with parameter $\lambda r$. Then,
		\be
		\Prob{\sum_{j=0}^{k} Z_{j}\leq 1}\geq \Prob{Z_{j}\leq 1/(k+1) \text{ for all }j\in\{0,\ldots, k\}} = \prod_{j=0}^{k} \left(1 - \e^{-r_{j}/(k+1)}\right) \geq \prod_{j=0}^{k} \frac{r_{j}}{r_{j} + (k+1)},
		\ee
		where the right-hand side uses the inequality $1- \e^{-x} \geq \frac{x}{1+x}$, for all $x \geq 0$. 
		Meanwhile, the upper bound in~\eqref{eq:exp-conc-2} uses a standard Chernoff bound. First, note that for $c < 1$,
		\be \label{eq:mux}
		c\mu_{i}^{-1} = c\left(\sum_{j=i+1}^{\infty} \frac{1}{r_{j}}\right)^{-1} < r_{j},\qquad \text{for all }j\geq i+1.
		\ee 
		Furthermore, for each $i \in \mathbb{N}$,
		\be \label{eq:mgf-exp-inverse-bound}
		\prod_{j=i+1}^{\infty} \frac{r_{j} - c\mu_{i}^{-1}}{r_{j}} = \prod_{j=i+1}^{\infty} \left(1 - \frac{c\mu_{i}^{-1}}{r_{j}}\right) > 1 - c \mu_{i}^{-1} \sum_{j=i+1}^{\infty} \frac{1}{r_{j}} = 1-c,
		\ee
		where the last step uses the inequality $\prod_{i} (1-p_{i}) \geq 1 - \sum_{i} p_{i}$, for $p_i\in[0,1]$.
		By the moment generating function of exponential random variables, we have 
		\be \label{eq:mgf-exp-bound}
		\E{\e^{c\mu_{i}^{-1} \sum_{j=i+1}^{\infty} X_{w^*}(j)}} = \prod_{j=i+1}^{\infty} \frac{r_{j}}{r_{j} - c\mu_{i}^{-1}} \stackrel{\eqref{eq:mgf-exp-inverse-bound}}{\leq} \frac{1}{1-c},
		\ee
		where the exponential moments exist by~\eqref{eq:mux}. Finally, this yields 
		\[
		\Prob{\sum_{j=0}^{k} Z_{j} \leq \sum_{j=i+1}^{\infty} X_{w^{*}}(j)} \leq \Prob{\e^{c\mu_{i}^{-1}\left(\sum_{j=i+1}^{\infty} X_{w^{*}}(j) - \sum_{j=0}^{k} Z_{j}\right)} \geq 1} \leq \frac{1}{1-c} \prod_{j=0}^{k} \frac{r_{j}}{r_{j} + c\mu_{i}^{-1}},
		\]
		as desired.
	\end{proof}
	
	\begin{lemma}[Paley-Zygmund Inequality] \label{lem:paley-zygmund}
		Let $Z$ be a non-negative random variable with finite variance, and let $d\in(0,1)$.  Then, 
		\begin{equation} \label{eq:paley-zygmund}
			\Prob{Z \geq d \E{Z}} \geq (1-d)^2 \frac{\E{Z}^2}{\Var{(Z)} + \E{Z}^2}.
		\end{equation} 
	\end{lemma}
	\subsection{Structure theorems related to explosive recursive trees}

	\subsubsection{Proof of Theorem~\ref{thm:star-path-rif}} \label{proof:star-path-rif}
	
	\begin{proof}[Proof of Theorem~\ref{thm:star-path-rif}]
		Recall that for fixed $u, W_{u}$, the random variables $(X_{W_{u}}(ui))_{i \in \mathbb{N}}$ are independent, with each $X_{W_{u}}(ui) \sim \Exp{f(i-1, W_{u})}$, conditionally on $W_u$. Since the exponential distribution is a smooth distribution, one readily verifies that the conditions of Assumption~\ref{ass:uniqueness} are met; hence the associated tree $\mathcal{T}_{\infty}$ contains either a unique node of infinite degree, or a unique infinite path. Now, to prove Item~\ref{item:star-explosive-rif}, note that Conditions~\ref{item:starindep} and~\ref{item:starnonzero} of Assumption~\ref{ass:star} are immediately satisfied. Moreover, since $f$ satisfies~\eqref{eq:minfass}, Equation~\eqref{eq:stochastic-bound} and thus Condition~\ref{item:stardom} is satisfied by setting 
		\be \label{eq:ynexplicit}
		Y_{n} := \sum_{i=n+1}^{\infty} \widetilde{X}_{w^{*}}(i),
		\ee 
		where each $\widetilde{X}_{w^{*}}(i) \sim \Exp{f(i-1, w^{*})}$. By choosing $c < 1$ according to Equation~\eqref{eq:star-explosive-rif} and by following the calculations in Equations~\eqref{eq:mgf-exp-inverse-bound} and~\eqref{eq:mgf-exp-bound} (with $r_j:=f(j-1,w^*)$), we see that Condition~\ref{item:starlimsup} is satisfied. Finally, note that, uniformly in $w$, we have 
		\be 
		\prod_{i=n}^{\infty} \frac{f(i,w)}{f(i,w) + c\mu_{n}^{-1}} > 0. 
		\ee 
		Indeed, taking logarithms and using the inequality 
		\be
		-\log{(1-x)} \leq x + \frac{x^{2}}{2(1-x)}, \quad 0 < x < 1,
		\ee
		we obtain 
		\begin{linenomath}
			\begin{align}
				-\log{\left(\prod_{i=n}^{\infty} \frac{f(i,w)}{f(i,w) + c\mu_{n}^{-1}} \right)} & =    -\sum_{i=n}^{\infty} \log{\left(1 - \frac{c\mu_{n}^{-1}}{f(i,w) + c\mu_{n}^{-1}} \right)} \\ & \leq \sum_{i=n}^{\infty} \left(\frac{c\mu_{n}^{-1}}{f(i,w^{*}) + c\mu_{n}^{-1}} + \frac{c^{2}\mu_{n}^{-2}}{2f(i,w^{*})(f(i,w^{*}) + c\mu_{n}^{-1})}\right) < \frac{3}{2}, 
			\end{align}
		\end{linenomath}
		which implies that 
		\be \label{eq:prodfromn}
		\prod_{i=n}^{\infty} \frac{f(i,w)}{f(i,w) + c\mu_{n}^{-1}} > \e^{-3/2}.
		\ee 
		Thus, by~Equation~\eqref{eq:star-explosive-rif},
		\be
		\sum_{n=1}^{\infty}\E{\mathcal{L}_{c\mu^{-1}_{n}}(\mathcal{P}_{n}(\varnothing); W)}=\sum_{n=1}^{\infty} \E{\prod_{i=0}^{n-1}\frac{f(i,W)}{f(i,W) + c\mu_{n}^{-1}}} < \e^{3/2} \sum_{n=1}^{\infty} \E{\prod_{i=0}^{\infty}\frac{f(i,W)}{f(i,W) + c\mu_{n}^{-1}}} < \infty, 
		\ee 
		so that Condition~\ref{item:starlaplace} is satisfied, where the final step uses the assumption in Item~\ref{item:star-explosive-rif}. We can now apply Theorems~\ref{thm:star} and~\ref{thm:uniqueness} to obtain the desired result.
		
		For Item~\ref{item:path-explosive-rif}, if there exists $c > 1$ satisfying Equation~\eqref{eq:path-explosive-rif}, we write $c = d^{-1}(1+\eps)$, for some $0 < d < 1$ and $\eps > 0$. Otherwise, assume $\eqref{eq:div-condition}$ is satisfied with $\nu^{w}_{n} = d \mu^{w}_{n} = d\E{\sum_{i=n+1}^{\infty} X_{w}(i)}$ for some $d\in(0,1)$. First note that, since the $X_w(i)$ are mutually independent exponential  random variables, 
		\be \label{eq:exp-variance-bound}
		\Var{\left(\sum_{i=n+1}^{\infty} X_{w}(i)\right)} = \sum_{i=n}^{\infty} \frac{1}{f(i,w)^2} \leq (\mu^{w}_{n})^{2}. 
		\ee
		By the Paley-Zygmund inequality, it thus follows that 
		\be \label{eq:tail-lower-bound-prob}
		\Prob{\sum_{i=n+1}^{\infty} X_{w}(i) \geq d \mu^{w}_{n}} \geq (1-d)^{2} \frac{(\mu^{w}_{n})^{2}}{(\mu^{w}_{n})^{2} + \sum_{i=n}^{\infty} f(i,w)^{-2}} \stackrel{\eqref{eq:exp-variance-bound}}{\geq} \frac{(1-d)^{2}}{2},  
		\ee
		which implies~\eqref{eq:smallest-expl-prob} in either case. We now need only show that Equation~\eqref{eq:path-explosive-rif} implies that~\eqref{eq:div-condition} is satisfied with this choice of $(\nu^{w}_{n})_{n \in \mathbb{N}}$. We apply Lemma~\ref{lem:lower-bound-infinite-events} with $b_{n} := d \mu_{n}^w$ and $\eps > 0$ so that $c = d^{-1}(1+ \eps)$, and by conditioning on the vertex-weight $W$, to obtain
		\[
		\Prob{\mathcal{P}< d \mu^{w}_{n}} \geq \E{\prod_{i=0}^{\infty}\frac{f(i,W)}{f(i,W) + c (\mu^{w}_{n})^{-1} \log{n}}} - \frac{1}{n^{1+ \eps}}.
		\]
		We thus obtain that~\eqref{eq:div-condition} is satisfied when Equation~\eqref{eq:path-explosive-rif} holds. 
	\end{proof}
	
	\subsubsection{Proof of Theorem~\ref{thm:sub-treecount}} \label{sec:sub-tree-count-general-rif}
	\begin{proof}[Proof of Theorem~\ref{thm:sub-treecount}]
		In the proof we seek to apply Theorem~\ref{thm:structure}. We first apply Item~\ref{item:structure-1} of Theorem~\ref{thm:structure} to show that, if~\eqref{eq:tree-summ-cond} is satisfied, $T$ appears as a sub-tree of $\mathcal{T}_{\infty}$ infinitely often. That is, we show that Condition~\ref{item:treediv-cond} of Assumption~\ref{ass:structure} is satisfied when assuming Equation~\eqref{eq:tree-summ-cond} holds. As we assume that Equation~\ref{eq:starindep} holds as well, this implies Item~\ref{item:structure-1} of Theorem~\ref{thm:structure}.
		
		In a similar manner to the proof of Theorem~\ref{thm:star-path-rif} above, we  choose $\nu^{w}_{n} = d \mu^{w}_{n}$, for $d < 1$. This way, Equation~\eqref{eq:smallest-expl-prob-sub} follows from Equation~\eqref{eq:tail-lower-bound-prob}, using the Paley-Zygmund inequality. To show~\eqref{eq:div-condition-sub-tree}, we first condition on an ordering $O \in \mathcal{O}$ in which the vertices of $T$ appear. In particular, given such an ordering, by extending the definitions of the stopping times associated with the $(X,W)$-CMJ process, we can define the stopping times
		\[
		\tau_{O_{|j}} := \inf{\left\{t \geq 0: O_{|j} \subseteq \mathscr{T}_{t} \right\}}.  
		\]
		Now, with $\cO(T)$ the set of possible orderings of the vertices of $T$, we wish to show that 
		\be \label{eq:req-tree-div}
		\sum_{n=1}^{\infty} \Prob{T \subseteq \mathscr{T}_{d \mu^{w}_{n}}} = \sum_{n=1}^{\infty} \sum_{O \in \mathcal{O}(T)} \Prob{0 = \tau_{O_{|_0}} \leq \tau_{O_{|_1}} \cdots \leq \tau_{O_{|_k}} \leq d \mu^{w}_{n}} = \infty. 
		\ee
		We make the following observations:
		\begin{enumerate}[label = (\Roman*)]
			\item \label{item:roman-1} The probability of seeing an ordering $O_{|j}, j = 0, \ldots, k$, depends on the minima of the exponential random variables associated with the next vertex in the tree to appear. Suppose that the vertices of $T$ are $\left\{v_{0}, \ldots, v_{k}\right\}$, labelled by order of appearance, and that $v^{+}_{j} \in \{v_0, \ldots, v_{j-1}\}$ is the \emph{parent} of $v_{j}$ for each $j\in[k]$. Then, conditionally on the weights of the nodes in $T$, $W_{v_0}, \ldots, W_{v_k}$, the ordering $(O_{|j})_{j = 0, \ldots, k}$ occurs with probability
			\[
			\Prob{O \, \bigg  | \, W_{v_0}, \ldots, W_{v_{k}}}  =  \prod_{j=0}^{k} \frac{f(\outdeg{v^{+}_{j}}, W_{v^{+}_{j}})}{\sum_{i=0}^{j}f(\outdeg{v_i, O_{|_j}},W_{v_i}) \mathbf{1}_{\outdeg{v_i, O_{|_j}} < \outdeg{v_i, T}}}.   
			\]
			The $j$th product in the above equation denotes the probability that the next exponential clock associated with the process, amongst those yet to ring, corresponds to $v^{+}_{j}$, and hence respects the ordering $O$. 
			\item \label{item:roman-2} We may write
			\be 
			\prod_{j=0}^{k} f(\outdeg{v^{+}_{j},O_{|_j}}, W_{v^{+}_{j}})=\prod_{j=0}^k\prod_{\ell=0}^{\outdeg{v_j, T} - 1} f(\ell, W_{v_j}). 
			\ee
			\item \label{item:roman-3} By properties of the exponential distribution, the distribution of the minimum of a collection of independent exponential random variables is independent of the index that attains that minimum. Thus, conditionally on seeing the ordering $O$, the waiting times $(\tau_{O_{|_{j+1}}} - \tau_{O_{|_{j}}})_{j=0,\ldots, k}$ are independent, with 
			\[
			\tau_{O_{|_{j+1}}} - \tau_{O_{|_{j+1}}} \sim \Exp{\sum_{i=0}^{j}f(\outdeg{i, O_{|_j}},W_i) \mathbf{1}_{\outdeg{i, O_{|_j}} < \outdeg{i, T}}},\qquad j=0,\ldots, k. 
			\]
			Thus, by Lemma~\ref{lem:exponential-concentration}, we have 
			\begin{linenomath}
				\begin{align*}
					& \Prob{0 = \tau_{O_{|_0}} \leq \tau_{O_{|_1}} \leq \cdots \leq \tau_{O_{|_k}} \leq d \mu^{w}_{n} \, \bigg | \, W_{v_0}, \ldots, W_{v_{k}}, O} \\ & \hspace{2cm}  \geq \Big(\frac{d}{k+1}\Big)^{k+1}\prod_{j=0}^{k} \frac{\sum_{i=0}^{j}f(\outdeg{i, O_{|_j}},W_i) \mathbf{1}_{\outdeg{i, O_{|_j}} < \outdeg{i, T}}}{\sum_{i=0}^{j}f(\outdeg{i, O_{|_j}},W_i) \mathbf{1}_{\outdeg{i, O_{|_j}} < \outdeg{i, T}} + (\mu^{w}_{n})^{-1}}
					\\ & \hspace{2cm}  \geq \Big(\frac{d c_{1}(w)}{k+1}\Big)^{k+1}\prod_{j=0}^{k} \frac{\sum_{i=0}^{j}f(\outdeg{i, O_{|_j}},W_i) \mathbf{1}_{\outdeg{i, O_{|_j}} < \outdeg{i, T}}}{\sum_{i=0}^{j}f(\outdeg{i, O_{|_j}},W_i) \mathbf{1}_{\outdeg{i, O_{|_j}} < \outdeg{i, T}} + \mu_{n}^{-1}}, 
				\end{align*}
			\end{linenomath}
			where the last line follows from our assumption that $\mu^{w}_{n} \geq c_{1}(w) \mu_{n}$, with $c_{1}(w) < 1$.
		\end{enumerate}
		
		\noindent Combining Items~\ref{item:roman-1}-\ref{item:roman-3} and taking expectations over the weights $W_{v_{0}}, \ldots, W_{v_{k}}$, we deduce that Equation~\eqref{eq:req-tree-div} holds if Equation~\eqref{eq:tree-summ-cond} is satisfied. 
		
		For the converse statement, if Equation~\eqref{eq:tree-summ-cond} is not satisfied and the sum instead converges, we  now apply Item~\ref{item:structure-2} of Theorem~\ref{thm:structure} in a similar manner. In particular, we now wish to show that, for any $w \in S$,
		\be \label{eq:req-tree-conv}
		\sum_{i=1}^{\infty} \Prob{T \subseteq \mathscr{T}_{\sum_{j=i+1}^{\infty} \widetilde{X}_{w}(j)}} = \sum_{i=1}^{\infty} \sum_{O \in \mathcal{O}(T)} \Prob{0 = \tau_{O_{|_0}} \leq \tau_{O_{|_1}} \cdots \leq \tau_{O_{|_k}} \leq \sum_{j=i+1}^{\infty} \widetilde{X}_{w}(j)} < \infty. 
		\ee
		But now, for each $j$ we have $\widetilde{X}_{w}(j) \leq_{S} \widetilde{X}_{w^{*}}(j)$, where $\widetilde{X}_{w^{*}}(j) \sim X_{w^{*}}(j)$, since we assume that Equation~\eqref{eq:minfass} holds. Thus, using the same observations as above but instead looking for an upper bound in Item~\ref{item:roman-3}, we apply Equation~\eqref{eq:exp-conc-2} in Lemma~\ref{lem:exponential-concentration} with $c < 1$ fixed, to see that
		\begin{linenomath}
			\begin{align} \label{eq:comp-waiting-conv}
				& \Prob{0 = \tau_{O_{|_0}} \leq \tau_{O_{|_1}} \cdots \leq \tau_{O_{|_k}} \leq \sum_{j=i+1}^{\infty} \widetilde{X}_{w}(j) \, \bigg | \, W_{v_0}, \ldots, W_{v_{k}}, O} \\ & \hspace{2cm}  \leq \Prob{0 = \tau_{O_{|_0}} \leq \tau_{O_{|_1}} \cdots \leq \tau_{O_{|_k}} \leq \sum_{j=i+1}^{\infty} \widetilde{X}_{w^{*}}(j) \, \bigg | \, W_{v_0}, \ldots, W_{v_{k}}, O} 
				\\ & \hspace{2cm} \stackrel{\eqref{eq:exp-conc-2}}{\leq} \frac{1}{1-c} \prod_{j=0}^{k} \frac{\sum_{i=0}^{j}f(\outdeg{i, O_{|_j}},W_i) \mathbf{1}_{\outdeg{i, O_{|_j}} < \outdeg{i, T}}}{\sum_{i=0}^{j}f(\outdeg{i, O_{|_j}},W_i) \mathbf{1}_{\outdeg{i, O_{|_j}} < \outdeg{i, T}} + c\mu_{i}^{-1}}
				\\ \\ & \hspace{2cm} \leq \frac{1}{(1-c)c^{k+1}} \prod_{j=0}^{k} \frac{\sum_{i=0}^{j}f(\outdeg{i, O_{|_j}},W_i) \mathbf{1}_{\outdeg{i, O_{|_j}} < \outdeg{i, T}}}{\sum_{i=0}^{j}f(\outdeg{i, O_{|_j}},W_i) \mathbf{1}_{\outdeg{i, O_{|_j}} < \outdeg{i, T}} + \mu_{i}^{-1}}.
			\end{align}
		\end{linenomath}
		Again, combining Equation~\eqref{eq:comp-waiting-conv} with Items~\ref{item:roman-1} and~\ref{item:roman-2} and taking expectations over the weights $W_{v_{0}}, \ldots, W_{v_{k}}$, we deduce that Equation~\eqref{eq:req-tree-conv} holds if Equation~\eqref{eq:tree-summ-cond} is not satisfied. Hence, Condition~\ref{item:treedomsum} of Assumption~\ref{ass:structure} holds, so that Item~\ref{item:structure-2} in Theorem~\ref{thm:structure} yields the desired result.
	\end{proof}
	
	\subsection{Possibility for both a star and infinite path to occur with positive probability} \label{sec:counter}
	The goal of this section is to produce the counter-example in Theorem~\ref{thm:counter-example}. We first show, in a similar spirit to \cite[Claim~2.3]{Komjathy2016ExplosiveCB}, that infinite paths, when they appear in finite time, can appear arbitrarily fast. 
	
	\begin{lemma} \label{lem:arbitrarily-fast-path}
		Let $(\mathscr{T}_{t})_{t \geq 0}$ be an explosive $(X,W)$-CMJ process, and suppose that $\mathcal{B}(1) > 0$ almost surely. It is either the case that $\Prob{\tau_{\Path}(\varnothing) < \infty} = 0$, or, for any $\eps > 0$,
		$\Prob{\tau_{\Path}(\varnothing) < \eps} > 0$.
	\end{lemma}
	
	\begin{proof}
		Suppose that $\Prob{\tau_{\Path}(\varnothing) < \infty} > 0$. Then, we define 
		\[
		\eps^{*} := \inf \left\{\eps \geq 0: \Prob{\tau_{\Path}(\varnothing) < \eps} > 0 \right\}.
		\]
		By assumption, we have $\eps^{*} < \infty$, and moreover, $\Prob{\tau_{\Path}(\varnothing) < \eps^{*}} = 0$  (by monotone convergence). Suppose that $\eps^{*} > 0$. As an infinite path from $\varnothing$ must pass through the first generation, we have   
		\begin{equation} \label{eq:expl-first-gen}
			\Prob{\tau_{\Path}(\varnothing) < \eps^{*}} = \Prob{\bigcup_{i \in \mathbb{N}} \left\{\mathcal{B}(i) + \tau_{\Path}(i) < \eps^{*}\right\}}.
		\end{equation}
		Since for each $i$ we have $\mathcal{B}(i) \geq \mathcal{B}(1)$, we must have $\Prob{\mathcal{B}_{1} + \tau_{\Path} < \eps^{*}} > 0$; where $\mathcal{B}_{1} \sim \mathcal{B}(1) \sim X(1)$ and $\tau_{\Path} \sim \tau_{\Path}(\varnothing)$, and $\mathcal{B}_1, \tau_{\Path}$ are independent. Indeed, otherwise a union bound would show that the right-hand side of~\eqref{eq:expl-first-gen} equals zero. But then, it must be the case that
		\begin{equation}
			\left\{\mathcal{B}_{1} \in (0, \eps^{*}): \Prob{\tau_{\Path} < \eps^{*} - \mathcal{B}_1 \, \bigg| \,  \mathcal{B}_1} > 0\right\}
		\end{equation}
		must be a set of positive $\mathbb{P}$-measure. Again, as in (the proof of) Lemma~\ref{lem:no-atom}, this implies that with respect to the distribution induced by $\mathcal{B}_{1}$, the set 
		\begin{equation} 
			\left\{x \in (0, \eps^{*}): \Prob{\tau_{\Path} <  \eps^{*} - x} > 0 \right\}
		\end{equation}
		has positive measure. In particular, there exists an $x \in (0, \eps^{*})$ such that $\Prob{\tau_{\Path} <  \eps^{*} - x} > 0$, a contradiction. We deduce that $\eps^{*} = 0$.
	\end{proof}
	
	\subsubsection{Proof of Theorem~\ref{thm:counter-example}}
	
	\begin{proof}[Proof of Theorem~\ref{thm:counter-example}]
		We define an explosive $(X,W)$-CMJ process  with weights $W_{u} = (R_{u}, I_{u}) \in (0, \infty) \times \{0,1\}$ for $u\in\cU_\infty$. Here $R_{u}$ and $I_{u}$ are independent and distributed such that $\Prob{I_{u} = 1} = \Prob{I_{u}=0} = 1/2$, and $R_{u}$ has a smooth distribution such that, for some constant $\alpha > 1$, for all $x \in [1, \infty)$,
		\be \label{eq:power-law}
		\Prob{R_u > x} \geq x^{-(\alpha -1)}. 
		\ee 
		We also fix $p > 1$ such that $(\alpha-1)(p-1) < 1$. We can think of the variable $I_{u}$ as an indication of the \emph{type} of a node $u$,  effecting the distribution of $(X_{W_{u}}(uj))_{j \in \mathbb{N}}$.  
		We then define the $(X,W)$-CMJ process $(\mathscr{T}_{t})_{t \geq 0}$ so that, conditionally on $W_{u}$, the random variables $(X_{W_{u}}(ui))_{ i \in \mathbb{N}}$ are mutually independent, and for each $u\in\cU_\infty,i \in \mathbb{N}$,
		\[
		X_{W_u}(ui) \sim \begin{cases}
			\Exp{R_{u}i^{p}}, &\mbox{if } W_{u} = (R_u, 0); \\
			\Exp{1}, &\mbox{if } W_{u} = (R_u, 1), i=1; \\
			s_{i}, &\mbox{if } W_{u} = (R_u, 1), i > 1; 
		\end{cases}
		\]
		where the $(s_{i})_{i \geq 2}$ are constants with each $s_{i} > 0$ defined to satisfy the following: if for each $k \geq 1$ we set
		\begin{equation} \label{eq:det-exp-bound}
			\varsigma_{k} := \sum_{j=k+1}^{\infty} s_{j} < \infty, \quad \text{ then we have } \quad  \E{ \e^{-(R_{k} \vee 1) \varsigma_{k}}} >1- 2^{-k};
		\end{equation}
		where $R_{k} \sim R_{\varnothing}$ is the the first element of the weight associated with the individual $k \in \mathcal{U}_{\infty}$. We stress that this condition is satisfied when  we choose the constants $s_i$ sufficiently small. Note that under this construction, the conditions of Assumption~\ref{ass:uniqueness} are satisfied: clearly, in general $\mathcal{B}(1) \sim  X(1)$ contains no atom on $[0,\infty)$, so Condition~3 is satisfied. Moreover, for any $z \geq 0$, by the conditional independence of $(X_{W_{u}}(ui))_{i \in \mathbb{N}}$,
		\[
		\Prob{\sum_{i=1}^{\infty} X_{W_{u}}(ui) = z\, \bigg | \, W_{u}} =  \Prob{X_{W_u}(u1) = z - \sum_{i=2}^{\infty} X_{W_u}(ui)\, \bigg | \, W_{u}}= 0,
		\]
		since $X_{W_{u}}(u1)$ always has an exponential (hence smooth) distribution. Note also, that $\Prob{\tau_{\infty} < \infty} = 1$, and $\Prob{|\mathcal{T}_{\infty}| = \infty} = 1$. Now, we have the following claim:
		
		\begin{clm} \label{clm:child-explodes-first-on-root}
			For all $u \in \mathcal{U}_{\infty}$, 
			\begin{equation} \label{eq:child-explodes-first-on-root}
				\Prob{\bigcap_{j=1}^{\infty}\bigcup_{i = j}^{\infty}\left\{\mathcal{B}(ui) + \mathcal{P}(ui) < \mathcal{B}(u) + \mathcal{P}(u), I_{ui} = 0\right\} \, \bigg | \, I_{u} = 0} = 1. 
			\end{equation}  
		\end{clm}
		
		\noindent Given Claim~\ref{clm:child-explodes-first-on-root}, in a similar manner to the proof of Theorem~\ref{thm:path} (which uses~\eqref{eq:child-explodes-first-forall}), we can exploit Equation~\eqref{eq:child-explodes-first-on-root} to deduce that $\Prob{\tau_{\Path}(\varnothing) < \infty \,| \, I_{\varnothing} = 0} = 1$. In particular, we can ensure that $\tau_{\Path}(\varnothing) \leq \mathcal{P}(\varnothing)$. Indeed, Equation~\eqref{eq:child-explodes-first-on-root} guarantees that there exists a child $i$ such that $\mathcal{B}(i) + \mathcal{P}(i) \leq \mathcal{P}(\varnothing)$; this child, in turn, by~\eqref{eq:child-explodes-first-on-root} has a child of infinite degree by time $\mathcal{P}(\varnothing)$. Iterating this argument, we have an infinite path by time $\mathcal{P}(\varnothing)$. 
		
		Now, it cannot be the case that there exists $u \in \mathcal{U}_{\infty}$ such that $I_{u} = 0$ and $\tau_{\infty} = \mathcal{B}(u) + \mathcal{P}(u)$. Indeed, assume such a $u$ does exist. Then, by the above argument it would follow that $\tau_{\Path}(u) = \mathcal{B}(u) + \mathcal{P}(u) = \tau_{\infty}$. Hence $u$ is contained in an infinite path \emph{and} has infinite degree in $\mathcal{T}_{\infty}$, contradicting Theorem~\ref{thm:uniqueness}. On the other hand, for any $u$ such that $I_{u} = 1$, by construction, we have $\mathcal{P}(u) \geq \varsigma_{1}$. However, by Lemma~\ref{lem:arbitrarily-fast-path} we deduce that, since $\Prob{\tau_{\Path}(\varnothing) < \infty} > 0$ by the above argument, we have $\Prob{\tau_{\Path}(\varnothing) < \varsigma_{1}} > 0$. It thus follows that, on $\{\tau_{\Path}(\varnothing) < \varsigma_{1}\}$ the tree $\mathcal{T}_{\infty}$ contains an infinite path almost surely. 
		
		Now, suppose that $I_{\varnothing} = 1$, an event that also occurs with probability $1/2$. Then, for each child $k \in \mathbb{N}$ (in the first generation), 
		\begin{linenomath}
			\begin{align} \label{eq:prob-child-born}
				\Prob{\mathcal{B}(k1) < \mathcal{P}(\varnothing) \, | \, I_{\varnothing} = 1} & = \Prob{X(k1) < \!\!\!\sum_{j=k+1}^{\infty} \!\! X(j) \, \bigg | \, I_{\varnothing} = 1} = \Prob{X(k1) < \varsigma_{k}} \leq 1-\E{ \e^{-(R_{k} \vee 1) \varsigma_{k}}},
			\end{align}
		\end{linenomath}
		where the last inequality follows from the fact that $X(k1)$ is exponentially distributed, either with parameter $1$ or $R_{k}$, and the probability is maximised if we choose the maximum of the two. But then, by~\eqref{eq:det-exp-bound} and a union bound,
		\[
		\Prob{\bigcup_{k \in \mathbb{N}} \mathcal{B}(k1) < \mathcal{P}(\varnothing) \, \bigg | \, I_{\varnothing} = 1} < \sum_{k=1}^{\infty} 2^{-k} = 1. 
		\]
		It follows that, with positive probability, $\mathcal{T}_{\infty}$ consists of a single infinite star and hence contains no infinite path. 
	\end{proof}
	
	\noindent  We conclude with the proof of Claim~\ref{clm:child-explodes-first-on-root}.
	
	\begin{proof}[Proof of Claim~\ref{clm:child-explodes-first-on-root}]
		By using a similar Borel-Cantelli argument as in the proof of Equation~\eqref{eq:child-explodes-first} of Lemma~\ref{lemma:explodingchild} (as in Equations~\eqref{eq:div-condition} and~\eqref{eq:smallest-expl-prob}), it suffices to prove the following:
		there exists a collection of numbers $\left\{\nu^{r}_{n} \in (0, \infty): r \in (0, \infty), n \in \mathbb{N}\right\}$, such that for any $r \in (0, \infty)$,
		\begin{align} \label{eq:div-condition-counterexample}
			\sum_{n=1}^{\infty} \Prob{\mathcal{P} < \nu^{r}_{n}, I_{n} = 0} &= \infty,
			\shortintertext{and} 
			\hspace{-1cm}\liminf_{n \to \infty} \Prob{\sum_{i=n+1}^{\infty} X_{(r,0)}(i) \geq \nu^{r}_{n}}& > 0.\label{eq:smallest-expl-prob-counterexample}
		\end{align}
		Here, we set 
		\be 
		\nu^{r}_{n} := d \E{\sum_{i=n+1}^{\infty} X_{(r,0)}(i)} = d \sum_{i=n+1}^{\infty} \frac{1}{r i^{p}}, 
		\ee 
		for some $d < 1$. Then, note that we may deduce Equation~\eqref{eq:smallest-expl-prob-counterexample} by a similar application of the Paley-Zygmund inequality as in Equation~\eqref{eq:tail-lower-bound-prob}. Moreover, applying Lemma~\ref{lem:lower-bound-infinite-events} with $b_{n} = \nu^{r}_{n}$ and $\eps > 0$ fixed  and using the independence of $R_n$ and $ I_n$, for $n \in \mathbb{N}$, we have 
		\begin{linenomath}
			\begin{align*}
				\Prob{\mathcal{P} <\nu^{r}_{n}, I_{n} = 0} = \frac{1}{2} \Prob{\sum_{j=1}^{\infty} X_{(R_{n}, 0)}(j) <\nu^{r}_{n}} \geq \frac{1}{2} \left(\E{\prod_{j=1}^{\infty} \frac{R_{n} j^{p}}{R_{n} j^{p} + (1+\eps) (\nu^{r}_{n})^{-1} \log(n)}} - \frac{1}{n^{1+\eps}} \right).
			\end{align*}
		\end{linenomath}
		It thus suffices to show that 
		\be \label{eq:counter-divergence-cond}
		\sum_{n=1}^{\infty} \E{\prod_{j=1}^{\infty} \frac{R_{n} j^{p}}{R_{n} j^{p} + (1+\eps) (\nu^{r}_{n})^{-1} \log(n)}} = \infty.
		\ee
		To this end, we note that we may bound 
		\be \label{eq:lower-bound-denom-counter}
		(1+ \eps) (\nu^{r}_{n})^{-1} \log(i) = \frac{ (1+\eps) r}{d} \log(n) \left(\sum_{j=n+1}^{\infty} \frac{1}{j^{p}}\right)^{-1} <C_1\log(n) n^{p-1},
		\ee
		where we bound the sum from below by an integral and where $C_1>0$ is a constant. For $n$ sufficiently large such that $(1+\eps)dr  \log(n) n^{p-1} > 1$, we now bound 
		\begin{linenomath}
			\begin{align}  \label{eq:exp-lower-bound-counter}
				\E{\prod_{j=1}^{\infty} \frac{R_{n} j^{p}}{R_{n} j^{p} + (1+\eps) (\nu^{r}_{n})^{-1} \log(n)}} & \geq \E{\prod_{j=1}^{\infty} \frac{R_{n} j^{p}}{R_{n} j^{p} + C_1  \log(n) n^{(p-1)}} \mathbf{1}_{R_{n} \geq C_1 \log(n) n^{(p-1)}}} 
				\\ & \geq \left(\prod_{j=1}^{\infty} \frac{j^{p}}{j^{p} + 1} \right) \Prob{R_{n} \geq C_1  \log(n) n^{p-1}}
				\\ & \geq C_2\left(\prod_{j=1}^{\infty} \frac{j^{p}}{j^{p} + 1} \right) 
				\big(\log(n) n^{p-1}\big)^{-(\alpha-1)},
			\end{align}
		\end{linenomath}
		for some constant $C_2>0$ and where we use~\eqref{eq:power-law} in the final step. Since $p>1$, $\sum_{j=1}^{\infty} j^{-p} < \infty$, which implies that the infinite product on the right-hand side is strictly larger than zero. Since $(\alpha-1)(p-1) < 1$, the lower bound on the right-hand side is not summable in $i$, so that we obtain~\eqref{eq:counter-divergence-cond}, which concludes the proof.
	\end{proof}

	\section{Examples of phase transitions in specific models of recursive trees with fitness}\label{sec:examplesproof}
	
	The aim of this section is to prove Theorems~\ref{thrm:cmjexamples} and~\ref{thrm:sub-treecount}. The previous section provided conditions for the emergence of a unique vertex with infinite degree or a unique infinite path to appear almost surely in the case the inter-birth times are exponential random variables whose rate depends on some fitness function $f$. In this section we turn these conditions into 
	phase transitions for three specific examples. We start by proving Theorem~\ref{thrm:cmjexamples} in Section~\ref{sec:thm-star-phase} and then prove Theorem~\ref{thrm:sub-treecount} in Section~\ref{sec:sub-treeproof}.
	
	We also note that, in this section, we use the commonly applied `big O', `little o' and `big theta' notation: we say $f(x) = \mathcal{O}(g(x))$ if there exists positive constants $M,x_0$ such that, for all $x \geq x_{0}$, we have $|f(x)| \leq M g(x)$; we say $f(x) = o(g(x))$ if $\lim_{x\to\infty} |f(x)/g(x)|=0$; and finally, we say $f(x) = \Theta(g(x))$ if $f(x) = \mathcal{O}(g(x))$ and $g(x)=\cO(f(x))$.
	
	\subsection{Phase transitions for the emergence of an infinite star or path: proof of Theorem~\ref{thrm:cmjexamples}} \label{sec:thm-star-phase}
	
	To prove Theorem~\ref{thrm:cmjexamples}, it suffices to check the conditions subject to which Theorem~\ref{thm:star-path-rif} holds. In the remainder of this section, it is be useful to have the following preliminary results from~\cite{BinGolTeu89} that we use throughout. 
	
	\begin{lemma}[{\cite[Propositions 1.3.6, 1.5.7, 1.5.9a, and 1.5.10]{BinGolTeu89}}]\label{lemma:regvar}
		Let $\ell$ be a slowly-varying function and $r_1,r_2$ be regularly-varying functions with exponents $\alpha_1,\alpha_2\in\R$, respectively, and fix $a\in \R$. Then, 
		\begin{enumerate}[label=\normalfont(\roman*)]
			\item \label{item:slow-neg} For $a>0$, it holds that $\lim_{x\to\infty} \ell(x)x^a=\infty$ and $\lim_{x\to\infty}\ell(x)x^{-a}=0$. 
			\item The function $r_1^a:=(r_1(\cdot))^a$ is regularly varying with exponent $\alpha_1a$.
			\item If $\alpha_2>0$, then $r_1(r_2(\cdot))$ is regularly varying with exponent $\alpha_1\alpha_2$. 
			\item The function $r_1+r_2$ is regularly varying with exponent $\max\{\alpha_1,\alpha_2\}$. 
			\item The function $r_1r_2$ is regularly varying with exponent $\alpha_1+\alpha_2$.
			\item\label{item:reg-var-int}  If $\alpha_1<-1$, then 
			\be 
			\int_x^\infty r_1(t)\, \dd t=(-\alpha_1-1+o(1))xr_1(x), \qquad \text{as }x\to\infty.
			\ee 
			When $\alpha_1=-1$, the result remains true in the sense that the integral is slowly varying in $x$, and that it is $o(xr_1(x))$.
		\end{enumerate}
	\end{lemma}
	
	\noindent We generally use Lemma~\ref{lemma:regvar} without always referring to it explicitly in the remainder of this section. We mainly use it in the following manner. First, Lemma~\ref{lemma:regvar} shows that regularly-varying functions are closed under elementary operations. Next, Item~\ref{item:reg-var-int} shows that, if $s(n)$ is regularly varying with exponent $p > 0$, then $\sum_{i=n}^{\infty} s(i)^{-1}$ is regularly varying with exponent $p-1$. In addition, Item~\ref{item:slow-neg} shows that slowly-varying functions are often "negligible" in the sense that, for a regularly-varying function $f$ with exponent $\gamma\neq0$, we can write $f(x)=x^{\gamma+o(1)}$. Moreover, it shows that for any regularly-varying function $s$ with exponent $p > 1$ that $\sum_{i=1}^{\infty} s(i)^{-1} < \infty$, whilst for any regularly-varying function $s$ with exponent $p < 1$ we have $\sum_{i=1}^{\infty} s(i)^{-1} = \infty$. 
	
	The following approximations for $\mu_n^w$, as defined in~\eqref{eq:mu-def}, are also used heavily.
	
	\begin{lemma}\label{lemma:mun}
		Consider the two choices for the fitness function $f$ $($mixed or additive weights$)$, as in Assumption~\ref{ass:f}, as well as  the two choices for the degree function $s$ (super-linear and barely super-linear), as in Assumption~\ref{ass:deg}. For the \hyperlink{superlinear}{super-linear} case and for each fixed $w\in[0,\infty)$,
		\begin{align}
			\mu_n^w &=\frac{1+o(1)}{g(w)(p-1)}ns(n)^{-1}.
			\shortintertext{For the \hyperlink{barsuperlinear}{barely super-linear} case there exists a slowly-varying function $L$ such that for each $w\in[0,\infty)$,}
			\mu_n^w&= \frac{1+o(1)}{g(w)}L(n).
			\shortintertext{In particular, for the barely super-linear \hyperlink{logstrechted}{log-stretched} case, it follows that}
			L(n)&=\beta^{-1}(\log n)^{1-\beta}\e^{-(\log n)^\beta}.
		\end{align}
	\end{lemma}
	
	\begin{proof}
		We fix $w\in[0,\infty)$ throughout and generally use that 
		\be \label{eq:genmu}
		\mu_n^w=\sum_{i=n}^\infty \frac{1}{f(i,w)}=\sum_{i=n}^\infty \frac{1}{g(w)s(i)+h(w)}=\frac{1+o(1)}{g(w)}\sum_{i=n}^\infty \frac{1}{s(i)}, 
		\ee 
		where we recall that $g\equiv 1$ in the additive case, and we can omit the term $h(w)$ in the fraction at the cost of a $o(1)$ term. Taking $s$ as in the~\hyperlink{superlinear}{super-linear} case, since $s$ is assumed to be regularly varying with exponent $p > 1$, by the integral test and using Lemma~\ref{lemma:regvar}, we thus find that
		\be
		\mu_n^w=\frac{1+o(1)}{g(w)(p-1)}ns(n)^{-1}. 
		\ee 
		With $s$ as in the barely super-linear case, it follows from $(vi)$ in Lemma~\ref{lemma:regvar} and an integral test that 
		\be 
		L(n):=\int_n^\infty \frac{1}{s(x)}\,\dd x.
		\ee 
		For the~\hyperlink{log-stretched}{log-stretched case}, we have with $s$ as in~\eqref{eq:slogstretched},
		\be \label{eq:munlogstretchedexpup}
		L(n)=\int_{n+1}^\infty x^{-1}\exp(-(\log x)^\beta)\,\dd x.
		\ee 
		Using a change of variables $y=(\log x)^\beta$, we obtain, with $\Gamma(s,x),s,x>0,$ the upper incomplete gamma function,\footnote{The upper incomplete gamma function is defined, for $s,x>0$ by $\Gamma(s,x):=\int_{x}^{\infty} t^{s-1}\e^{-t} \dd t$.}
		\be 
		L(n)=\beta^{-1}\Gamma(1/\beta,(\log(n+1))^\beta)=(\beta^{-1}+o(1))(\log n)^{1-\beta}\e^{-(\log n)^\beta},
		\ee 
		which concludes the proof.
	\end{proof}
	
	\subsubsection{Conditions for an infinite star: Item~\ref{item:star-explosive-rif} of Theorem~\ref{thm:star-path-rif}.} 
	
	To provide conditions for an infinite star to appear almost surely, we apply Theorem~\ref{thm:star-path-rif} by proving the condition in Item~\ref{item:star-explosive-rif} is satisfied in certain cases. To this end, we state the following lemma. 
	\begin{lemma}\label{lemma:checklaplace}
		Equation~\eqref{eq:star-explosive-rif} in Item~\ref{item:star-explosive-rif} of Theorem~\ref{thm:star-path-rif} is satisfied when the following conditions are met, based on the assumptions for the fitness type, degree function $s$, and vertex-weight distribution:
		\begin{table}[H]
			
			\centering
			\footnotesize
			\captionsetup{width=0.89\textwidth}
			
			\begin{tabular}{|c|c|l|}
				\hline
				\textbf{Weight} & \textbf{Degree} &  \textbf{Star}  \\  
				\hline
				\hyperlink{mixed}{Mixed} & \hyperlink{superlinear}{Super-linear} & \eqref{eq:weightasspowerlawub} \& $(p-1)(\alpha-1)>\big(\gamma-\tfrac{\gamma-1}{p}\big)\vee 1$  \\ 
				\hline 
				\hyperlink{additive}{Additive} & \hyperlink{superlinear}{Super-linear} & \eqref{eq:weightasspowerlawub} \& $p(\alpha-1)>1$\\
				\hline
				\hyperlink{mixed}{Mixed} & \hyperlink{log-stretched}{Log-stretched} & \eqref{eq:weightasslogstrechtedub} \& $\beta\nu>1$ \\    
				\hline 
			\end{tabular}
			\caption{\footnotesize The first column represents the form of the fitness function, as in Assumption~\ref{ass:f}; the second represents the form of the degree function, as in Assumption~\ref{ass:deg}. The third column lists the required assumptions on the vertex-weight distribution, as in Assumption~\ref{ass:weights}, together with the choices of the parameters that lead to a unique node of infinite degree.}
		\end{table}
	\end{lemma}

	\noindent We split the proof of Lemma~\ref{lemma:checklaplace} based on the different choices for the fitness function $f$ and degree function $s$, as in Assumptions~\ref{ass:f} and~\ref{ass:deg}, respectively. As a general approach, we bound, for some sequence $(k_n)_{n\in\N} \in [0, \infty)^{\mathbb{N}}$, 
	\be 
	\E{\prod_{j=0}^\infty\frac{f(j,W)}{f(j,W)+c\mu_n^{-1}}} \leq \prod_{j=0}^\infty \frac{\sup_{w\leq k_n}f(j,w)}{\sup_{w\leq k_n}f(j,w)+c\mu_n^{-1}}+\P{W\geq k_n}.
	\ee 
	We now use that, by Assumption~\ref{ass:f}, for $x\geq 0$ and $j\in\N$,
	\be \label{eq:fsupub}
	\sup_{w\leq x} f(j-1,w)\leq\big(\sup_{w\leq x} g(w)\big)s(j-1)+\sup_{w\leq x}h(w)=: \overline g(x) s(j-1)+\overline  h(x),
	\ee 
	where the suprema are well-defined as we assume $g$ and $h$ to be continuous. Since $g$ and $h$ are regularly varying with exponents $1$ and $\gamma$ in the~\hyperlink{mixed}{mixed} case and $g\equiv 1$ and $h$ is regularly varying with exponent $1$ in the~\hyperlink{additive}{additive} case, it follows from~\cite[Theorem $1.5.3$]{BinGolTeu89} that $\overline g(x)=(1+o(1))g(x)$ and $\overline h=(1+o(1))h(x)$; in particular, they are regularly varying with the same exponent (and $\overline g=g\equiv 1$ in the additive case). As such, using the weight functions $\overline g$ and $\overline h$ falls in the same family of weight functions as $g$ and $h$. The advantage of using $\overline g $ and $\overline h$ is that these are monotone increasing. However, by this argument we can thus use $g$ and $h$ instead and assume, without loss of generality, that they are monotone increasing. This yields 
	\be \label{eq:splitbound}
	\E{\prod_{j=0}^\infty\frac{f(j,W)}{f(j,W)+c\mu_n^{-1}}} \leq \exp\bigg(-c\mu_n^{-1}\sum_{j=0}^\infty \frac{1}{f(j,k_n)+c\mu_n^{-1}}\bigg)+\P{W\geq k_n}.
	\ee
	Using Assumption~\ref{ass:weights}, one can choose $k_n$ to grow with $n$ sufficiently fast so that the second term on the right-hand side is summable in $n$. It then remains to show that the exponential term is sufficiently small for this choice of $k_n$ as well, to obtain that the left hand side is summable in $n$. 
	
	\subsubsection{Conditions for an infinite path, Assumption~\ref{ass:path}.} 
	
	To prove the appearance of an infinite path, we use Item~\ref{item:path-explosive-rif} of Theorem~\ref{thm:star-path-rif}. In particular we verify Condition~\eqref{eq:div-condition} with $\nu_n^w:=d\mu_n^w$ for some constant $d\in(0,1)$ and $\mu_n^w$ as in~\eqref{eq:mu-def}. Notably, the approach we use to do so works for \emph{any} inter-birth distribution satisfying the variance condition of Remark~\ref{rem:more-general}; it is not tailored to the exponential distribution.
	
	We thus have the following lemma.
	
	\begin{lemma}\label{lemma:pathcond}
		The condition in~\eqref{eq:div-condition} satisfied when the following conditions are met, based on the assumptions for the fitness type, degree function $s$, and vertex-weight distribution: 
		\begin{table}[H]
			
			\centering
			\footnotesize
			\captionsetup{width=0.89\textwidth}
			
			\begin{tabular}{|c|c|l|}
				\hline
				\textbf{Weight} & \textbf{Degree} & \textbf{Path} \\  
				\hline
				\hyperlink{mixed}{Mixed} & \hyperlink{superlinear}{Super-linear} & \eqref{eq:weightasspowerlawlb} \& $(p-1)(\alpha-1)<\big(\gamma-\tfrac{\gamma-1}{p}\big)\vee 1$ \\ 
				\hline 
				\hyperlink{additive}{Additive} & \hyperlink{superlinear}{Super-linear}  & \eqref{eq:weightasspowerlawlb} \& $p(\alpha-1)<1$ \\
				\hline
				\hyperlink{mixed}{Mixed} & \hyperlink{log-stretched}{Log-stretched} & \eqref{eq:weightasslogstrechtedlb} \& $\beta \nu < 1$ \\
				\hline 
			\end{tabular}
			\caption{\footnotesize The first column represents the form of the fitness function, as in Assumption~\ref{ass:f}; the second represents the form of the degree function, as in Assumption~\ref{ass:deg}. The third column lists the required assumptions on the vertex-weight distribution, as in Assumption~\ref{ass:weights}, together with the choices of the parameters that lead to a unique infinite path.}
		\end{table}
	\end{lemma}
	
	\noindent We observe that combining Lemmas~\ref{lemma:checklaplace} and~\ref{lemma:pathcond} proves Theorem~\ref{thrm:cmjexamples}.\\[0.05cm]
	
	As a general approach to proving Lemma~\ref{lemma:pathcond}, we start with a general lower bound for each of the terms in the sum in~\eqref{eq:div-condition}, and then show that this lower bound is not summable for each of the specific cases dealt with in Lemma~\ref{lemma:pathcond}. We first introduce the event $\{W\geq k_n\}$ for some sequence $(k_n)_{n\in\N}$ such that the tail probability $\P{W\geq k_n}$ is not summable. By a similar argument as in~\eqref{eq:fsupub}, but applied to $\inf_{w\geq x}f(j-1,w)$ (which is monotone increasing in $x$), we may assume, without loss of generality, that $f$ is monotone increasing in its second argument. 
	
	Then, on the event $\{W\geq k_n\}$, we stochastically dominate each inter-birth time $X_{W}(j)$ by $\wt X_{k_n}(j)$, which is distributed as an exponential random variable with rate $f(j-1,k_n)$, and is independent of both $W$ and $X_{W}(j)$, for each $j\in\N$. It follows that $\cP$ stochastically dominates $\wt \cP=\sum_{j=1}^\infty \wt X_{k_n}(j)$, which again is independent of $W$ and $\cP$. We thus obtain the lower bound
	\be 
	\P{\cP < \nu_n^w}\geq \P{\{\cP< \nu_n^w\}\cap\{W\geq k_n\}}\geq \P{\wt\cP< \nu_n^w}\P{W\geq k_n}.
	\ee 
	Since we choose $k_n$ such that the second probability on the right-hand side is not summable, it suffices to show that the first probability is uniformly bounded from below in $n$. By Markov's inequality and the choice of $\nu_n^w$, we obtain the lower bound
	\be\ba\label{eq:tildepnlb}
	\P{\wt \cP< \nu_n^w}&\geq \Big(1-\frac{1}{d\mu_n^w}\E{\wt\cP}\Big)\P{W\geq k_n}\\
	&=\bigg(1-\frac{1}{d\mu_n^w}\sum_{j=0}^\infty \frac{1}{f(j,w)}\bigg)\P{W\geq k_n}. 
	\ea \ee 
	With $k_n$ such that the probability on the right-hand side is not summable and since $d\in(0,1)$ is arbitrary, it thus suffices to prove that 
	\be \label{eq:limsupbound}
	\limsup_{n\to\infty}\frac{1}{\mu_n^w}\sum_{j=0}^\infty \frac{1}{f(j,k_n)}<1. 
	\ee 
	If this holds, it follows that the condition in~\eqref{eq:div-condition} is satisfied. We thus prove~\eqref{eq:limsupbound} for the cases in Lemma~\ref{lemma:pathcond}.
	
	\subsubsection{Proofs of Lemmas~\ref{lemma:checklaplace} and~\ref{lemma:pathcond}}
	
	We now proceed to state the proofs of Lemmas~\ref{lemma:checklaplace} and~\ref{lemma:pathcond}. We do this case by case, and we combine the proofs of the related statements in both lemmas. We use the approaches outlined after the statements of both lemmas.  
	
	\begin{proof}[Proof of Lemmas~\ref{lemma:checklaplace} and~\ref{lemma:pathcond}, $s$ super-linear, mixed weights]
		We first prove the claim in Lemma~\ref{lemma:checklaplace}. We distinguish between $\gamma\leq 1$ and $\gamma>1$, and we start with the former case. Observe that, when $\gamma\leq 1$, it follows that $\max\{\gamma-(\gamma-1)/p,1\}=1$. Hence, we can take $\eps>0$ small enough such that $(1-\eps)(p-1)(\alpha-1)>1 = \max\{\gamma-(\gamma-1)/p,1\}$. We set $\beta:=(1-\eps)(p-1)$ and $k_n:=n^\beta$. We use~\eqref{eq:weightasspowerlawub} to deduce that $\P{W\geq k_n}\leq \overline \ell(k_n)n^{-(\alpha-1)\beta}$, which, as $\beta(\alpha-1) > 1$, is summable. For the first term on the right-hand side of~\eqref{eq:splitbound}, we write
		\be\label{eq:prodbound}
		\exp\Bigg(-\sum_{j=0}^\infty\frac{c\mu_n^{-1}}{f(j,k_n)+c\mu_n^{-1}}\Bigg)=\exp\Bigg(-\frac{c \mu_n^{-1}}{g(n^\beta)}\sum_{j=0}^\infty\frac{1}{s(j)+h(n^\beta)/g(n^\beta)+c\mu_n^{-1}/ g(n^\beta)}\Bigg).
		\ee 
		As $\gamma\leq 1$, it follows that $\beta \gamma < p-1$. Recalling that, by Lemma~\ref{lemma:mun}, $\mu_{n}^{-1}$ is regularly varying with exponent $p-1$, it follows that
		\be \label{eq:a-n-def}
		a_n:=\frac{h(n^\beta)+c\mu_n^{-1}}{g(n^\beta)}
		\ee 
		is regularly varying with exponent $\eps(p-1)$. 
		We now write $s(x)=\ell(x)x^p$ for some slowly-varying function $\ell$. Moreover, since $p>1$, we have $j^p+a_n\leq (j+a_n^{1/p})^p$. As a result, for any $\eta > 0$ there exists $J=J(\eta)\in \N$ such that for all $j\geq J$, 
		\be \label{eq:potterbound}
		s(j)+a_n=\ell(j)j^p+a_n\leq j^\eta (j^p+a_n)\leq (j+a_n^{1/p})^{p+\eta}.
		\ee 
		We can thus bound the sum in~\eqref{eq:prodbound} from below by 
		\be \label{eq:potinsum}
		\sum_{j=J}^\infty\frac{1}{s(j)+a_n}\geq \sum_{j=J}^\infty(j+a_n^{1/p})^{-(p+\eta)} \geq \!\!\!\sum_{j=J+\lceil a_n^{1/p}\rceil}^\infty\!\!\! j^{-(p+\eta)}=\Theta\big(a_n^{-(p+\eta-1)/p}\big), 
		\ee 
		where the final step uses an integral test. Using this in~\eqref{eq:prodbound}, we thus obtain, for some constants $C,C'>0$, the upper bound
		\be 
		\exp\bigg(-C\frac{\mu_n^{-1}a_n^{-(p+\eta-1)/p}}{g(n^\beta)}\bigg)=\exp\big(-C'a_n^{(1-\eta)/p}\big).
		\ee 
		As the term in the exponential varies regularly with exponent $\eps(p-1)(1-\eta)/p>0$, the exponential term is summable. We thus conclude that~\eqref{eq:star-explosive-rif} is satisfied, which concludes the proof for $\gamma\leq 1$.
		
		We then consider the case $\gamma>1$. We first use our assumption, and fix $\eps>0$ small such that 
		\be \label{eq:bineq}
		\frac{(1-\eps)(p-1)(\alpha-1)}{\gamma-(\gamma-1)/p}>1,
		\ee 
		which is possible since $\max\{\gamma-(\gamma-1)/p,1\}=\gamma-(\gamma-1)/p$ when $\gamma, p>1$. 
		We then set $\beta:=(1-\eps)(p-1)/(\gamma-(\gamma-1)/p)$. By the same argument as before, since $\beta(\alpha-1) > 1$, the second term on the right-hand side of~\eqref{eq:splitbound} is summable. To bound the first term, we again write the upper bound 
		\be \label{eq:expplusprob}
		\exp\bigg(-c\mu_n^{-1}\sum_{j=0}^\infty\frac{1}{f(j,n^{\beta})+c\mu_n^{-1}}\bigg)=\exp\bigg(-\frac{c\mu_n^{-1}}{g(n^\beta)}\sum_{j=0}^\infty\frac{1}{s(j)+(h(n^{\beta})+c\mu_n^{-1})/g(n^{\beta})}\bigg). 
		\ee 
		Note that the exponent associated with the regularly-varying function $h(n^{\beta})$ is $\beta \gamma = (1- \eps)(p-1) \gamma/(\gamma - (\gamma - 1)/p) > p-1$ for $\eps$ sufficiently small (since we assume that $\gamma>1$). Thus, making $\eps$ smaller if necessary (which means that~\eqref{eq:bineq} still holds), recalling that, by Lemma~\ref{lemma:mun}, $\mu_{n}$ is regularly varying with exponent $-(p-1)$ in the super-linear case, it follows that, if $a_n:=(h(n^\beta)+c\mu_n^{-1})/g(n^\beta)$, $a_{n}^{1/p}$ is regularly varying with exponent
		\[
		\beta(\gamma -1)/p = \frac{(1-\eps)(p-1)(\gamma-1)}{\gamma(p-1) + 1}.
		\] 
		As a result, using the same approach as in~\eqref{eq:potinsum}, we can bound the sum appearing in the exponent in the right-hand side of~\eqref{eq:expplusprob} from below by
		\be 
		\sum_{j=J}^\infty\frac{1}{s(j)+a_n}\geq  \sum_{j=J}^\infty(j+a_n^{1/p})^{-(p+\eta)} \geq \sum_{j=J+\lceil a_{n}^{1/p} \rceil}^\infty\!\!\! j^{-(p+\eta)} = \Theta(a_n^{-(p+\eta-1)/p}).
		\ee  
		Combining this lower bound with~\eqref{eq:expplusprob}, for some constant $C''>0$, we obtain the upper bound
		\be 
		\exp\Big(-C''\frac{\mu_n^{-1}a_n^{-(p+\eta-1)}}{g(n^{\beta})}\Big).
		\ee 
		Since we can choose $\eta$ arbitrarily small, balancing the exponent of $\mu_{n}^{-1}$ (which is $p-1$) with the exponents of $a_{n}^{-(p-1)/p}$ and $g(n^{\beta})$ it follows that the fraction is regularly varying in $n$ with an exponent that is positive if
		\be 
		p-1> (1- \eps) \left(\frac{(p-1)^2(\gamma-1)}{\gamma(p-1) + 1} + \frac{p(p-1)}{\gamma(p-1) + 1} \right) = (1-\eps)(p-1) \frac{(p-1)(\gamma -1) + p}{\gamma(p-1) + 1} = (1-\eps)(p-1).
		\ee 
		As this inequality is clearly always satisfied for any $\eps > 0$ we obtain that the left-hand side of~\eqref{eq:splitbound} is summable (which implies~\eqref{eq:star-explosive-rif}), as desired.
		
		We then prove the claim in Lemma~\ref{lemma:pathcond}. Again, we distinguish between the cases $\gamma\leq 1$ and $\gamma>1$. Let $\eps>0$ be sufficiently small such that 
		\begin{equation}\label{eq:par}
			\begin{cases}
				(p-1)(\alpha-1)<1-\eps & \text{if $\gamma \leq 1$, }\\
				(p-1)(\alpha-1)/(\gamma-(\gamma-1)/p)<1-\eps & \text{if $\gamma > 1$. }
			\end{cases}
		\end{equation}
		In either case, we set $k_{n} = n^{(1-\eps)/(\alpha-1)}$. Then $\Prob{W \geq k_{n}} \geq \underline{\ell}(k_{n}) n^{-(1- \eps)}$ by~\eqref{eq:weightasspowerlawlb}, which is not summable in $n$. Moreover, we have $\mu_n(w)=\ell(n) n^{-(p-1)}$ for $\ell$ some slowly-varying function (which may depend on $w$ up to constants only) by Lemma~\ref{lemma:mun}. Now, when $\gamma \leq 1$, we have the upper bound
		\be 
		\sum_{j=0}^\infty \frac{1}{f(j,k_n)}=\sum_{j=0}^\infty \frac{1}{g(k_n)s(j)+h(k_n)}\leq \frac{1}{g(k_n)}\sum_{j=0}^\infty \frac{1}{s(j)}.
		\ee 
		Since $s$ varies regularly with exponent $p>1$, it follows that the sum is finite. Since $g$ is regularly varying with exponent $1$ we can write $g(x)=x^{1+o(1)}$ sp that we obtain the upper bound $n^{-(1-\eps+o(1))/(\alpha-1)}$. As $(1-\eps)/(\alpha-1) > p-1$, it follows that~\eqref{eq:limsupbound} is satisfied.
		
		Otherwise, for $\gamma>1$, by similar computations as in~\eqref{eq:potterbound} and~\eqref{eq:potinsum} and using that $j^p+x\geq \eta (j+x^{1/p})^p$ for $j,x$, sufficiently large and $\eta$ small, we obtain, for $\eta>0$ sufficiently small,
		\begin{linenomath}
			\begin{align*}
				\sum_{j=0}^\infty \frac{1}{f(j,k_n)}=\frac{1}{g(k_n)}\sum_{j=0}^\infty \frac{1}{s(j)+h(k_n)/g(k_n)} & \leq \frac{1}{\eta g(k_n)}\sum_{j=\lfloor (h(k_n)/g(k_n))^{1/p}\rfloor}^\infty\!\!\!\!\!\!\!\!\!\!\!\!\!\!\! j^{-(p-\eta)} = \cO\Big(\frac{g(k_n)^{-(1+\eta)/p}}{h(k_n)^{(p-1-\eta)/p}}\Big).
			\end{align*}
		\end{linenomath}
		By Lemma~\ref{lemma:regvar}, the term $g(k_n)^{-(1-\eta)/p}/ h(k_n)^{(p-1-\eta)/p}$ is regularly varying, with exponent
		\be
		-\frac{1-\eps}{\alpha-1}\Big(\frac{1+\eta}{p}+\gamma\frac{p-1-\eta}{p}\Big)=-\frac{1-\eps}{\alpha-1}\Big(\gamma-\frac{\gamma-1}{p}-\frac{(\gamma-1)\eta}{p}\Big)< -(p-1),
		\ee
		for $\eta>0$ sufficiently small, where the last step uses the second inequality in~\eqref{eq:par}. Consequentially, it follows that $g(k_n)^{-(1-\eta)/p}/ h(k_n)^{(p-1-\eta)/p} = o(\mu^{w}_{n})$ and we obtain~\eqref{eq:limsupbound}.
	\end{proof}
	
	\begin{proof}[Proof of Lemmas~\ref{lemma:checklaplace} and~\ref{lemma:pathcond}, $s$ super-linear, additive weights]
		We first prove the claim in Lemma~\ref{lemma:checklaplace}. We fix $\eps\in(0,1/p)$ small such that $p(\alpha-1)(1-\eps)>1$. Then, with $k_n:=n^{(1-\eps)p}$ and using~\eqref{eq:weightasspowerlawub}, $\Prob{W \geq k_{n}} \leq \overline{\ell}(k_{n}) n^{-(\alpha-1)p(1-\eps)}$, and since $(\alpha-1)p(1-\eps) > 1$ we can deduce that the second term on the right-hand side of~\eqref{eq:splitbound} is summable. For the first term, we bound $f(j,k_n)$ from above by $\sup_{j\leq n^{1-\eps}}f(j,k_n)$ for all $j\leq n^{1-\eps}$. With a similar argument as in~\eqref{eq:fsupub}, we may assume that $f$ is monotone increasing in its first argument, since $s$ is regularly-varying with exponent $p>1$. This thus yields the upper bound
		\be
		\exp\bigg(-\sum_{j=0}^\infty\frac{c\mu_n^{-1}}{f(j,k_n)+c\mu_n^{-1}}\bigg)\leq \exp\Big(-\frac{c\mu_n^{-1}n^{1-\eps}}{f(n^{1-\eps},k_n)+c \mu_n^{-1}}\Big).
		\ee 
		As $(1-\eps)p>p-1$  using Lemma~\ref{lemma:mun}, we see that 
		$f(n^{1-\eps},k_n)+c\mu_n^{-1}=s(n^{1-\eps})+h(n^{(1-\eps)p})+ c\mu_n^{-1}$ is regularly varying with exponent $(1-\eps)p$. It follows that the fraction is a regularly-varying function in $n$, with exponent $(p-1)\eps>0$. So, for some slowly-varying function $L$, we can write the exponential term as $\exp(-L(n)n^{(p-1)\eps})$. We thus arrive at the desired conclusion that the exponential term is summable in $n$. This shows that the left-hand side of~\eqref{eq:splitbound} is summable and hence that~\eqref{eq:star-explosive-rif} is satisfied.
		
		We then prove the claim in Lemma~\ref{lemma:pathcond}. We fix $\eps>0$ sufficiently small such that $p(\alpha-1)<1-\eps$ and let $k_n:=n^{(1-\eps)/(\alpha-1)}$. We then have $\Prob{W \geq k_{n}} \geq \underline{\ell}(k_{n}) n^{-(1- \eps)}$ by~\eqref{eq:weightasspowerlawlb}, which is not summable in $n$. Then, for $w$ fixed, $\mu_n^w=\ell(n) n^{-(p-1)}$, for some slowly-varying function $\ell$ (which may depend on $w$ up to constants only), by Lemma~\ref{lemma:mun}. By calculations similar to those in~\eqref{eq:potterbound} and~\eqref{eq:potinsum}, we have for any $\eta>0$ and some large constant $C>0$, for all $n$ sufficiently large, 
		\be \label{eq:addcomp}
		\sum_{j=0}^\infty \frac{1}{f(j,k_n)}=\sum_{j=0}^\infty \frac{1}{s(j)+h(k_n)}\leq C\!\!\!\!\!\!\sum_{j=\lfloor h(k_n)^{1/p}\rfloor}^\infty \!\!\!\!\!\!j^{-(p-\eta)}=\cO(h(k_n)^{-(p-1-\eta)/p}).
		\ee 
		Since $h$ varies regularly with exponent $1$, the last term is $\cO(n^{-(p-1-\eta)/ (\alpha - 1)p}) = \cO(n^{-(p-1-\eta)/ (1-\eps)})$ since $p(\alpha - 1) < 1 -\eps$. As we can choose $\eta$ arbitrarily small, by the choice of $k_n$, it follows that~\eqref{eq:limsupbound} holds.
	\end{proof}
	
	\begin{proof}[Proof of Lemmas~\ref{lemma:checklaplace} and~\ref{lemma:pathcond}, $s$ barely super-linear~\hyperlink{log-stretched}{log-stretched case}, mixed weights] 
		We first prove the claim in Lemma~\ref{lemma:checklaplace}. We assume that $\beta\nu>1$ and recall that 
		\be 
		f(i,w)=g(w)(i+1)\e^{(\log(i+1))^\beta}+h(w), 
		\ee 
		where $g$ and $h$ are regularly-varying functions with exponents $1$ and $\gamma\geq 0$, respectively. Fix $\eps>0$ small so that $\beta\nu>1+\eps$. We apply~\eqref{eq:splitbound} with 
		\be \label{eq:knmun1}
		k_n:=\exp((\log n)^{(1+\eps)/\nu}), \quad \text{and} \quad \mu_n=\frac{1+o(1)}{g(0)\beta}(\log n)^{1-\beta}\exp(-(\log n)^\beta),
		\ee 
		where the latter follows from Lemma~\ref{lemma:mun}. Now, using~\eqref{eq:weightasslogstrechtedub} we obtain 
		\be  \label{eq:log-stretched-summ}
		\P{W\geq k_n}\leq \exp\big(-\overline c(\log k_n)^\nu\big)=\exp\big(-\overline c(\log n)^{1+\eps}\big) \leq n^{-(1+ \eps)}, 
		\ee 
		for $n$ sufficiently large, which implies that the second term on the right-hand side of~\eqref{eq:splitbound} is summable. For the first term, we define 
		\be \label{eq:inf1mix}
		I_n:=\exp\big((\log n)^\beta\big).
		\ee 
		Since $g$ is regularly varying with exponent one, we can write $g(x)=\ell(x)x=x^{1+o(1)}$ for some slowly-varying function $\ell$. As a result, using~\eqref{eq:knmun1},
		\be \label{eq:mu-n-over-g-k-n}
		\mu_n^{-1}/g(k_n) =(g(w)\beta+o(1))(\log n)^{\beta-1}\exp\big((\log n)^\beta - (\log n)^{(1+\eps)/\nu}(1+o(1))\big).
		\ee
		We thus obtain that for all $j\geq I_n$, when $n$ is large, 
		\be 
		(j+1)\exp((\log(j+1))^\beta)\geq I_n\exp((\log I_n)^\beta)=\exp((\log n)^\beta+(\log n)^{\beta^2})\geq \mu_n^{-1}/g(k_n).
		\ee  
		Moreover, since $\beta\nu>1+\eps$, it holds that $h(k_n)=o(\mu_n^{-1})$, irrespective of the regularly-varying exponent $\gamma\geq 0$ of $h$ (since we  can again write $h(x)=x^{\gamma+o(1)}$). As a result, we obtain the lower bound
		\be \ba 
		\sum_{j=0}^\infty \frac{1}{g(k_n)(j+1)\exp((\log( j+1))^\beta)+h(k_n)+\mu_n^{-1}}&\geq  \frac{1}{g(k_n)}\sum_{j=I_n}^\infty\frac{1}{3(j+1)\exp((\log(j+1))^\beta)}\\ 
		&=\frac{\mu_{I_n}(1+o(1))}{3g(k_n)/g(0)} \geq \frac{g(0)\mu_{I_n}}{4g(k_n)},
		\ea \ee 
		for $n$ sufficiently large, where we recall the definition of $\mu_n$ from~\eqref{eq:mu-def} (with $w^*=0$) and use (the proof of) Lemma~\ref{lemma:mun} in the final step. Substituting this bound into the sum in the exponent of the first term on the right-hand side of~\eqref{eq:splitbound}, we obtain the upper bound $\exp(-\frac14 cg(0)\mu_{I_n}\mu_n^{-1}/g(k_n))$. By the choice of $I_n$ in~\eqref{eq:inf1mix} and $\mu_n$ as in~\eqref{eq:knmun1},
		\be
		\mu_{I_n}=\frac{1+o(1)}{g(0)\beta }(\log I_n)^{1-\beta}\exp(-(\log I_n)^\beta)=\exp(-(\log n)^{\beta^2}(1-o(1))).
		\ee 
		By~\eqref{eq:mu-n-over-g-k-n}, we deduce that that the dominant term in $\exp(-\frac{c}{4C}\mu_{I_n}\mu_n^{-1}/g(k_n))$ is $\exp(- \exp(C' (\log{n})^{\beta}))$ for some $C' > 0$, since $\beta\nu>1+\eps$. By a similar argument as in~\eqref{eq:log-stretched-summ}, this upper bound is summable in $n$, as desired.
		
		We then prove the claim in Lemma~\ref{lemma:pathcond}. We fix $\eps>0$ sufficiently small such that $\beta\nu<1-\eps$ and set $k_n:=\exp((\log n)^{(1-\eps)/\nu})$. It follows from~\eqref{eq:weightasslogstrechtedlb} that, for $n$ sufficiently large
		\[
		\P{W\geq k_n} \geq \e^{-\underline{c} (\log{n})^{1- \eps}} \geq n^{-1},
		\] 
		which is not summable. We then have upper bound
		\begin{linenomath}
			\begin{align} \label{eq:similar-comp}
				\sum_{i=0}^\infty \frac{1}{f(i,k_n)}& =\frac{1}{g(k_n)}\sum_{i=0}^\infty \frac{1}{(i+1)\exp((\log (i+1))^\beta)+h(k_n)/g(k_n)} \\ & \leq \frac{1}{g(k_n)} \sum_{i=0}^{\infty} \frac{1}{(i+1)\exp((\log (i+1))^\beta)} \leq\frac{C_\beta}{g(k_n)}, 
			\end{align} 
		\end{linenomath}
		for some constant $C_\beta>0$, which follows since for any $\eps>0$ and $i$ sufficiently large, 
		\be 
		\frac{1}{(i+1)\exp((\log (i+1))^\beta)}\leq \frac{1}{(i+1)(\log (i+1))^{1+\eps}},
		\ee 
		so that the sum is indeed finite. Hence, since $g$ is regularly varying with exponent one and we can thus write $g(x)=x^{1+o(1)}$, we obtain 
		\be 
		\sum_{i=0}^\infty \frac{1}{f(i,k_n)}\leq C_\beta \exp(-(1+o(1))(\log n)^{(1-\eps)/\nu}).
		\ee 
		Since $\beta\nu<1-\eps$, the $-(\log n)^{(1-\eps)/\nu}$ term in the exponential function dominates the $(\log n)^\beta $ term in the exponential function in $(\mu^{w}_{n})^{-1}$; it thus follows that~\eqref{eq:limsupbound} is satisfied. 
	\end{proof}

	\subsection{Sub-tree counts in specific models of recursive trees with fitness}\label{sec:sub-treeproof}
	
	The aim of this subsection is to turn the conditions of Theorem~\ref{thm:sub-treecount} into more concrete conditions, which we then proceed to check for the different examples discussed in Section~\ref{sec:examples} to prove Theorem~\ref{thrm:sub-treecount}. We start by bounding the summands of~\eqref{eq:tree-summ-cond} from below and above. 
	
	\paragraph{Lower bound.} For a lower bound, we use  any one fixed ordering $O\in\cO(T)$, drop the indicators from the denominator, and introduce the indicator of the event $\{\sup_{i\leq k,j\leq k-1}f(i,W_j)<\mu_n^{-1}\}$. By bounding all terms $f(\outdeg{v_i,O_{|_j}},W_{v_i})$ (except for $f(\ell,W_{v_j}))$ in the denominator from above by $\mu_n^{-1}$, we obtain the lower bound  
	\be \ba \label{eq:first-lower-tree-finite-moment}
	\mathbb E\bigg[{}&\ind{\sup_{i\leq k,j\leq k-1}f(i,W_{v_j})<\mu_n^{-1}}\prod_{j=0}^{k-1}\prod_{\ell=0}^{\outdeg{v_j,T}-1}\frac{f(\ell,W_{v_j})}{\sum_{i=0}^j f(\outdeg{v_i,O_{|_j}},W_{v_i})+\mu_n^{-1}}\bigg]\\ 
	&\geq \mathbb E\bigg[\ind{\sup_{i\leq k,j\leq k-1}f(i,W_{v_j})<\mu_n^{-1}}\prod_{j=0}^{k-1}\prod_{\ell=0}^{\outdeg{v_j,T}-1}\frac{f(\ell,W_{v_j})}{f(\ell,W_{v_j})+k\mu_n^{-1}}\bigg],
	\ea \ee 
	where we use that the sum in the denominator has at most $k-1$ terms besides $f(\ell,W_{v_j})$. By the independence of the vertex-weights, setting 
	\be \label{eq:+-}
	f^+(k,W):=\sup_{i\leq k}f(i,W), \quad\text{and}\quad f^-(k,W):=\inf_{i\leq k}f(i,W),
	\ee 
	and using that the mapping $x\mapsto x+y$ is increasing in $x$, we obtain the lower bound
	\be \label{eq:Tproblb}
	\prod_{j=0}^{k-1}\mathbb E\bigg[\ind{f^+(k,W)<\mu_n^{-1}}\Big(\frac{f^-(k,W)}{f^-(k,W)+k\mu_n^{-1}}\Big)^{\outdeg{v_j,T}}\bigg],
	\ee 
	where we use that $\outdeg{v_j,T}\leq k$ for any vertex $v_j$ and any tree $T$ of size $k+1$.
	
	\paragraph{Upper bound.} For an upper bound we use a similar approach. First, irrespective of the ordering $O\in\cO(T)$, we can drop all terms from the denominator except $f(\ell,W_{v_j})$ and $\mu_n^{-1}$ and use independence of the vertex weights to obtain the upper bound 
	\be\ba \label{eq:first-upper-tree-finite-moment}
	\mathbb E\Bigg[{}&\prod_{j=0}^{k-1}\frac{\prod_{\ell=0}^{\outdeg{v_j,T}-1}f(\ell,W_{v_j})}{\sum_{i=0}^j f(\outdeg{v_i,O_{|_j}},W_{v_i})\ind{ \outdeg{v_i, O_{|_j}} < \outdeg{v_i, T}}+\mu_n^{-1}}\Bigg]\\ 
	&\leq \prod_{j=0}^{k-1} \E{\prod_{\ell=0}^{\outdeg{v_j,T}-1} \frac{f(\ell, W_{v_j})}{f(\ell, W_{v_j}) + \mu_{n}^{-1}}} \leq 
	\prod_{j=0}^{k-1}\E{\Big(\frac{f^+(k,W)}{f^+(k,W)+\mu_n^{-1}}\Big)^{\outdeg{v_j,T}}},
	\ea \ee
	where now $f^+(k,w):=\sup_{i\leq k}f(i,w)$ and we again use that the out-degrees are at most $k$. As the terms in the expected value are at most one, we condition on the size of $f^+(k,W)$ to obtain the upper bound
	\be 
	\prod_{j=0}^{k-1}\bigg(\E{\ind{f^+(k,W)<\mu_n^{-1}}\Big(\frac{f^+(k,W)}{f^+(k,W)+\mu_n^{-1}}\Big)^{\outdeg{v_j,T}}}+\P{f^+(k,W)\geq \mu_n^{-1}}\bigg).
	\ee 
	As a  result, since $|\cO(T)|< (k+1)!$ for any tree of size $k+1$, we can use the above in~\eqref{eq:tree-summ-cond} to obtain that, for some constant $D = D(k) >0$, the inner sum in~\eqref{eq:tree-summ-cond} is at most 
	\be \label{eq:Tprobub}
	D\prod_{j=0}^{k-1}\bigg(\E{\ind{f^+(k,W)<\mu_n^{-1}}\Big(\frac{f^+(k,W)}{f^+(k,W)+\mu_n^{-1}}\Big)^{\outdeg{v_j,T}}}+\P{f^+(k,W)\geq \mu_n^{-1}}\bigg).
	\ee 
	We now use the lower and upper bound to determine, based on certain conditions on the vertex-weights, whether the double sum in Theorem~\ref{thm:sub-treecount} is finite or infinite. In particular, we are now in a position to prove Corollary~\ref{cor:sumnufinmean}. 
	
	\subsubsection{Proof of Corollary~\ref{cor:sumnufinmean}} \label{sec:cor-sumnifinmean-proof}
	
	\begin{proof}[Proof of Corollary~\ref{cor:sumnufinmean}]
		We start by using the lower bound in~\eqref{eq:Tproblb}. For each expected value in the product, we have, since $f^-(k,W)\leq f^+(k,W)$, the lower bound
		\be 
		\E{\ind{f^+(i,W)<\mu_n^{-1}}\Big(\frac{f^-(k,W)}{(k+1)\mu_n^{-1}}\Big)^{\outdeg{v_j,T}}}\geq (k+1)^{-k}\mu_n^{\outdeg{v_j,T}}\E{\ind{f^+(i,W)<\mu_n^{-1}}f^-(k,W)^{\outdeg{v_j,T}}}.
		\ee 
		As $\mu_n^{-1}$ diverges with $n$ and by the moment condition, by monotone convergence we find that the expected value on the right-hand side equals $\E{f^-(k,W)^{\outdeg{v_j,T}}}(1-o(1))$, which is finite. As such, substituting this in~\eqref{eq:Tproblb}, we arrive at the lower bound
		\be 
		C\prod_{j=0}^{k-1} \mu_n^{\outdeg{v_j,T}}=C\mu_n^k,
		\ee 
		since the sum of the out-degrees equals $|T|-1=k$, and where $C>0$ is a constant. 
		
		For an upper bound we use~\eqref{eq:first-upper-tree-finite-moment}: omitting the terms $f^{*}(k,W)$ from the denominator, yields the upper bound 
		\be 
		D\prod_{j=0}^{k-1}\Big(\mu_n^{\outdeg{v_j,T}}\E{f^+(k,W)^{\outdeg{v_j,T}}}\Big).
		\ee   
		By the moment condition we again have that, irrespective of the choice of $k\in\N$ and the tree $T$, that expected values are finite. As such, we obtain an upper bound $D'\mu_n^k$ for some constant $D'$. Applying Theorem~\ref{thm:sub-treecount} then yields the desired result.
	\end{proof}
	
	\subsubsection{Proof of Theorem~\ref{thrm:sub-treecount}}
	\label{sec:proof-of-sub-treecount-superlinear}
	
	Recall the quantities $G_1$ and $G_2$ from~\eqref{eq:nus}. We first have the following result, which deals with the case that $f(k,W)$ is a regularly varying random variable.
	
	\begin{prop}\label{prop:sumnuinfmean}
		Assume that $f(k,W)$ satisfies~\eqref{eq:weightasspowerlawub} and~\eqref{eq:weightasspowerlawlb} for some slowly-varying functions $\overline \ell_k,\underline \ell_k$, respectively, and an exponent $z>0$, for any $k\in\N_0$ (where $z$ is independent of $k$), and assume that for each $k\in\N_0$ there exists $i_k\in\{0,\ldots ,k\}$ such that $\inf_{i\leq k}f(i,w)=f(i_k,w)$ for all $w\in S$. If there exists $\eta_0>0$ such that   
		\be \label{eq:tree-finite-often-summ-cond}
		\sum_{n=1}^\infty \mu_n^{k-( G_1(T,z) -zG_2(T,z))-\eta} <\infty\qquad \text{for all }\eta\in(0,\eta_0), 
		\ee 
		then the tree $\cT_\infty$ contains $T$ as a sub-tree finitely often, almost surely. If there exists $\eta_0>0$ such that 
		\be \label{eq:tree-infinite-often-div-cond}
		\sum_{n=1}^\infty \mu_n^{k-(G_1(T,z)-z G_2(T,z))+\eta}=\infty, \qquad \text{for all }\eta\in(0,\eta_0),
		\ee 
		then the tree $\cT_\infty$ contains $T$ infinitely often, almost surely.
	\end{prop}
	
	\noindent We finally state the following technical lemma, whose proof we defer to Section~\ref{sec:regapp} in the Appendix.
	
	\begin{lemma}\label{lemma:regexp}
		Let $r$ be a  regularly-varying function with exponent $\rho>0$ and let $W$ be a  random variable with a regularly-varying tail distribution with exponent $-(\zeta-1)<0$. Then, $r(W)$ is has a regularly-varying distribution with exponent $-(\zeta-1)/\rho$. 
	\end{lemma}
	
	\noindent With this proposition and lemma at hand, we can prove the results in Theorem~\ref{thrm:sub-treecount}.
	
	\begin{proof}[Proof of Theorem~\ref{thrm:sub-treecount}]
		We fix $k\in\N$ and let $T$ be a sibling-closed tree of size $k+1$. In all of the cases we discuss, we use the approximations of $\mu_n$ for the two main choices of the degree function $s$ (\hyperlink{superlinear}{super-linear} and \hyperlink{barsuperlinear}{barely super-linear}), as in Lemma~\ref{lemma:mun}, and distinguish the proof between these two main cases.
		
		\emph{Super-linear case.} For the super-linear case with additive weights, we have that $f(k,x)$ is regularly varying in $x$ with exponent $1$; for the super-linear case with mixed weights, $f(k,x)$ is regularly varying in $x$ with exponent $\gamma\vee 1$, for any $k\in\N_0$. Applying (the proof of) Lemma~\ref{lemma:regexp} and assuming that the vertex-weight distribution satisfies~\eqref{eq:weightasspowerlawub} and~\eqref{eq:weightasspowerlawlb} with exponent $\alpha-1$, we obtain that $f(k,W)$ satisfies these tail distribution inequalities with exponent $z=\alpha-1$ in the additive weight case and $z=(\alpha-1)/(\gamma\vee 1)$ in the mixed weight case for any $k\in\N_0$ (where the slowly-varying functions may depend on $k$). In both cases, for any $k\in\N_0$ and $w\geq 0$,
		\be 
		\inf_{i\leq k}f(i,w)=g(w)(\inf_{i\leq k}s(i))+h(w)=g(w)s(i_k)+h(w), 
		\ee 
		where $i_k=\argmin_{i\leq k}s(i)$. As $i_k$ does not depend on $w\geq 0$, the condition on $\inf_{i\leq k}f(i,w)$ in Proposition~\ref{prop:sumnuinfmean} is satisfied. Moreover, in both cases, $\mu_n$ is regularly varying with exponent $-(p-1)$. Hence, the inequalities in~\eqref{eq:Tsumcond} and~\eqref{eq:Tnonsumcond} follow by applying~\eqref{eq:tree-finite-often-summ-cond} and~\eqref{eq:tree-infinite-often-div-cond}, respectively.
		
		\emph{Barely super-linear case.} The proof follows directly from the fact that $\mu_n^a$ is not summable for any $a\geq 0$, as follows from Lemma~\ref{lemma:mun}, and that 
		\be 
		k-(G_1(T,z)-zG_2(T,z))=\sum_{v\in T}\outdeg{v,T}-\sum_{v\in T} (\outdeg{v,T}-z )\ind{\outdeg{v,T} >z}\geq 0,
		\ee 
		which concludes the proof.
	\end{proof}
	
	\noindent It remains to prove Proposition~\ref{prop:sumnuinfmean}.
	
	\begin{proof}[Proof of Proposition~\ref{prop:sumnuinfmean}]
		We first recall that we assume that the random variables $(f(k,W))_{k\in\N_0}$ satisfy~\eqref{eq:weightasspowerlawub} and~\eqref{eq:weightasspowerlawlb} for some slowly-varying functions $\overline \ell_k,\underline \ell_k$, respectively, and some exponent $z>0$ independent of $k$. By the definition of $f^+$ and $f^-$ in~\eqref{eq:+-} and the assumption that for each $k\in\N_0$ there exists $i_k\in\{0,\ldots, k\}$ such that $f^-(k,w):=\inf_{i\leq k}f(i,w)=f(i_k,w)$ for all $w\in S$, we have 
		\be \label{eq:+-bound}
		\P{f(i_k,W)\geq x}=\P{f^-(k,W)\geq x}\leq \P{f^+(k,W)\geq x}\leq \sum_{i=0}^k \P{f(i,W)\geq x}.
		\ee 
		By the assumptions on the tail distribution of $f(i,W)$, and by Lemma~\ref{lemma:regvar}, we thus find that we obtain an upper and lower bound for the tail probabilities of $f^-(k,W)$ and $f^+(k,W)$ that are regularly-varying with exponent $-z$.  We now apply~\cite[Proposition $1.3.2$]{Mik99} to $f^+(k,W)$, using the regularly-varying upper bound on its tail distribution (with $X=f^+(k,W)$, $\beta=\outdeg{v,T}$, and regularly-varying exponent $-z$) to determine that, for some slowly-varying function $L_1$, 
		\be \label{eq:mikres}
		\mathbb E\Big[\ind{f^+(k,W)<\mu_n^{-1}} f^+(k,W)^{\outdeg{v,T}}\Big]\leq   \begin{cases}
			\E{f^+(k,W)^{\outdeg{v,T}}}(1-o(1)) &\mbox{if } \outdeg{v,T}<z,\\
			L_1(\mu_n^{-1}) &\mbox{if }\outdeg{v,T}=z,\\
			\frac{z+o(1)}{\outdeg{v,T}-z} L_1(\mu_n^{-1})\mu_n^{z-\outdeg{v,T}} &\mbox{if } \outdeg{v,T}>z.
		\end{cases}
		\ee 
		We intend to use this in each of the product terms in the upper bound in~\eqref{eq:Tprobub} by omitting the term $f^+(k,W)$ in the denominator (which yields a further upper bound) and taking out a factor $\mu_n^{\outdeg{v_j,T}}$ for each $j\leq k-1$. We thus have a sum of two quantities in each product term, namely 
		\be \label{eq:suminprod}
		\mu_n^{\outdeg{v_j,T}}\mathbb E\Big[\ind{f^+(k,W)<\mu_n^{-1}} f^+(k,W)^{\outdeg{v,T}}\Big]+\P{f^+(k,W)\geq \mu_n^{-1}}. 
		\ee 
		When $\outdeg{v_j,T}<z$, the case distinction in~\eqref{eq:mikres} tells us that the first term is regularly varying in $\mu_n$ with exponent $\outdeg{v_j,T}$. Since, by~\eqref{eq:+-bound} $f^+(k,W)$ satisfies~\eqref{eq:weightasspowerlawub} (for some slowly-varying function $\overline \ell_k$, constant $ \overline x_k>0$ and exponent $z>0$), the second term is regularly varying in $\mu_n$ with exponent $z>\outdeg{v_j,T}$. Since $\mu_n$ tends to zero, it follows that, in this case, the first term in~\eqref{eq:suminprod} is dominant.
		
		When $\outdeg{v_j,T}=z$, the case distinction tells us that the first term is regularly varying in $\mu_n$ with exponent $\outdeg{v_j,T}=z$, and by again applying~\eqref{eq:weightasspowerlawub}, so is $\P{f^+(k,W)\geq \mu_n^{-1}}$. Finally, when $\outdeg{v_j,T}>z$, it is readily verified that both terms are equal, up to a multiplicative constant. Hence, we can bound the entire term in~\eqref{eq:suminprod} from above by $L_2(\mu_n^{-1})\mu_n^{\min\{z,\outdeg{v_j,T}\}}$ for some slowly-varying function $L_2$. Using this in each of the product terms in~\eqref{eq:Tprobub}, we obtain the upper bound
		\be \ba 
		D\prod_{j=0}^{k-1}{}&\bigg(\mathbb E\bigg[\ind{f^+(k,W)<\mu_n^{-1}}\Big(\frac{f^+(k,W)}{f^+(k,W)+\mu_n^{-1}}\Big)^{\outdeg{v_j,T}}\bigg]+\P{f^+(k,W)\geq \mu_n^{-1}}\bigg)\\
		&\leq D\mu_n^k \prod_{j=0}^{k-1}\big(L_2(\mu_n^{-1})\mu_n^{\min\{z-\outdeg{v_j,T},0\}}\big),
		\ea \ee 
		where we have taken the terms $\mu_n^{\outdeg{v_j,T}}$ outside of the product to obtain a term $\mu_n^k$. We now recall that $G_2(T,z)$ counts the number of vertices $v_j$ in $T$ such that $\outdeg{v_j,T}>z$, and $G_1(T,z)$ counts the sum of all degrees of such vertices $v_j$ (see~\eqref{eq:nus}). As a result, for some slowly-varying function $L_3$, using that products of slowly-varying functions are still slowly varying by Lemma~\ref{lemma:regvar}, we obtain the upper bound
		\be 
		L_3(\mu_n^{-1}) \mu_n^{k-G_1(T,z)} \P{f^+(k,W)\geq \mu_n^{-1}}^{G_2(T,z)},  
		\ee 
		Now, we use that $f^+(k,W)$ satisfies~\eqref{eq:weightasspowerlawub}  and apply~\ref{lemma:regvar} to arrive at 
		\be \label{eq:finubTprob}
		L_4(\mu_n^{-1})\mu_n^{k-(G_1(T,z)-zG_2(T,z))}, 
		\ee 
		for some slowly-varying function $L_4$. Since we can bound $L_4(\mu_n^{-1})\leq \mu_n^{-\eta}$ for any $\eta>0$ and $n$ sufficiently large, the desired result in~\eqref{eq:tree-finite-often-summ-cond} is obtained by applying Theorem~\ref{thm:sub-treecount}.
		
		To prove~\eqref{eq:tree-infinite-often-div-cond}, we start from the lower bound in~\eqref{eq:Tproblb}. Noting that the integrand is bounded from above by one, we have  
		\be \ba \label{eq:lbpluserror}
		\mathbb E&{}\bigg[\ind{f^+(k,W)<\mu_n^{-1}}\prod_{j=0}^{k-1}\prod_{\ell=0}^{\outdeg{v_j,T}-1}\frac{f(\ell,W_{v_j})}{f(\ell,W_{v_j})+k\mu_n^{-1}}\bigg] \\  
		&\geq \prod_{j=0}^{k-1}\mathbb E\bigg[\ind{f^-(k,W)<\mu_n^{-1}}\Big(\frac{f^-(k,W)}{f^-(k,W)+k\mu_n^{-1}}\Big)^{\outdeg{v_j,T}}\bigg]-\P{f^-(k,W)<\mu_n^{-1},f^+(k,W)\geq \mu_n^{-1}}^k.
		\ea \ee 
		We first tend to the product on the right-hand side. By bounding the term $f^-(k,W)$ in the denominator from above by $\mu_n^{-1}$, we bound it from below by 
		\be 
		(k+1)^{-k}\mu_n^k \prod_{j=0}^k \E{\ind{f^-(k,W)<\mu_n^{-1}}f^-(k,W)^{\outdeg{v_j,T}}}
		\ee 
		We now use the regularly-varying lower bound on the tail distribution of $f^-(k,W)$, as derived from~\eqref{eq:+-bound}. Again applying~\cite[Proposition $1.3.2$]{Mik99}, we obtain a lower bound analogous to~\eqref{eq:mikres} (and by possibly changing the slowly-varying function). The steps in in~\eqref{eq:suminprod} through~\eqref{eq:finubTprob} are then identical, so that we obtain, for some slowly-varying function $L_5$, the lower bound
		\be 
		(k+1)^{-k}L_5(\mu_n^{-1})\mu_n^{k-(G_1(T,z)-zG_2(T,z))}.
		\ee 
		We now substitute this in~\eqref{eq:lbpluserror} with $C=(k+1)^{-k}$ and bound the probability from above by omitting the first event, to arrive at 
		\be\ba \label{eq:difflb}  
		\prod_{j=0}^{k-1}\mathbb E\bigg[\ind{f^+(k,W)<\mu_n^{-1}}\Big(\frac{f^-(k,W)}{f^-(k,W)+k\mu_n^{-1}}\Big)^{\outdeg{v_j,T}}\bigg]\geq{}& CL_5(\mu_n^{-1})\mu_n^{k-(G_1(T,z)-zG_2(T,z))}\\ 
		&-\P{f^+(k,W)\geq \mu_n^{-1}}^k.
		\ea \ee 
		As follows from~\eqref{eq:+-bound}, the probability $\P{f^+(k,W)\geq \mu_n^{-1}}^k$ is bounded from above by a  regularly-varying term (in $\mu_n$) with exponent $kz$. By Lemma~\ref{lemma:regvar}, the right-hand side is regularly-varying with exponent $\min\{k-(G_1(T,z)-zG_2(T,z)),kz\}$. In particular, it is positive for all $n$ large if the first argument attains the minimum. To determine that this is the case, we observe, with a similar argument as in from~\eqref{eq:Gspos}, that the first argument of the minimum is at most $k$. Hence, when $z\geq 1$, it is immediate. When $z<1$, we instead use that $G_1(T,z)=k$ (since now $G_1$ simply sums all degrees) and $G_2(T,z)=k+1-N$, where $N\geq 1$ is the number of leaves (i.e.\ nodes with out-degree zero) in $T$. We thus have
		\be 
		k(G_1(T,z)-zG_2(T,z))\leq kz \quad \Leftrightarrow\quad z(k+1-N)\leq kz, 
		\ee 
		which holds since $N\geq 1$. As a result, we can bound the right-hand side of~\eqref{eq:difflb} from below, for some slowly-varying function $L_6$, by 
		\be 
		L_6(\mu_n^{-1})\mu_n^{k-(G_1(T,z)-(\alpha-1)G_2(T,z)/z)}.
		\ee 
		We finally use Lemma~\ref{lemma:regvar} to bound $L_6(\mu_n^{-1})\geq \mu_n^\eta$ for any $\eta>0$ and $n$ sufficiently large and apply Theorem~\ref{thm:sub-treecount} to finally arrive at~\eqref{eq:tree-infinite-often-div-cond}, which concludes the proof.
	\end{proof}

	\section*{Acknowledgements}
	
	The research of TI has been funded by Deutsche Forschungsgemeinschaft (DFG) through DFG Project no. $443759178$. TI also acknowledges funding from the 2023 WIAS Young Scientist grant, as well as previous funding from the LMS Early Career Fellowship~ECF-1920-55, which helped support this research. 
	
	\noindent BL has received funding from the European Union’s Horizon 2022 research and innovation programme under the Marie Sk\l{}odowska-Curie grant agreement no.\ $101108569$ and has been supported by the grant GrHyDy ANR-20-CE40-0002. 
	
	\bibliographystyle{abbrv}
	\bibliography{refs}

\providecommand{\vander}{van der }
\begin{thebibliography}{10}

\bibitem{amini-et-al}
O.~Amini, L.~Devroye, S.~Griffiths, and N.~Olver.
\newblock On explosions in heavy-tailed branching random walks.
\newblock {\em Ann. Probab.}, 41(3B):1864--1899, 2013.

\bibitem{arthreya-gen-weight-function}
K.~B. Athreya.
\newblock Preferential attachment random graphs with general weight function.
\newblock {\em Internet Math.}, 4(4):401--418, 2007.

\bibitem{arthreya-karlin-split-times}
K.~B. Athreya and S.~Karlin.
\newblock Limit theorems for the split times of branching processes.
\newblock {\em J. Math. Mech.}, 17:257--277, 1967.

\bibitem{arthreya-karlin-embedding-68}
K.~B. Athreya and S.~Karlin.
\newblock Embedding of urn schemes into continuous time {M}arkov branching
  processes and related limit theorems.
\newblock {\em Ann. Math. Statist.}, 39:1801--1817, 1968.

\bibitem{arthreya-ney}
K.~B. Athreya and P.~E. Ney.
\newblock {\em Branching processes}, volume Band 196 of {\em Die Grundlehren
  der mathematischen Wissenschaften}.
\newblock Springer-Verlag, New York-Heidelberg, 1972.

\bibitem{ball-epidemics}
F.~Ball, M.~Gonz\'{a}lez, R.~Mart\'{\i}nez, and M.~Slavtchova-Bojkova.
\newblock Stochastic monotonicity and continuity properties of functions
  defined on {C}rump-{M}ode-{J}agers branching processes, with application to
  vaccination in epidemic modelling.
\newblock {\em Bernoulli}, 20(4):2076--2101, 2014.

\bibitem{ball-epidemic-2016}
F.~Ball, M.~Gonz\'{a}lez, R.~Mart\'{\i}nez, and M.~Slavtchova-Bojkova.
\newblock Total progeny of {C}rump-{M}ode-{J}agers branching processes: an
  application to vaccination in epidemic modelling.
\newblock In {\em Branching processes and their applications}, volume 219 of
  {\em Lect. Notes Stat.}, pages 257--267. Springer, [Cham], 2016.

\bibitem{banerjee-bhamidi}
S.~Banerjee and S.~Bhamidi.
\newblock Persistence of hubs in growing random networks.
\newblock {\em Probab. Theory Related Fields}, 180(3-4):891--953, 2021.

\bibitem{berger-noam-chayes-2014}
N.~Berger, C.~Borgs, J.~T. Chayes, and A.~Saberi.
\newblock Asymptotic behavior and distributional limits of preferential
  attachment graphs.
\newblock {\em Ann. Probab.}, 42(1):1--40, 2014.

\bibitem{berndtsson-jagers}
B.~Berndtsson and P.~Jagers.
\newblock Exponential growth of a branching process usually implies stable age
  distribution.
\newblock {\em J. Appl. Probab.}, 16(3):651--656, 1979.

\bibitem{bertoin-local-expl}
J.~Bertoin and R.~Stephenson.
\newblock Local explosion in self-similar growth-fragmentation processes.
\newblock {\em Electron. Commun. Probab.}, 21:Paper No. 66, 12, 2016.

\bibitem{bhamidi}
S.~Bhamidi.
\newblock Universal techniques to analyze preferential attachment trees: global
  and local analysis, 2007.
\newblock Preprint available at
  \url{https://scholar.google.com/citations?view_op=view_citation&hl=en&user=4JsyJK8AAAAJ&citation_for_view=4JsyJK8AAAAJ:IjCSPb-OGe4C}.

\bibitem{BinGolTeu89}
N.~H. Bingham, C.~M. Goldie, J.~L. Teugels, and J.~Teugels.
\newblock {\em Regular variation}.
\newblock Number~27. Cambridge university press, 1989.

\bibitem{bogachev}
V.~I. Bogachev.
\newblock {\em Measure theory. {V}ol. {I}, {II}}.
\newblock Springer-Verlag, Berlin, 2007.

\bibitem{borgs-chayes}
C.~Borgs, J.~Chayes, C.~Daskalakis, and S.~Roch.
\newblock First to market is not everything: an analysis of preferential
  attachment with fitness.
\newblock In {\em S{TOC}'07---{P}roceedings of the 39th {A}nnual {ACM}
  {S}ymposium on {T}heory of {C}omputing}, pages 135--144. ACM, New York, 2007.

\bibitem{choi-sethuraman-2013}
J.~Choi and S.~Sethuraman.
\newblock Large deviations for the degree structure in preferential attachment
  schemes.
\newblock {\em Ann. Appl. Probab.}, 23(2):722--763, 2013.

\bibitem{crump-mode}
K.~S. Crump and C.~J. Mode.
\newblock A general age-dependent branching process. {I}, {II}.
\newblock {\em J. Math. Anal. Appl.}, 24:494--508; ibid. { 25 (1968), 8--17},
  1968.

\bibitem{dereich-unfolding}
S.~Dereich.
\newblock Preferential attachment with fitness: unfolding the condensate.
\newblock {\em Electron. J. Probab.}, 21:Paper No. 3, 38, 2016.

\bibitem{dereich-mailler-morters}
S.~Dereich, C.~Mailler, and P.~M\"{o}rters.
\newblock Nonextensive condensation in reinforced branching processes.
\newblock {\em Ann. Appl. Probab.}, 27(4):2539--2568, 2017.

\bibitem{dereich-morters-persistence}
S.~Dereich and P.~M\"{o}rters.
\newblock Random networks with sublinear preferential attachment: degree
  evolutions.
\newblock {\em Electron. J. Probab.}, 14:no. 43, 1222--1267, 2009.

\bibitem{dereich-ortgiese}
S.~Dereich and M.~Ortgiese.
\newblock Robust analysis of preferential attachment models with fitness.
\newblock {\em Combin. Probab. Comput.}, 23(3):386--411, 2014.

\bibitem{pref-att-neighbours}
N.~Fountoulakis and T.~Iyer.
\newblock {Condensation phenomena in preferential attachment trees with
  neighbourhood influence}.
\newblock {\em Electronic Journal of Probability}, 27:1 -- 49, 2022.

\bibitem{galashin2014existence}
P.~Galashin.
\newblock Existence of a persistent hub in the convex preferential attachment
  model.
\newblock {\em arXiv preprint arXiv:1310.7513}, 2013.

\bibitem{garavaglia2023universality}
A.~Garavaglia, R.~S. Hazra, R.~van~der Hofstad, and R.~Ray.
\newblock Universality of the local limit of preferential attachment models.
\newblock {\em arXiv preprint arXiv:2212.05551}, 2022.

\bibitem{pagerank}
A.~Garavaglia, R.~van~der Hofstad, and N.~Litvak.
\newblock Local weak convergence for {P}age{R}ank.
\newblock {\em Ann. Appl. Probab.}, 30(1):40--79, 2020.

\bibitem{vdh-aging-mult-fitness-2017}
A.~Garavaglia, R.~van~der Hofstad, and G.~Woeginger.
\newblock The dynamics of power laws: Fitness and aging in preferential
  attachment trees.
\newblock {\em J. Stat. Phys.}, 168:1137--1179, 2017.

\bibitem{GottGross23}
T.~Gottfried and S.~Grosskinsky.
\newblock Asymptotics of generalized {P}\'olya urns with non-linear feedback.
\newblock {\em arXiv preprint arXiv:2303.01210}, 2023.

\bibitem{Grishechkin-92}
S.~Grishechkin.
\newblock On a relationship between processor-sharing queues and
  {C}rump-{M}ode-{J}agers branching processes.
\newblock {\em Adv. in Appl. Probab.}, 24(3):653--698, 1992.

\bibitem{Hof16}
R.~{\vander H}ofstad.
\newblock {\em Random graphs and complex networks}, volume~43.
\newblock Cambridge university press, 2016.

\bibitem{holmgren-janson}
C.~Holmgren and S.~Janson.
\newblock Fringe trees, {C}rump-{M}ode-{J}agers branching processes and
  {$m$}-ary search trees.
\newblock {\em Probab. Surv.}, 14:53--154, 2017.

\bibitem{iksanov-kolesko-fluctuations}
A.~Iksanov, K.~Kolesko, and M.~Meiners.
\newblock Gaussian fluctuations and a law of the iterated logarithm for
  {N}erman's martingale in the supercritical general branching process.
\newblock {\em Electron. J. Probab.}, 26:Paper No. 160, 22, 2021.

\bibitem{iksanov-marynych-fluctuations}
A.~Iksanov, A.~Marynych, and B.~Rashytov.
\newblock Stable fluctuations of iterated perturbed random walks in
  intermediate generations of a general branching process tree.
\newblock {\em Lith. Math. J.}, 62(4):447--466, 2022.

\bibitem{limiting-structure}
T.~Iyer.
\newblock On explosion in {C}rump-{M}ode-{J}agers branching processes, with
  applications to the limiting structure of recursive trees with fitness.
\newblock \emph{In preparation.}

\bibitem{rec-trees-fit}
T.~Iyer.
\newblock Degree distributions in recursive trees with fitnesses.
\newblock {\em Adv. in Appl. Probab.}, 55(2):407--443, 2023.

\bibitem{jagers-origin}
P.~Jagers.
\newblock A general stochastic model for population development.
\newblock {\em Skand. Aktuarietidskr.}, pages 84--103, 1969.

\bibitem{jagers-book}
P.~Jagers.
\newblock {\em Branching processes with biological applications}.
\newblock Wiley Series in Probability and Mathematical Statistics---Applied
  Probability and Statistics. Wiley-Interscience [John Wiley \& Sons],
  London-New York-Sydney, 1975.

\bibitem{jagers-89}
P.~Jagers.
\newblock General branching processes as {M}arkov fields.
\newblock {\em Stochastic Process. Appl.}, 32(2):183--212, 1989.

\bibitem{jagers-nerman-96}
P.~Jagers and O.~Nerman.
\newblock The asymptotic composition of supercritical multi-type branching
  populations.
\newblock In {\em S\'{e}minaire de {P}robabilit\'{e}s, {XXX}}, volume 1626 of
  {\em Lecture Notes in Math.}, pages 40--54. Springer, Berlin, 1996.

\bibitem{jagers-nerman89}
P.~Jagers, O.~Nerman, and Z.~Ta\"{\i}b.
\newblock When did {J}oe's great{$\ldots$}grandfather live? {O}r: {O}n the time
  scale of evolution.
\newblock In {\em Selected {P}roceedings of the {S}heffield {S}ymposium on
  {A}pplied {P}robability ({S}heffield, 1989)}, volume~18 of {\em IMS Lecture
  Notes Monogr. Ser.}, pages 118--126. Inst. Math. Statist., Hayward, CA, 1991.

\bibitem{svante-condensation}
S.~Janson.
\newblock Simply generated trees, conditioned {G}alton-{W}atson trees, random
  allocations and condensation.
\newblock {\em Probab. Surv.}, 9:103--252, 2012.

\bibitem{janson-cmj-fluctuations}
S.~Janson.
\newblock Asymptotics of fluctuations in {C}rump-{M}ode-{J}agers processes: the
  lattice case.
\newblock {\em Adv. in Appl. Probab.}, 50(A):141--171, 2018.

\bibitem{jonathan-1}
J.~Jordan.
\newblock Preferential attachment graphs with co-existing types of different
  fitnesses.
\newblock {\em J. Appl. Probab.}, 55(4):1211--1227, 2018.

\bibitem{kingman1975}
J.~F.~C. Kingman.
\newblock The first birth problem for an age-dependent branching process.
\newblock {\em Ann. Probab.}, 3(5):790--801, 10 1975.

\bibitem{clt-cmj-kolesko}
K.~Kolesko and E.~Sava-Huss.
\newblock Limit theorems for discrete multitype branching processes counted
  with a characteristic.
\newblock {\em Stochastic Process. Appl.}, 162:49--75, 2023.

\bibitem{Komjathy2016ExplosiveCB}
J.~Komj{\'a}thy.
\newblock Explosive {Crump-Mode-Jagers} branching processes.
\newblock {\em arXiv preprint arXiv:1602.01657}, 2016.

\bibitem{KrapKri08}
P.~Krapivsky and D.~Krioukov.
\newblock Scale-free networks as preasymptotic regimes of superlinear
  preferential attachment.
\newblock {\em Physical Review E}, 78(2):026114, 2008.

\bibitem{KunBlatMos13}
J.~Kunegis, M.~Blattner, and C.~Moser.
\newblock Preferential attachment in online networks: Measurement and
  explanations.
\newblock In {\em Proceedings of the 5th annual ACM web science conference},
  pages 205--214, 2013.

\bibitem{last-penrose}
G.~Last and M.~Penrose.
\newblock {\em Lectures on the {P}oisson process}, volume~7 of {\em Institute
  of Mathematical Statistics Textbooks}.
\newblock Cambridge University Press, Cambridge, 2018.

\bibitem{Levesque-21}
J.~Levesque, D.~W. Maybury, and R.~H. A.~D. Shaw.
\newblock A model of {COVID}-19 propagation based on a gamma subordinated
  negative binomial branching process.
\newblock {\em J. Theoret. Biol.}, 512:Paper No. 110536, 10, 2021.

\bibitem{lo2021weak}
T.~Y.~Y. Lo.
\newblock Weak local limit of preferential attachment random trees with
  additive fitness, 2021.

\bibitem{Lod21}
B.~Lodewijks.
\newblock Location of high-degree vertices in weighted recursive graphs with
  bounded random weights.
\newblock {\em To appear in the Journal of Applied Probability}, 2021.

\bibitem{LodOrt20}
B.~Lodewijks and M.~Ortgiese.
\newblock The maximal degree in random recursive graphs with random weights.
\newblock {\em arXiv preprint arXiv:2007.05438}, 2020.

\bibitem{bas}
B.~Lodewijks and M.~Ortgiese.
\newblock A phase transition for preferential attachment models with additive
  fitness.
\newblock {\em Electron. J. Probab.}, 25:Paper No. 146, 54, 2020.

\bibitem{Mik99}
T.~Mikosch.
\newblock {\em Regular variation, subexponentiality and their applications in
  probability theory}, volume~99.
\newblock Eindhoven University of Technology Eindhoven, The Netherlands, 1999.

\bibitem{exp-scaling}
M.~Mitzenmacher, R.~Oliveira, and J.~Spencer.
\newblock A scaling result for explosive processes.
\newblock {\em Electron. J. Combin.}, 11(1):Research Paper 31, 14, 2004.

\bibitem{nerman_81}
O.~Nerman.
\newblock On the convergence of supercritical general ({C}-{M}-{J}) branching
  processes.
\newblock {\em Z. Wahrsch. Verw. Gebiete}, 57(3):365--395, 1981.

\bibitem{nerman-jagers-84}
O.~Nerman and P.~Jagers.
\newblock The stable double infinite pedigree process of supercritical
  branching populations.
\newblock {\em Z. Wahrsch. Verw. Gebiete}, 65(3):445--460, 1984.

\bibitem{Oliveira-spencer}
R.~Oliveira and J.~Spencer.
\newblock Connectivity transitions in networks with super-linear preferential
  attachment.
\newblock {\em Internet Math.}, 2(2):121--163, 2005.

\bibitem{olofsson-x-log-x}
P.~Olofsson.
\newblock The {$x\log x$} condition for general branching processes.
\newblock {\em J. Appl. Probab.}, 35(3):537--544, 1998.

\bibitem{PhamSherShi16}
T.~Pham, P.~Sheridan, and H.~Shimodaira.
\newblock Joint estimation of preferential attachment and node fitness in
  growing complex networks.
\newblock {\em Scientific reports}, 6(1):32558, 2016.

\bibitem{pittel}
B.~Pittel.
\newblock Note on the heights of random recursive trees and random {$m$}-ary
  search trees.
\newblock {\em Random Structures Algorithms}, 5(2):337--347, 1994.

\bibitem{rudas}
A.~Rudas, B.~T\'{o}th, and B.~Valk\'{o}.
\newblock Random trees and general branching processes.
\newblock {\em Random Structures Algorithms}, 31(2):186--202, 2007.

\bibitem{Sen21}
D.~S{\'e}nizergues.
\newblock Geometry of weighted recursive and affine preferential attachment
  trees.
\newblock {\em Electronic Journal of Probability}, 26:1--56, 2021.

\bibitem{sethuraman-venkataramani-2019}
S.~Sethuraman and S.~C. Venkataramani.
\newblock On the growth of a superlinear preferential attachment scheme.
\newblock In {\em Probability and analysis in interacting physical systems},
  volume 283 of {\em Springer Proc. Math. Stat.}, pages 243--265. Springer,
  Cham, 2019.

\end{thebibliography}
	
	\appendix
	
	\section{Additional results for other fitness functions}\label{sec:appadd}
	
	In Theorem~\ref{thrm:cmjexamples} we discussed three particular choices of the fitness function: a super-linear degree function $s$ with either mixed or additive weight functions, and a log-stretched super-linear degree function $s$ with mixed weight functions. For these choices, we were able to prove a complete phase diagram for the emergence of a unique vertex with infinite degree or a unique infinite path, and for sub-tree counts in the former case. 
	
	In this section, we discuss three additional choices of the fitness function: barely super-linear log-stretched degree function $s$ with \emph{additive} weight functions, and \emph{poly-log} degree function $s$ with either \emph{mixed} or \emph{additive} weight functions. We present these cases here, since we are unable to prove a complete phase diagram; for each case there is a gap in the parameter choices where we cannot prove the emergence of a unique vertex with infinite degree nor a unique infinite path. 
	
	We summarise these choices in the following assumptions.
	
	\begin{assumption}\label{ass:fapp}
		Let $f(i,w)=g(w)s(i)+h(w), i\in\N_0,w\geq 0$ satisfy~\eqref{eq:minfass} and one of the following three cases.
		\begin{itemize}
			\item \textbf{\hyperlink{logstretched}{Barely super-linear log-stretched with additive weights}:} $s$ is as in~\eqref{eq:slogstretched}, $g\equiv 1$, and $h(x)=\ell(x)x$ for some slowly-varying function $\ell:[0,\infty) \to [0,\infty)$, such that either
			\begin{align} \label{eq:elllim} 
				\exists \underline a\in[0,1):       \lim_{x\to\infty}\frac{\log\log(1/\ell(x))}{\log\log x}=\underline a,
				\shortintertext{or} 
				\hspace{-1cm}\exists \overline a\in[0,1):\lim_{x\to\infty}\frac{\log\log(\ell(x))}{\log\log x}=\overline a. \label{eq:elllim2}
			\end{align}  
			\item \textbf{\hypertarget{poly-log}{Barely super-linear poly-log}:} For some $\sigma>1$, 
			\be 
			s(i)=(i+2)(\log(i+2))^\sigma, \qquad i\in\N_0, 
			\ee 
			and $g,h$ are as in the additive or mixed weight case of Assumption~\ref{ass:f}. 
		\end{itemize}
	\end{assumption}
	
	\noindent We then also introduce the following additional assumptions for the vertex-weight distribution. 
	
	\begin{assumption}\label{ass:weightsapp}
		The vertex-weights $(W_i)_{i\in\N}$ are i.i.d.\ and their tail distribution satisfies one of the following cases.
		\begin{itemize}
			\item \textbf{Power law with lower-order term.} Let  $\tau\in(0,1)$. We have the following two conditions. 
			\begin{enumerate}
				\item There exist $\overline c,\overline x>0$ and $\tau\in(0,1)$ such that 
				\be\label{eq:weightasspowerlawplusub}
				\P{W\geq x}\leq x^{-1}\e^{-\overline c(\log x)^\tau}, \qquad x\geq \overline x.
				\ee 
				\item There exist $\underline c,\underline x>0$ and $\tau'\in(0,1)$ such that
				\be \label{eq:weightasspowerlawpluslb}
				\P{W\geq x }\geq x^{-1}\e^{\underline c(\log x)^{\tau'}}, \qquad x\geq \underline x.
				\ee             
			\end{enumerate}
			\item \textbf{Stretched exponential.} Let $\kappa>0$. We have the following two conditions.
			\begin{enumerate}
				\item There exist $\overline c,\overline x>0$ such that
				\be \label{eq:weightassstrechtedub}
				\P{W\geq x}\leq \e^{-\overline cx^\kappa}, \qquad x\geq \overline x. 
				\ee 
				\item There exists $\underline c,\underline x>0$ such that 
				\be \label{eq:weightassstrechtedlb}
				\P{W\geq x}\geq \e^{-\underline cx^\kappa}, \qquad x\geq \underline x. 
				\ee 
			\end{enumerate}
		\end{itemize}
		
	\end{assumption}
	
	\noindent We then have the following theorem, which is an equivalent result to Theorem~\ref{thrm:cmjexamples}.
	
	\begin{thm}\label{thrm:examplesadd}
		Let $(\mathcal{T}_{i})_{i \in \mathbb{N}}$ be a $(W, f)$-recursive tree with fitness, where the fitness function $f$ and degree function $s$ satisfy one of the cases in Assumptions~\ref{ass:f} and~\ref{ass:deg}, respectively, and the vertex-weight distribution satisfies one of the cases in Assumption~\ref{ass:weights}. The tree $\mathcal T_\infty$ either contains a unique vertex with infinite degree and no infinite path almost surely, or contains a unique infinite path and no vertex with infinite degree almost surely, when the following conditions are met, based on the \emph{fitness function}, \emph{degree function}, and \emph{vertex-weight} assumptions:
		\begin{table}[H]
			\footnotesize
			\captionsetup{width=0.89\textwidth}
			\centering
			\begin{tabular}{|c|c|l|l|}
				\hline
				\textbf{Weight} & \textbf{Degree} & \textbf{Star} & \textbf{Path}\\
				\hline 
				\hyperlink{additive}{Additive} & \hyperlink{log-stretched}{Log-Stretched} & \eqref{eq:elllim},\eqref{eq:weightasspowerlawplusub} \& $(\tau\vee \beta)>(1-\beta)\vee \underline a$ & \eqref{eq:elllim2},\eqref{eq:weightasspowerlawpluslb} \& $\tau'> \beta\vee (1-\beta)$, $(\tau'\vee \beta) >\overline a$ \\
				\hline 
				\hyperlink{mixed}{Mixed} & \hyperlink{poly-log}{Poly-log} & \eqref{eq:weightassstrechtedub} \& $(\sigma-1)\kappa>1+ \kappa$ & \eqref{eq:weightassstrechtedlb} \& $(\sigma-1)\kappa < 1$ \\
				\hline
				\hyperlink{additive}{Additive} & \hyperlink{poly-log}{Poly-log} & \eqref{eq:weightasslogstrechtedub} \& $(\sigma-1)(1-1/\nu)>1$ & \eqref{eq:weightasspowerlawlb} \& $\alpha < 2$ \\
				\hline   
			\end{tabular}
			\caption{\footnotesize The first column represents the form of the fitness function, as in Assumption~\ref{ass:f}; the second represents the form of the degree function, as in Assumption~\ref{ass:fapp}. The third and fourth column, respectively, list the required assumptions on the vertex-weight distribution, as in Assumptions~\ref{ass:weights} and~\ref{ass:weightsapp}, together with the choices of the parameters that lead to either a unique node of infinite degree or a unique infinite path.}
		\end{table}
	\end{thm}
	
	\noindent The proof of Theorem~\ref{thrm:examplesadd} follows from two lemmas, equivalent to Lemmas~\ref{lemma:checklaplace} and~\ref{lemma:pathcond}, which we state now.

	\begin{lemma}\label{lemma:checklaplaceapp}
		Equation~\eqref{eq:star-explosive-rif} in Item~\ref{item:star-explosive-rif} of Theorem~\ref{thm:star-path-rif} is satisfied when the following conditions are met, based on the assumptions for the fitness type, degree function $s$, and vertex-weight distribution:
		\begin{table}[H]
			
			\centering
			\footnotesize
			\captionsetup{width=0.89\textwidth}
			
			\begin{tabular}{|c|c|l|}
				\hline
				\textbf{Weight} & \textbf{Degree} &  \textbf{Star}  \\  
				\hline
				\hyperlink{additive}{Additive} & \hyperlink{log-stretched}{Log-Stretched} & \eqref{eq:elllim},\eqref{eq:weightasspowerlawplusub} \& $(\tau\vee \beta)>(1-\beta)\vee \underline a$  \\
				\hline 
				\hyperlink{mixed}{Mixed} & \hyperlink{poly-log}{Poly-log} & \eqref{eq:weightassstrechtedub} \& $(\sigma-1)\kappa>1+ \kappa$ \\
				\hline
				\hyperlink{additive}{Additive} & \hyperlink{poly-log}{Poly-log} & \eqref{eq:weightasslogstrechtedub} \& $(\sigma-1)(1-1/\nu)>1$\\
				\hline    
			\end{tabular}
			\caption{\footnotesize The first column represents the form of the fitness function, as in Assumption~\ref{ass:f}; the second represents the form of the degree function, as in Assumption~\ref{ass:fapp}. The third column lists the required assumptions on the vertex-weight distribution, as in Assumption~\ref{ass:weightsapp}, together with the choices of the parameters that lead to a unique node of infinite degree.}
		\end{table}
	\end{lemma}
	
	\begin{lemma}\label{lemma:pathcondapp}
		The condition in~\eqref{eq:div-condition} satisfied when the following conditions are met, based on the assumptions for the fitness type, degree function $s$, and vertex-weight distribution: 
		\begin{table}[H]
			
			\centering
			\footnotesize
			\captionsetup{width=0.89\textwidth}
			
			\begin{tabular}{|c|c|l|}
				\hline
				\textbf{Weight} & \textbf{Degree} & \textbf{Path} \\  
				\hline
				\hyperlink{additive}{Additive} & \hyperlink{log-stretched}{Log-Stretched} & \eqref{eq:elllim2},\eqref{eq:weightasspowerlawpluslb} \& $\tau'> \beta\vee (1-\beta)$, $(\tau'\vee \beta )>\overline a$\\
				\hline 
				\hyperlink{mixed}{Mixed} & \hyperlink{poly-log}{Poly-log} & \eqref{eq:weightassstrechtedlb} \& $(\sigma-1)\kappa < 1$ \\
				\hline
				\hyperlink{additive}{Additive} & \hyperlink{poly-log}{Poly-log}  & \eqref{eq:weightasspowerlawlb} \& $\alpha < 2$ \\
				\hline    
			\end{tabular}
			\caption{\footnotesize The first column represents the form of the fitness function, as in Assumption~\ref{ass:f}; the second represents the form of the degree function, as in Assumption~\ref{ass:fapp}. The third column lists the required assumptions on the vertex-weight distribution, as in Assumptions~\ref{ass:weights} and~\ref{ass:weightsapp}, together with the choices of the parameters that lead to a unique infinite path.}
		\end{table}
	\end{lemma}
	
	\noindent It is clear that Theorem~\ref{thrm:examplesadd} follows from Lemmas~\ref{lemma:checklaplaceapp} and~\ref{lemma:pathcondapp}. Before we prove these two lemmas, we state the following result, which is an analogoue of Lemma~\ref{lemma:mun}. 
	
	\begin{lemma}\label{lemma:munapp}
		Let $f$ satisfy the~\hyperlink{poly-log}{barely super-linear poly-log} case, as in Assumption~\ref{ass:fapp}. Then, for any $w\geq 0$,
		\be 
		\mu_n^w=\frac{1+o(1)}{g(w)(\sigma-1)}(\log n)^{-(\sigma-1)}.
		\ee 
	\end{lemma}
	
	\begin{proof}
		We use Lemma~\ref{lemma:mun}, which tells us that $\mu_n^w=(1+o(1))g(w)^{-1}L(n)$, where 
		\be 
		L(n):=\int_{n+2}^\infty x^{-1}(\log x)^{-\sigma}\,\dd x.
		\ee 
		Using a variable substitution $y=\log x$ and determining the integral yields the desired result.
	\end{proof}
	
	\noindent The proofs of Lemmas~\ref{lemma:checklaplaceapp} and~\ref{lemma:pathcondapp} follow the same approach as those of Lemmas~\ref{lemma:checklaplace} and~\ref{lemma:pathcond}. 
	
	\begin{proof}[Proof of Lemmas~\ref{lemma:checklaplaceapp} and~\ref{lemma:pathcondapp}, $s$ barely super-linear \hyperlink{log-stretched}{log-stretched case}, additive weights]
		To start, we prove the claim in Lemma~\ref{lemma:checklaplaceapp}. We assume that $\beta\tau>1$ and recall that 
		\be 
		f(i,w)=(i+1)\e^{(\log(i+1))^\beta}+h(w), 
		\ee 
		where $h$ is a regularly-varying function with exponent $1$. Fix $\eps>0$ sufficiently small so that $\beta\tau>1+\eps$. We apply~\eqref{eq:splitbound} with 
		\be \label{eq:knmun2}
		k_n:=n\exp(-(1-\eps)\overline c(\log n)^\tau), \quad\text{and}\quad \mu_n=\frac{1+o(1)}{\beta }(\log n)^{1-\beta}\exp(-(\log n)^\beta), 
		\ee 
		where the latter follows from Lemma~\ref{lemma:mun}. Since $\tau \in (0, 1)$ and by~\eqref{eq:weightasspowerlawplusub}, we have 
		\be \ba \label{eq:log-stretched-summable}
		\Prob{W > k_{n}} &\leq k_n^{-1}\e^{-\overline c(\log k_n)^\tau} \\ 
		&=\frac{1}{n}\exp\big((1-\eps)\overline c(\log n)^\tau-\overline c(\log n-(1-\eps)\overline c(\log n)^\tau)^\tau\big)\\ 
		&=\frac1n \exp\big(-\eps\overline c(\log n)^\tau(1+o(1))\big),
		\ea\ee 
		for $n$ sufficiently large, so that the second term on the right-hand side of~\eqref{eq:splitbound} is summable. Next, we define 
		\be \label{eq:inf1add}
		I_n:=h(k_n)\exp\big(-(\log h(k_n))^\beta +\beta (\log h(k_n))^{2\beta-1}\big).
		\ee 
		By Taylor's theorem, for $|x| < 1$, $(1+ x)^{\beta} = 1 + \beta x + \frac{\beta(\beta-1)}{2} x^{2} + o(x^3)$. Factoring out the term $\log h(k_n)$ we can thus write
		\be\ba \label{eq:taylorin}
		(\log I_n)^\beta ={}& \log h(k_n)^{\beta}\big(1-(\log h(k_n))^{\beta-1}+\beta(\log h(k_n))^{2\beta-2}\big)^\beta \\
		={}& (\log h(k_n))^\beta-\beta(\log h(k_n))^{2\beta-1}+\frac{3}{2}\beta\big(\beta-\tfrac13\big)(\log h(k_n))^{3\beta-2}\\
		&+o\big((\log h(k_n))^{3\beta-2}\big).
		\ea\ee 
		It thus follows that for $j\geq I_n$, 
		\be \ba 
		(j+1)\exp\big((\log(j+1))^\beta\big)&\geq I_n\exp\big((\log I_n)^\beta\big)\\
		&=h(k_n)\exp\Big(\frac{3\beta}{2}\big(\beta-\tfrac13\big)(\log h(k_n))^{3\beta-2}+o((\log h(k_n))^{3\beta-2})\Big).
		\ea \ee 
		We claim that the exponential term is at least $1/2$ for $n$ large. Indeed, for $\beta>2/3$ the exponential term diverges with $n$, whilst for $\beta \in (0, 2/3]$ is converges to a constant that is at least $1$. So, for all $j\geq I_n$ and $n$ sufficiently large, 
		\be 
		h(k_n)\leq 2(j+1)\exp((\log(j+1))^\beta).
		\ee 
		Moreover, as $\mu_n$ is slowly varying and $k_n$ and $h$ are both regularly varying with exponents $1$, we have $\mu_n^{-1}=o(h(k_n))$. We thus  obtain the lower bound
		\be \ba 
		\sum_{j=0}^\infty \frac{1}{(j+1)\exp(c(\log (j+1))^\beta)+h(k_n)+\mu_n^{-1}}&\geq  \sum_{j=I_n}^\infty \frac{1}{5(j+1)\exp((\log(j+1))^\beta)}&=\frac{\mu_{I_n}(1+o(1))}{5},
		\ea \ee 
		where we recall the definition of $\mu_n$ from~\eqref{eq:mu-def} (with $w^*=0$) and use (the proof of) Lemma~\ref{lemma:mun} in the final step.
		Using this in the first term on the right-hand side of~\eqref{eq:splitbound}, we obtain the upper bound 
		\be\label{eq:splitub}
		\exp\big(- \tfrac{c+o(1)}{5}\mu_{I_n}\mu_n^{-1}\big). 
		\ee 
		With $I_n$ and $\mu_n$ as in~\eqref{eq:knmun2} and using~\eqref{eq:taylorin}, we find 
		\be \ba\label{eq:muin}
		\mu_{I_n}&=\frac{1+o(1)}{\beta }(\log I_n)^{1-\beta}\exp\big(-(\log I_n)^\beta\big)\\ 
		&=\frac{1+o(1)}{\beta }(\log n)^{1-\beta}\exp\big(-(\log h(k_n))^\beta+\beta(\log h(k_n))^{2\beta-1}(1+o(1))\big).
		\ea\ee 
		We now use that $h(x)=\ell(x)x$ for some slowly-varying function $\ell$ which satisfies~\eqref{eq:elllim} for some $a\in[0,\tau\vee\beta)$. We can then write, using that $h(x)=x^{1+o(1)}$ and $k_n=n^{1+o(1)}$, and for $|x| < 1$, that  $(1+ x)^{\beta} = 1 + \beta x + o(x^2)$, 
		\be\ba\label{eq:taylorh}
		-(\log{}& h(k_n))^\beta+\beta(\log h(k_n))^{2\beta-1}(1+o(1))\\ 
		={}&-(\log k_n)^{\beta}(1+\log(\ell(k_n)) (\log{k_n})^{-1})^\beta +\beta(\log n)^{2\beta-1}(1+o(1))\\ 
		={}&-(\log k_n)^\beta -\beta (\log k_n)^{\beta-1}\log(\ell(k_n))(1+o(1))+\beta(\log n)^{2\beta-1}(1+o(1))
		\ea\ee
		Now, again using the expression for $(1+x)^\beta$, observe that \[(\log k_n)^\beta = (\log n)^\beta + (1-\eps)\overline c \beta (\log n)^{\beta+\tau-1} (1+o(1)).\] Hence we may write~\eqref{eq:taylorh} as 
		\be\label{eq:fullapprox}
		-(\log n)^\beta +\big[(1-\eps)\overline c \beta (\log n)^{\beta+\tau-1}+\beta(\log n)^{2\beta-1}-\beta(\log k_n)^{\beta-1}\log(\ell(k_n))\big](1+o(1)).
		\ee
		Recall that we assume that $\ell$ is such that $\lim_{x\to\infty} \log\log(\ell(x))/\log\log x=\overline a$ for some $\overline a\in[0,\tau\vee\beta)$ (as we assume that $\tau\vee\beta>\overline a$), that is, when $\log (\ell(x))=(\log x)^{\overline a+o(1)}$, we can write 
		\be 
		(\log k_n)^{\beta-1}\log(\ell(k_n))=(\log k_n)^{\beta+\overline a-1+o(1)}. 
		\ee 
		Since $\overline a<\tau\vee\beta$, this term is negligible compared to the other two terms in the square brackets in~\eqref{eq:fullapprox}, so that it can be included in the $o(1)$ term. We thus obtain, combining~\eqref{eq:muin} with~\eqref{eq:fullapprox},
		\be \ba
		\mu_{I_n}&=\frac{1+o(1)}{\beta }(\log n)^{1-\beta}\exp\big(-(\log n)^\beta +\big[(1-\eps)\overline c \beta (\log n)^{\beta+\tau-1}+\beta(\log n)^{2\beta-1}\big](1+o(1))\big)\\ 
		&=\mu_n \exp\big((C'+o(1))(\log n)^{\beta+(\tau\vee\beta)-1}\big).
		\ea\ee 
		We can thus, finally, bound the first term on the right-hand side of~\eqref{eq:splitbound} from above by substituting this in~\eqref{eq:splitub}, which yields 
		\be 
		\exp\Big(-\frac{c+o(1)}{5}\exp\big((C'+o(1))(\log n)^{\beta+(\tau\vee\beta)-1}\big)\Big),
		\ee 
		which is summable since $\tau\vee \beta>1-\beta$ (and using an argument similar to that in~\eqref{eq:log-stretched-summable}), as desired.
		
		We then prove the claim in Lemma~\ref{lemma:pathcondapp}. We set $k_n:=n\exp(\underline c(\log n)^{\tau'})$. It follows from~\eqref{eq:weightasspowerlawpluslb} that 
		\[
		\P{W\geq k_n} \geq n^{-1} \exp(- \underline{c}(\log n)^{{\tau'}}) \exp(\underline c(\log n + \underline{c} (\log{n})^{{\tau'}})^{{\tau'}}) \geq n^{-1}, 
		\] 
		which is not summable. Then, we define 
		\be \label{eq:inlogadd}
		I_n:=h(k_n)\exp(-(\log h(k_n))^\beta-(\log h(k_n))^{2\beta-1}), 
		\ee 
		and write
		\be 
		\sum_{i=0}^\infty \frac{1}{f(i,k_n)}=\sum_{i=0}^\infty \frac{1}{(i+1)\exp((\log (i+1))^\beta)+h(k_n)}\leq \sum_{i=0}^{I_n}\frac{1}{h(k_n)}+\!\!\!\sum_{i=I_n+1}^\infty \frac{1}{(i+1)\exp((\log (i+1))^\beta)}.
		\ee 
		We recall the definition of $\mu_n$ from~\eqref{eq:mu-def} (with $w^*=0$) and use (the proof of) Lemma~\ref{lemma:mun} to deduce that the above equals 
		\be \label{eq:goalll}
		\exp(-(\log h(k_n))^\beta-(\log h(k_n))^{2\beta-1})+(1+o(1))\mu_{I_n}. 
		\ee
		We now need only show that this expression is $o(\mu_n^w)$, which clearly implies~\eqref{eq:limsupbound}. 
		
		Let us start with the first term. By a similar sequence of computations as in~\eqref{eq:taylorh} through~\eqref{eq:fullapprox}, and using the fact that, for $|x| < 1$ we have $(1 +x)^{\beta} = 1 + \beta x + o(x^2)$, we obtain
		\be \ba\label{eq:taylorh2}
		\exp{}&(-(\log h(k_n))^\beta-(\log h(k_n))^{2\beta-1})\\
		={}&\exp(-(\log k_n)^\beta -\beta (\log k_n)^{\beta-1}\log(\ell(k_n))(1+o(1))-(\log n)^{2\beta-1}(1+o(1)))\\ 
		={}&\exp\big(-(\log n)^\beta -\big[\underline c\beta (\log n)^{\beta+\tau'-1} +(\log n)^{2\beta-1}+\beta (\log n)^{\beta-1}\log(\ell(k_n))\big](1+o(1))\big).
		\ea \ee  
		We recall that $h(x) = \ell(x) x$, where $\ell$ is slowly varying, such that 
		\be 
		\lim_{x\to\infty}\log\log (1/\ell(x))/\log\log x=\underline a,
		\ee 
		for some $\underline a\in[0,\tau'\vee \beta)$ (as by our assumption). That is, $\log(\ell(x))=-(\log x)^{\underline a+o(1)}$. We thus have 
		\be \label{eq:ellass}
		(\log k_n)^{\beta-1}\log(\ell(k_n))=-(\log k_n)^{\beta+\underline a-1+o(1)}.
		\ee 
		Since $\underline a<\tau'\vee \beta$, it follows that this term can be included in the $o(1)$ term in~\eqref{eq:taylorh2}. We thus arrive at
		\begin{linenomath*}
			\begin{align*}
				\exp(-(\log h(k_n))^\beta) &=\exp\big(-(\log n)^\beta-\big[\underline c \beta(\log n)^{\beta+\tau'-1}+(\log n)^{2\beta-1}\big](1+o(1))\big) \\
				& = \mu_n^ \beta(\log n)^{\beta -1}\exp\big(-\big[\underline c \beta(\log n)^{\beta+\tau'-1}+(\log n)^{2\beta-1}\big](1+o(1))\big).
			\end{align*}
		\end{linenomath*}
		Since we assume that $\tau'>\beta\vee(1-\beta)$, we directly obtain that this is $o(\mu_n^w)$.
		
		Recalling~\eqref{eq:goalll}, it thus remains to prove that $\mu_{I_n}=o(\mu_n^w)$. By~\eqref{eq:taylorh2}, \eqref{eq:inlogadd}, and~\eqref{eq:knmun2}, the desired result follows by showing that $\exp(-(\log I_n)^\beta)=o(\exp(-(\log n)^\beta))$ or, equivalently, $(\log I_n)^\beta - (\log n)^\beta$ diverges with $n$. First, we note that $h(x)=x^{1+o(1)}$ (as $h$ is regularly varying with exponent $1$) and $k_n=n^{1+o(1)}$. Again, using the approximation to $(1+x)^{\beta}$ in a similar manner as in~\eqref{eq:taylorh},
		\be \ba 
		(\log I_n)^\beta&=\big(\log h(k_n)-(\log h(k_n))^\beta-(\log h(k_n))^{2\beta-1}\big)^\beta \\
		&=(\log h(k_n))^\beta -\beta (\log h(k_n))^{2\beta-1}(1+o(1))\\
		&=(\log k_n)^\beta +\big[\beta(\log k_n)^{\beta-1}\log(\ell(k_n))-\beta (\log n)^{2\beta-1}\big](1+o(1))\\
		&=(\log n)^\beta+ \big[\underline c \beta (\log n)^{\beta+\tau'-1}+\beta(\log k_n)^{\beta-1}\log(\ell(k_n))-\beta (\log n)^{2\beta-1}\big](1+o(1)).
		\ea \ee 
		By the same argument that leads to~\eqref{eq:ellass}, we argue that we can include the second term in the square brackets within the $o(1)$ term as it is of lower order compared to either $(\log n)^{\beta+\tau'-1}$ or $(\log n)^{2\beta-1}$. The desired result thus follows since $\tau'>\beta\vee(1-\beta)$. 
	\end{proof}
	
	\begin{proof}[Proof of Lemmas~\ref{lemma:checklaplaceapp} and~\ref{lemma:pathcondapp}, $s$ barely super-linear~\hyperlink{poly-log}{poly-log case}, mixed weights]
		We first prove the claim in Lemma~\ref{lemma:checklaplaceapp}. We assume that $(\sigma-1)\kappa>1+ \kappa$ and recall that 
		\be 
		f(i,w)=g(w)(i+2)(\log(i+2))^\sigma+h(w), 
		\ee 
		where $g$ and $h$ are regularly-varying functions with exponents $1$ and $\gamma\geq 0$, respectively. Fix $\eps>0$ sufficiently small so that $(\sigma-1)\kappa>1+\eps$. We apply~\eqref{eq:splitbound} with 
		\be \label{eq:knmun3}
		k_n:=(\log n)^{(1+\eps)/\kappa},\quad\text{and}\quad\mu_n=\frac{1+o(1)}{g(0)(\sigma-1) }(\log n)^{-(\sigma-1)},
		\ee 
		where the latter follows from Lemma~\ref{lemma:mun}. Now, using~\eqref{eq:weightassstrechtedub}, we obtain 
		\be 
		\P{W\geq k_n}\leq \e^{-\overline c k_n^\kappa}=\e^{-\overline c(\log n)^{1+\eps}} \leq n^{-(1+\eps)}, 
		\ee 
		for $n$ sufficiently large, so that the second term on the right-hand side of~\eqref{eq:splitbound} is summable. Then, we define 
		\be \label{eq:inf2mix}
		I_n:=(h(k_n)+\mu_n^{-1})/g(k_n).
		\ee 
		We write $g(x)=\ell_1(x)x$ and $h(x)=\ell_2(x)x^\gamma$ for some slowly-varying functions $\ell_1,\ell_2$, and note that $g(x)=x^{1+o(1)}$ and $h(x)=x^{\gamma+o(1)}$. With $k_n$ and $\mu_n$ as in~\eqref{eq:knmun3}, we thus have 
		\be 
		I_n=(\log n)^{\max\{(\sigma-1)\kappa,(1+\eps)\gamma\}/\kappa-(1+\eps)/\kappa+o(1)}.
		\ee   
		We note that $I_n$ diverges with $n$ since $(\sigma-1)\kappa>1+\eps$. It then directly follows for all $n$ sufficiently large and all $j\geq I_n$, that
		\be 
		(j+2)(\log(j+2))^\sigma\geq j\geq I_n=(h(k_n)+\mu_n^{-1})/g(k_n).
		\ee 
		As a result, for all $n$ large we obtain the lower bound
		\be
		\sum_{j=0}^\infty \frac{1}{g(k_n)(j+2)(\log( j+2))^\sigma+h(k_n)+\mu_n^{-1}}\geq  \frac{1}{g(k_n)}\sum_{j=I_n}^\infty \frac{1}{3(j+2)(\log(j+2))^\sigma}=\frac{\mu_{I_n}(1+o(1))}{3g(k_n)/g(0)},
		\ee 
		where we recall the definition of $\mu_n$ from~\eqref{eq:mu-def} (with $w^*=0$) and use (the proof of) Lemma~\ref{lemma:mun} in the final step. Substituting this bound into the first term on the right-hand side of~\eqref{eq:splitbound}, and using~\eqref{eq:knmun3}, we obtain the upper bound 
		\be 
		\exp\Big(-\frac{cg(0)+o(1)}{3}
		\frac{\mu_{I_n}\mu_n^{-1}}{g(k_n)}\Big)=\exp\big(-(\log\log n)^{-(\sigma-1)}(\log n)^{(\sigma-1)-(1+\eps)/\kappa+o(1)}\big),
		\ee
		(where we incorporate the constants into the $o(1)$ in the exponent of the $\log{n}$ term). This is summable when $(\sigma-1)-(1+\eps)/\kappa>1$  (again using an argument similar to that in~\eqref{eq:log-stretched-summable}). Since $\eps$ is arbitrary, the desired result follows since $(\sigma-1)\kappa>1+\kappa$.
		
		We then prove the claim in Lemma~\ref{lemma:pathcondapp}. We fix $\eps>0$ sufficiently small such that $(\sigma-1)\kappa<1-\eps$ and set $k_n:=(\log n)^{(1-\eps)/\kappa}$. It follows from~\eqref{eq:weightassstrechtedlb} that 
		\[
		\P{W\geq k_n} \geq \e^{-\underline{c} (\log{n})^{1- \eps}} \geq n^{-1},
		\] 
		which is not summable in $n$. By a similar computation as in~\eqref{eq:similar-comp}, 
		\be 
		\sum_{i=0}^\infty \frac{1}{f(i,k_n)}=\frac{1}{g(k_n)}\sum_{i=0}^\infty \frac{1}{(i+2)(\log (i+2))^\sigma+h(k_n)/g(k_n)}\leq \frac{C_\sigma}{g(k_n)}, 
		\ee 
		for some constant $C_\sigma>0$. Since $g$ varies regularly with exponent $1$, we can write $g(x)=x^{1+o(1)}$, so that 
		\be 
		\sum_{i=0}^\infty \frac{1}{f(i,k_n)}\leq (\log n)^{-(1-\eps)/\kappa+o(1)}.
		\ee 
		As a result, since $(\sigma-1)\kappa<1-\eps$, this sum is $o(\mu_n^w)$. It follows that~\eqref{eq:limsupbound} is satisfied, which concludes the proof.
	\end{proof}

	\begin{proof}[Proof of Lemmas~\ref{lemma:checklaplaceapp} and~\ref{lemma:pathcondapp}, $s$ barely super-linear~\hyperlink{poly-log}{poly-log case}, additive weights]
		We first prove the claim in Lemma~\ref{lemma:checklaplaceapp}. We assume that $(\sigma-1)(1-1/\nu)>1$, and recall that 
		\be 
		f(i,w)=(i+2)(\log(i+2))^\sigma+h(w), 
		\ee 
		where $h$ is a regularly-varying function with exponent $1$. Fix $\eps>0$ sufficiently small so that $\sigma(1-1/\nu)>1+\eps$. We apply~\eqref{eq:splitbound} with 
		\be \label{eq:knmun4}
		k_n:=\exp((\log n)^{(1+\eps)/\nu}),\quad\text{and}\quad \mu_n=\frac{1+o(1)}{\sigma-1 }(\log n)^{-(\sigma-1)},
		\ee 
		where the latter follows from Lemma~\ref{lemma:munapp}. Now, using~\eqref{eq:weightasslogstrechtedub} we obtain
		\be 
		\P{W\geq k_n}\leq \e^{-\overline c (\log k_n)^\nu}=\e^{-\overline c (\log n)^{1+\eps}} < n^{-(1+\eps)},
		\ee 
		for $n$ sufficiently large, so that the second term on the right-hand side of~\eqref{eq:splitbound} is summable. Then, we define 
		\be \label{eq:inf2add}
		I_n:=\frac{h(k_n)}{(\log h(k_n))^\sigma}.
		\ee 
		It then follows that for all $j\geq I_n$ and $n$ sufficiently large, 
		\be 
		(j+2)(\log(j+2))^\sigma\geq I_n(\log I_n)^\sigma =h(k_n)(1+o(1))>h(k_n)/2.
		\ee
		Moreover, since $\mu_n^{-1}=o(h(k_n))$, irrespective of the values of $\sigma, \nu>1$ and since $h(x)=x^{1+o(1)}$, we obtain for all $n$ sufficiently large the lower bound 
		\be \label{eq:sumlb}
		\sum_{j=0}^\infty \frac{1}{(j+1)(\log (j+1))^\sigma+h(k_n)+\mu_n^{-1}}\geq  \sum_{j=I_n}^\infty \frac{1}{5(j+1)(\log(j+1))^\sigma}=\frac{\mu_{I_n}(1+o(1))}{5}\geq \frac{\mu_{I_n}}{6},
		\ee 
		where we recall the definition of $\mu_n$ from~\eqref{eq:mu-def} (with $w^*=0$) and use~\eqref{eq:knmun4} in the final step. We again use that $h(x)=x^{1+o(1)}$, which yields 
		\be 
		\mu_{I_n}=\frac{1+o(1)}{(\sigma-1) }(\log I_n)^{-(\sigma-1)}=\frac{1+o(1)}{(\sigma-1) }(\log k_n)^{-(\sigma-1)}=\frac{1+o(1)}{(\sigma-1) }(\log n)^{-(1+\eps)(\sigma-1)/\nu}.
		\ee 
		Substituting this into the lower bound in~\eqref{eq:sumlb}, we may bound the sum in the exponent in the first term on the right-hand side of~\eqref{eq:splitbound} to obtain the upper bound
		\be \label{eq:inmunub}
		\exp\big(-\tfrac{c}{6} \mu_{I_n}\mu_n^{-1}\big)\leq \exp\big(-(\log n)^{(\sigma-1)(1-(1+\eps)/\nu)+o(1)}\big).
		\ee 
		By choosing $\eps$ sufficiently small, this upper bound is summable since $(\sigma-1)(1-1/\nu)>1$, which yields the desired result.
		
		We then prove the claim in Lemma~\ref{lemma:pathcondapp}. We fix $\alpha\in(1,2)$, take $\eps>0$ sufficiently small such that $1-\eps>\alpha-1$, and set $k_n:=n^{(1-\eps)/(\alpha-1)}$. We then have $\Prob{W \geq k_{n}} \geq \underline{\ell}(k_{n}) n^{-(1- \eps)}$ by~\eqref{eq:weightasspowerlawlb}, which is not summable in $n$.  We define $I_n:=k_n/(\log k_n)^\sigma$ and bound
		\be 
		\sum_{i=0}^\infty \frac{1}{f(i,k_n)}\leq \sum_{i=0}^{I_n-1}\frac{1}{k_n}+\sum_{i=I_n}^\infty \frac{1}{(j+2)(\log (j+2))^\sigma} = \frac{1}{(\log k_n)^\sigma}+(1+o(1))\mu_{I_n},
		\ee 
		where we recall the definition of $\mu_n$ from~\eqref{eq:mu-def} (with $w^*=0$) and use (the proof of) Lemma~\ref{lemma:mun}. By the choice of $k_n$ and $I_{n}$ and with $\mu_n$ as in~\eqref{eq:knmun4}, we have
		\be 
		\frac{1}{\mu_n^w}\Big(\frac{1}{(\log k_n)^\sigma}+(1+o(1))\mu_{I_n}\Big)=\cO\Big(\frac{1}{\log n}\Big)+(1+o(1))\Big(\frac{\log I_n}{\log n}\Big)^{-(\sigma-1)}\to \Big(\frac{\alpha-1}{1-\eps}\Big)^{\sigma-1}.
		\ee 
		As a result, by choosing $\eps$ sufficiently small, since $\alpha -1 < 1$, the right-hand side of the above is strictly smaller than one; it follows that~\eqref{eq:limsupbound} is satisfied. 
	\end{proof}

	\section{Verifying conditions for other inter-birth time distributions}\label{sec:app}
	
	In this section we consider the other birth-time distributions listed in Remark~\ref{rem:otherdistr}, and check the conditions of Assumptions~\ref{ass:star} and~\ref{ass:path} (it is clear that all distributions are continuous and thus satisfy Assumption~\ref{ass:uniqueness}). We do not repeat all the detailed calculations in the previous subsection. Rather, we show where calculations differ, and where similar or analogous arguments yield the desired results. 
	
	\subsection{Conditions for an infinite star, Assumption~\ref{ass:star}.} We need to verify~\eqref{eq:limsup-mgf} and~\eqref{eq:laplacesum}. We start with the former, and split between the different inter-birth distributions listed in Remark~\ref{rem:otherdistr}. 
	
	\paragraph{Gamma.} 
	
	Let the inter-birth times be gamma random variables, i.e.\ for $i\in \N, k>0$, and $w\geq 0$, let $X_w(i)\sim \text{Gamma}(k,kf(i-1,w))$. Recall $Y_n$ from~\eqref{eq:ynexplicit} (with $w^*=0$). Then, 
	\be 
	\cM_\lambda(Y_n)=\prod_{\ell=n}^\infty \Big(\frac{kf(\ell,0)}{kf(\ell,0)-\lambda}\Big)^k=\bigg(\prod_{\ell=n}^\infty \frac{kf(\ell,0)}{kf(\ell,0)-\lambda}\bigg)^k =\cM_\lambda(\wt Y_n)^k,
	\ee 
	where 
	\be 
	\wt Y_n\overset d= \sum_{j=n+1}^\infty \mathrm{Exp}(kf(j-1,0)). 
	\ee 
	We can, by possibly changing the constant $c$ so that $c<\min\{1,1/k\}$, use the same computations as in~\eqref{eq:mux} through~\eqref{eq:mgf-exp-bound} to derive the upper bound $\cM_{c\mu_n^{-1}}(Y_n)\leq (1-c)^{-k}$ to conclude that Condition~\eqref{eq:limsup-mgf} of Assumption~\ref{ass:star} holds. 
	
	For the next two examples, we use that 
	\be 
	\sum_{\ell=n}^\infty \frac{1}{f(\ell,0)^2}\leq \bigg(\sum_{\ell=n}^\infty \frac{1}{f(\ell,0)}\bigg)^2=\mu_n^{-2}.
	\ee 
	
	\paragraph{Beta.} 
	
	Let the inter-birth times be distributed as follows. For any $i\in\N$ and $w\geq 0$, let 
	\be \label{eq:Xbeta}
	X_w(i)\overset d= \frac{\alpha+\beta}{\alpha}\frac{1}{ f(i-1,w)}B_i, 
	\ee 
	where $(B_i)_{i\in\N}$ is a sequence of i.i.d.\ copies of a $\mathrm{Beta}(\alpha,\beta)$ random variables, for some $\alpha\geq 1,\beta\in(0,1]$. Recall $Y_n$ from~\eqref{eq:ynexplicit}. Then, for $\lambda>0$, 
	\be 
	\E{\e^{\lambda X_0(j)}}=\sum_{k=0}^\infty \prod_{\ell=0}^{k-1}\frac{\alpha+\ell}{\alpha+\beta+\ell}\Big(\frac{\alpha+\beta}{\alpha}\frac{\lambda}{f(j-1,0)}\Big)^k \frac{1}{k!}, 
	\ee 
	where we set the empty product $\prod_{\ell=0}^{-1}$ equal to one. As the terms in the product are at most one, we directly obtain the upper bound 
	\be 
	\cM_{\mu_n^{-1}}(Y_n)=\prod_{j=n+1}^\infty \E{\e^{\mu_n^{-1}X_0(j)}}\leq \prod_{j=n+1}^\infty \exp\Big(\frac{\alpha+\beta}{\alpha}\frac{\mu_n^{-1}}{f(j-1,0)}\Big)=\exp\bigg(\frac{\alpha+\beta}{\alpha}\mu_n^{-1}\sum_{j=n}^\infty \frac{1}{f(j,0)}\bigg). 
	\ee 
	By the definition of $\mu_n^{-1}$, it follows that this upper bound equals $\e^{(\alpha+\beta)/\alpha}$, so that~\eqref{eq:limsup-mgf} is satisfied.

	\paragraph{Rayleigh.}
	
	Let the inter-birth times be Rayleigh random variables, i.e.\ for any $i\in\N$ and $w\geq 0$, let $X_w(i)\sim \text{Rayleigh}(\sqrt{2/\pi}/f(i-1,w))$. Recall $Y_n$ from~\eqref{eq:ynexplicit} (with $w^*=0$). Then, 
	\be 
	\cM_\lambda(Y_n)=\prod_{\ell=n}^\infty \bigg(1+ \frac{\lambda}{f(\ell,0)} \e^{\lambda^2/(\pi f(\ell,0)^2)}\Big(1+\mathrm{erf}\Big(\frac{\lambda}{\sqrt\pi f(\ell,0)}\Big)\Big)\bigg),
	\ee 
	where $\mathrm{erf}(x):=(2/\sqrt\pi)\int_0^x \e^{-t^2}\,\dd t$ denotes the error function. Since the error function is bounded from above by one, we immediately arrive at the upper bound 
	\be 
	\cM_\lambda(Y_n)\leq\exp\bigg(2\lambda \sum_{\ell=n}^\infty \frac{1}{f(\ell,0)} \e^{\lambda^2/(\pi f(\ell,0)^2)}\bigg)\leq \exp\bigg(2\lambda \sum_{\ell=n}^\infty \frac{1}{f(\ell,0)}\exp\bigg(\frac{\lambda^2}{\pi}\sum_{\ell=n}^\infty \frac{1}{f(\ell,0)^2}\bigg)\bigg). 
	\ee 
	Now, with $\lambda=c\mu_n^{-1}$, we finally arrive at the upper bound $\exp(2\exp(1/\pi))$, as desired.
	
	We now verify Condition~\ref{item:starlaplace} of Assumption~\ref{ass:star}, i.e.\ Equation~\eqref{eq:laplacesum}, which is summarised in the following lemma. 
	
	\begin{lemma}\label{lemma:checklaplacegen}
		Assume the inter-birth time distributions satisfy any of the choices in Remark~\ref{rem:otherdistr}. With the same conditions for the fitness type, degree function $s$, and the vertex-weight distribution, as in Lemmas~\ref{lemma:checklaplace} and~\ref{lemma:checklaplaceapp}, Condition~\ref{item:starlaplace} of Assumption~\ref{ass:star} is satisfied.
	\end{lemma}
	
	\noindent We split the proof of Lemma~\ref{lemma:checklaplacegen} into the different choices for the inter-birth distributions. We observe that, irrespective of the inter-birth distribution, we assume the mean of $X_W(j)$, conditionally on $W$, is always $1/f(j-1,W)$ for any $j\in\N$, so that the definition (and asymptotic behaviour) of $\mu_n$ remains unchanged and only depends on the fitness function $f$.
	
	\begin{proof}[Proof of Lemma~\ref{lemma:checklaplacegen}, Gamma case]
		For some sequence $(\ell_n)_{n\in\N}$, we have 
		\be \ba
		\E{ \cL_{c\mu_n^{-1}}(\cP_n;W)}&=\E{\bigg(\prod_{j=0}^{n-1} \frac{kf(j,W)}{kf(j,W)+c\mu_n^{-1}}\bigg)^k}\leq  \exp\bigg(-c\mu_n^{-1}\sum_{j=0}^{n-1} \frac{1}{kf(j,\ell_n)+c\mu_n^{-1}}\bigg)+\P{W\geq \ell_n}. 
		\ea \ee 
		This expression is, up to a constant in the exponential term, equivalent to the upper bound in~\eqref{eq:splitbound}. We can thus follow the proof of Lemma~\ref{lemma:checklaplace} to obtain the desired result for the gamma case in general.
	\end{proof}

	\begin{proof}[Proof of Lemma~\ref{lemma:checklaplacegen}, Beta case]
		Recall $X_w(j)$ from~\eqref{eq:Xbeta}, with $\alpha\geq 1, \beta\in(0,1]$. We have that 
		\be 
		\E{\cL_\lambda (X_w(j))}={}_1F_1\Big(\alpha;\alpha+\beta; -\frac{\alpha+\beta}{\alpha}\frac{\lambda}{f(j-1,w)}\Big),
		\ee 
		where ${}_1F_1(a;b;z)$  denotes the confluent hypergeometric function, defined as 
		\be 
		{}_1F_1(a;b;z):=\sum_{k=0}^\infty \frac{a^{(k)}z^k}{b^{(k)}k!},\qquad a,b>0, z\in\R,
		\ee 
		where $a^{(k)}:=\Gamma(a+k)/\Gamma(a)$, with $\Gamma$ the gamma function. We now use the Kummer transform ${}_1F_1(a;b;z)=\e^z {}_1F_1(b-a;b;-z)$, to obtain 
		\be \ba 
		\E{\cL_\lambda (X_w(j))}&=\exp\Big(-\frac{\alpha+\beta}{\alpha}\frac{\lambda}{f(j-1,w)}\Big){}_1F_1\Big(\beta; \alpha+\beta; \frac{\alpha+\beta}{\alpha}\frac{\lambda}{f(j-1,w)}\Big)\\ 
		&=\exp\Big(-\frac{\alpha+\beta}{\alpha}\frac{\lambda}{f(j-1,w)}\Big)\sum_{k=0}^\infty \frac{\Gamma(\alpha+\beta)\Gamma(\beta+k)}{\Gamma(\beta)\Gamma(\alpha+\beta+k)}\frac{1}{k!}\Big(\frac{\alpha+\beta}{\alpha}\frac{\lambda}{f(j-1,w)}\Big)^k.
		\ea\ee 
		We then use that, since $\beta\leq 1$ and thus $\lceil \beta\rceil =1$, 
		\be 
		\frac{\Gamma(\alpha+\beta)\Gamma(\beta+k)}{\Gamma(\beta)\Gamma(\alpha+\beta+k)}=\prod_{\ell=0}^{k-1}\frac{\beta+\ell}{\alpha+\beta+\ell}\leq \prod_{\ell=0}^{k-1}\frac{1 +\ell}{\lfloor \alpha\rfloor+1 +\ell}=\prod_{\ell=1}^{\lfloor \alpha \rfloor} \frac{\ell}{\ell+k}= \frac{k!(\lfloor \alpha\rfloor)!}{(k+\lfloor \alpha\rfloor)!}.
		\ee 
		This yields the upper bound 
		\be \label{eq:app1}
		\E{\cL_\lambda (X_w(j))}\leq \exp\Big(-\frac{\alpha+\beta}{\alpha}\frac{\lambda}{f(j-1,w)}\Big)\sum_{k=0}^\infty \frac{(\lfloor \alpha\rfloor)!}{(k+\lfloor\alpha\rfloor)!}\Big(\frac{\alpha+\beta}{\alpha}\frac{\lambda}{f(j-1,w)}\Big)^k.
		\ee 
		We then observe that, for any $z>0$ we have 
		\[
		\frac1z (\e^z-1)\leq \e^{z/2+z^2}. 
		\]
		Indeed, this is easy to check for $z \geq 1$, whereas for $z < 1$ we use the inequalities $\log{x} \leq x-1$ and $\e^{z} - z- 1 \leq \frac{z^2}{2} + z^{3}$ (the latter using $\sum_{j=3}^{\infty} \frac{1}{j!} \leq 1$). Thus, for $z > 0$ and $\alpha\geq 1$, 
		\be \label{eq:app2}
		\sum_{k=0}^\infty \frac{1}{(k+\lfloor \alpha\rfloor)!}z^k \leq \frac{1}{(\lfloor \alpha\rfloor)!}\sum_{k=0}^\infty \frac{1}{(k+1)!}z^k=\frac{1}{(\lfloor \alpha\rfloor)!}\frac1z (\e^z-1)\leq \frac{1}{(\lfloor \alpha\rfloor)!}\e^{z/2+z^2}.
		\ee 
		Combining~\eqref{eq:app1} and~\eqref{eq:app2}, we thus arrive at 
		\be \label{eq:laplacebeta}
		\E{\cL_\lambda (X_w(j))}\leq \exp\Big(-\frac12\frac{\alpha+\beta}{\alpha}\frac{\lambda}{f(j-1,w)}+\Big(\frac{\alpha+\beta}{\alpha}\frac{\lambda}{f(j-1,w)}\Big)^2\Big).
		\ee 
		We then take some sequence $(J_n)_{n\in\N}$ with $J_n\leq n$ to obtain 
		\be \ba \label{eq:laplaceunif}
		\mathbb E[\cL_{c\mu_n^{-1}}(\cP_n;W)]&\leq \E{\cL_{c\mu_n^{-1}}(\cP_n-\cP_{J_n}; W)}\\
		&\leq\exp\bigg(-c\mu_n^{-1}\frac{\alpha+\beta}{2\alpha}\sum_{j=J_n}^{n-1} \frac{1}{f(j,k_n)}+\Big(\frac{\alpha+\beta}{\alpha}\Big)^2c^2\mu_n^{-2}\sum_{j=J_n}^{n-1} \frac{1}{f(j,k_n)^2}\bigg)+\P{W\geq k_n}\\
		&=\exp\bigg(- c\mu_n^{-1}\frac{\alpha+\beta}{2\alpha}\sum_{j=J_n}^{n-1} \frac{1}{f(j,k_n)}\Big(1-\frac{2(\alpha+\beta)}{\alpha}\frac{c}{\mu_nf(j,k_n)}\Big)\bigg)+\P{W\geq k_n}.
		\ea\ee 
		We now choose $J_n$ for the different cases of the degree function $s$ (according to Assumption~\ref{ass:deg}) such that for all $j\geq J_n$, we have $\mu_nf(j,k_n)\geq 4c(\alpha+\beta)/\alpha$. This yields the upper bound
		\be 
		\exp\bigg(- c\frac{\alpha+\beta}{4\alpha}\mu_n^{-1}\sum_{j=J_n}^{n-1} \frac{1}{f(j,k_n)}\bigg)+\P{W\geq k_n}\leq \exp\bigg(- c\frac{\alpha+\beta}{4\alpha} \mu_n^{-1}\sum_{j=J_n}^{n-1} \frac{1}{f(j,k_n)+\mu_n^{-1}}\bigg)+\P{W\geq k_n}.
		\ee 
		With this upper bound at hand, we also claim that we can use the proof of Lemma~\ref{lemma:checklaplace} to obtain the desired upper bounds, despite the fact that the sum in the exponential term starts from $J_n$ rather than from $0$ (which is the case in Lemma~\ref{lemma:checklaplace}).
		
		For all additive cases we have $J_n=0$, so that we can directly use the proof Lemma~\ref{lemma:checklaplace}.  This is due to the fact that $\mu_nf(j,k_n)\geq \mu_n f(0,k_n)$ diverges with $n$, so that it is at least $4(\alpha+\beta)/\alpha$ for all $n$ sufficiently large. Indeed, for all additive weights types, for some small $\eps>0$ and some constant $C>0$,
		
		\begin{table}[H]
			\centering
			\small
			\begin{tabular}{lll}
				\textbf{\hyperlink{superlinear}{Super-linear}}:  &  $\mu_n=(C+o(1))ns(n)^{-1}$,  & $k_n:=n^{(1-\eps)p}$.\\
				\textbf{\hyperlink{log-stretched}{Log-stretched}}:  &$\mu_n=(C+o(1))(\log n)^{1-\beta}\exp(-(\log n)^\beta)$, & $k_n:=n\exp(-(1-\eps)\overline c(\log n)^\tau)$.\\
				\textbf{\hyperlink{poly-log}{Poly-log}}:  & $\mu_n=(C+o(1))(\log n)^{-(\beta-1)}$, & $k_n:=\exp((\log n)^{(1+\eps)/\tau})$. 
			\end{tabular}
		\end{table}
		
		\noindent One can directly verify that $\mu_nf(0,k_n)=\Theta(\mu_nh(k_n))$ diverges with $n$, where $h$ is some regularly-varying function with exponent $1$.
		
		For the three mixed weights cases (depending on the degree function $s$), we can set $J_n$ as follows, where $K$ is a sufficiently large constant:
		\begin{table}[H]
			\centering
			\small
			\begin{tabular}{ll}
				\textbf{\hyperlink{superlinear}{Super-linear}}:  &  $J_n:=Kn^{\eps(p-1)/p}$. \\
				\textbf{\hyperlink{log-stretched}{Log-stretched}}:  & $J_n:=KI_n$, with $I_n$ as in~\eqref{eq:inf1mix}.\\
				\textbf{\hyperlink{poly-log}{Poly-log}}:  & $J_n:=KI_n$, with $I_n$ as in~\eqref{eq:inf2mix}.
			\end{tabular}
		\end{table}
		
		\noindent Using that for all $j\geq J_n$, 
		\be 
		\mu_n f(j,k_n)\geq \mu_n f(J_n,k_n)\geq \mu_n g(k_n) s(J_n), 
		\ee  
		it is readily verified in all three cases that $\mu_n f(j,k_n)\geq 4(\alpha+\beta)/\alpha$ for all $j\geq J_n$ and all sufficiently large $n$. It also holds that the bounds used in the proof of Lemma~\ref{lemma:checklaplace} still hold when using this choice of $J_n$ in the mixed weights cases, so that the conclusions from Lemma~\ref{lemma:checklaplace} are valid here, too.
	\end{proof}

	\begin{proof}[Proof of Lemma~\ref{lemma:checklaplacegen}, Rayleigh case]
		We have 
		\be\ba\label{eq:rayleighbound}
		\E{\cL_{c\mu_n^{-1}}(\cP_n;W)}\leq \prod_{j=0}^{n-1} \bigg(1-\frac{c\mu_n^{-1}}{f(j,k_n)}\e^{\mu_n^{-2}/(\pi f(j,k_n)^2)}\Big(1-\mathrm{erf}\Big(\frac{c\mu_n^{-1}}{\sqrt\pi f(j,k_n)}\Big)\Big)\bigg)+\P{W\geq k_n}.
		\ea\ee 
		We now use that, for any $x>0$, 
		\be \label{eq:erfineq}
		\e^{x^2}(1-\mathrm{erf}(x))=\frac{2}{\sqrt \pi}\e^{x^2}\int_x^\infty \e^{-t^2}\, \dd t=\frac{2}{\sqrt\pi}\int_0^\infty \e^{-s(s+2x)}\,\dd s\geq \frac{2}{\sqrt\pi}\int_0^\infty \e^{-(s+2x)^2}\,\dd s=1-\mathrm{erf}(2x), 
		\ee 
		where we use a variable substitution $s=t-x$. Also using that $1-x\leq \e^{-x}$ for all $x\in \R$, we can bound the first term on the right-hand side of~\eqref{eq:rayleighbound} from above from 
		\be 
		\exp\bigg(-c\mu_n^{-1}\sum_{j=0}^{n-1} \frac{1}{f(j,k_n)}\Big(1-\mathrm{erf}\Big(\frac{2c\mu_n^{-1}}{\sqrt \pi f(j,k_n)}\Big)\Big)\bigg).
		\ee 
		Since the error function is increases and $\lim_{x\to\infty}\mathrm{erf}(x)=1$, we can now repeat the same argument as in the \textbf{Beta} case: we create an upper bound by starting the sum from some index $J_n$, such that we can bound 
		\be 
		1-\mathrm{erf}\Big(\frac{2c\mu_n^{-1}}{\sqrt \pi f(j,k_n)}\Big)\geq \delta,\qquad \text{for all }j\geq J_n,
		\ee 
		for some small constant $\delta>0$. The choice of $J_n$ can be the same as in the \textbf{Beta} case, and the result thus follows through here, as well.        
	\end{proof}

	\subsection{Conditions for an infinite path, Assumption~\ref{ass:path}.} We need to verify Conditions~\eqref{eq:div-condition} and~\eqref{eq:smallest-expl-prob}. Here, we choose $\nu_n^w=d\mu_n^w$ for some arbitrary constant $d\in(0,1)$, as in the proof of Theorem~\ref{thm:star-path-rif}. We recall from the proof of Lemma~\ref{lemma:pathcond} that the lower bound in~\eqref{eq:tildepnlb} uses Markov's inequality only. Hence, since we assume that the mean of the inter-birth time $X_w(j)$ equals $1/f(j-1,w)$, irrespective of the its distribution, it follows that the proofs of Lemmas~\ref{lemma:pathcond} and~\ref{lemma:pathcondapp} immediately follow for the other choices of inter-birth distributions in Remark~\ref{rem:otherdistr}. 
	
	It thus remains to verify~\eqref{eq:smallest-expl-prob}. We observe that the proof of~\eqref{eq:smallest-expl-prob} for exponentially distributed inter-birth times, as in the proof of Theorem~\ref{thm:star-path-rif}, holds more generally, as long as for any $i\in\N$ and $w\geq 0$, and for some $K>0$,
	\be 
	\Var(X_w(i))\leq K\E{X_w(i)}^2=\frac{K}{f(i-1,w)^2}.
	\ee 
	This is readily checked for all the cases in Remark~\ref{rem:otherdistr}, and is related to Remark~\ref{rem:more-general}.

	\section{Proof of Lemma~\ref{lemma:regexp}}\label{sec:regapp}

	\begin{proof}[Proof of Lemma~\ref{lemma:regexp}] 
		The aim of the proof is to provide an upper and lower bound for  $\P{r(W)\geq x}$ that are asymptotically equivalent (i.e.\ the same up to a $(1+o(1))$ multiplicative factor). First, we define, for a function $r:[0,\infty)\to (0,\infty)$, the \emph{generalised inverse} as
		\be\label{eq:geninv}
		r^{\sss\leftarrow}(x):=\inf\{y\geq 0: r(y)\geq x\}, \qquad x\geq 0.
		\ee   
		We start by  proving an upper bound. Suppose $r(W)\geq x$. Then $W\in \{y\geq 0: r(y)\geq x\}$, so that by the definition of the generalised inverse, it follows that $W\geq r^{\sss\leftarrow}(x)$. As a result
		\be \label{eq:rwub}
		\P{r(W)\geq x}\leq \P{W\geq r^{\sss\leftarrow}(x)},
		\ee
		which concludes the upper bound. For a lower bound, suppose $r(W)<x$. Then, $\{y\geq 0: r(y)\geq x\}\subseteq \{y\geq 0: r(y)\geq r(W)\}$. Hence, the infimum of the left-hand side is larger than the infimum of the right-hand side, so that 
		$r^{\sss\leftarrow}(x)\geq r^{\sss\leftarrow}(r(W))$. As a result, 
		\be 
		\P{r(W)\geq x}\geq \P{r^{\sss\leftarrow}(r(W))> r^{\sss\leftarrow}(x)}
		\ee

		We use~\cite[Proposition $1.5.12$]{BinGolTeu89} to obtain that $r^{\sss\leftarrow}(r(x))=(1+o(1))x$ as $x\to\infty$. Hence, since $r^{\sss\leftarrow}(r(W))>r^{\sss\leftarrow}(x)$ implies $W\geq r^{\sss\leftarrow}(x)$, we can write $r^{\sss\leftarrow}(r(W))=(1+t(W))W$, for some function $t$ such that $t(x)\to 0$ as $x\to \infty$. We thus obtain 
		\be
		\P{r^{\sss\leftarrow}(r(W))>r^{\sss\leftarrow}(x)}=\P{(1+t(W))W>r^{\sss\leftarrow}(x)}.
		\ee 
		Now, for any $\eps>0$ we can take $x$ sufficiently large so that, since $W\geq r^{\sss\leftarrow}(x)$, $|t(W)|<\eps$. We thus obtain the lower bound, 
		\be 
		\P{W>r^{\sss\leftarrow}(x)/(1-\eps)}.
		\ee 
		As the tail distribution of $W$ is regularly varying with exponent $-(\zeta-1)$, we obtain 
		\be 
		\P{W>r^{\sss\leftarrow}(x)/(1-\eps)}=((1-\eps)^{\zeta-1}+o(1))\P{W>r^{\sss\leftarrow}(x)}.
		\ee 
		As $\eps$ is arbitrary, combining this with the upper bound in~\eqref{eq:rwub}, we obtain 
		\be 
		\P{r(W)\geq x}=(1+o(1))\P{W\geq r^{\sss\leftarrow}(x)}.
		\ee 
		We now use that, by~\cite[Theorem $1.5.12$]{BinGolTeu89}, the function $r^{\sss\leftarrow}$ is regularly varying with exponent $1/\rho$. This implies, by Lemma~\ref{lemma:regvar} and the assumption on the tail distribution of $W$, that the right-hand side is regularly varying with exponent $-(\zeta-1)/\rho$ as desired.
	\end{proof}

\end{document}